\documentclass[english]{article}
\usepackage[T1]{fontenc}
\usepackage[latin1]{inputenc}
\usepackage{geometry}
\usepackage{float}
\usepackage{mathrsfs}
\usepackage{amsmath}
\usepackage{amssymb}
\usepackage{graphicx}
\usepackage{subfigure}

\makeatletter

\floatstyle{ruled}
\newfloat{algorithm}{tbp}{loa}
\providecommand{\algorithmname}{Algorithm}
\floatname{algorithm}{\protect\algorithmname}

\usepackage{amsthm}

\usepackage{mathrsfs}

\usepackage{amsfonts}

\usepackage{epsfig}

\usepackage{bm}

\usepackage{mathrsfs}

\usepackage{enumerate}

\@ifundefined{definecolor}{\@ifundefined{definecolor}
 {\@ifundefined{definecolor}
 {\usepackage{color}}{}
}{}
}{}

\usepackage{subfig}\usepackage[all]{xy}

\newcommand{\bbR}{\mathbb R}

\newtheorem{theorem}{Theorem}[section]
\newtheorem{lem}{Lemma}[section]
\newtheorem{rem}{Remark}[section]
\newtheorem{prop}{Proposition}[section]
\newtheorem{cor}{Corollary}[section]
\newcounter{hypA}
\newenvironment{hypA}{\refstepcounter{hypA}\begin{itemize}
  \item[({\bf A\arabic{hypA}})]}{\end{itemize}}
\newcounter{hypB}
\newenvironment{hypB}{\refstepcounter{hypB}\begin{itemize}
 \item[({\bf H\arabic{hypB}})]}{\end{itemize}}

\usepackage{babel}\date{}

\usepackage{babel}

\makeatother

\usepackage{babel}
\begin{document}

%+Title
\begin{center}

{\Large \textbf{Central Limit Theorems for Coupled Particle Filters}}

\vspace{0.5cm}

BY AJAY JASRA \& FANGYUAN YU

{Department of Statistics \& Applied Probability,
National University of Singapore, Singapore, 117546, SG.}
{\footnotesize E-Mail:\,}\texttt{\emph{\footnotesize staja@nus.edu.sg}}, \texttt{\emph{\footnotesize e0204835@u.nus.edu}}

\end{center}

\begin{abstract}
In this article we prove new central limit theorems (CLT) for several coupled particle filters (CPFs).
CPFs are used for the sequential estimation of the difference of expectations w.r.t.~filters which are in 
some sense close. Examples include the estimation of the filtering distribution associated to
different parameters (finite difference estimation) and filters associated to partially observed discretized diffusion processes (PODDP) and the implementation of the multilevel Monte Carlo (MLMC) identity. We develop new theory for CPFs and based upon several results, we propose a new CPF which approximates the maximal coupling (MCPF) of a pair of predictor distributions.
In the context of ML estimation associated to PODDP with discretization $\Delta_l$
we show that the MCPF and
the approach in \cite{mlpf_new} have, under assumptions, an asymptotic variance that is upper-bounded by an expression that is (almost) $\mathcal{O}(\Delta_l)$, uniformly in time. The $\mathcal{O}(\Delta_l)$ rate preserves the so-called forward rate of
the diffusion in some scenarios which is not the case for the CPF in \cite{mlpf}. 
%We also show that this optimal behaviour is not necessarily
%possessed by the original CPF in ***.
\\
\noindent \textbf{Key words}: Coupled Particle Filter, Central Limit Theorem, Multilevel Monte Carlo.
\end{abstract}

\section{Introduction}

The filtering problem is ubiquitous in statistics, applied probability and applied mathematics, with
far reaching applications in weather prediction, finance and engineering; see \cite{cappe,crisan} for example.
In most cases of practical interest, the filter must be numerically approximated and
a popular method for doing so is the particle filter (see e.g.~\cite{cappe,delm:04} and the references therein). The PF generates
$N\geq 1$ samples in parallel and uses a combination of sampling, importance sampling and resampling
to approximate the filter.  There is a substantial literature on the convergence of PFs (e.g.~\cite{delm:04})
and in particular there are CLTs which allow one to understand the errors associated to estimation.
Under assumptions, the associated asymptotic variance is bounded uniformly in time; see e.g.~\cite{chopin2}.

In this article, we are concerned with the filtering problem where one seeks to estimate the difference
of expectations of two different but `close' filters. As an example, if one observes data in discrete and
regular time points, associated to an unobserved diffusion process. In many cases, one must
time-discretize the diffusion process, if the transition density is not available up-to a non-negative
unbiased estimator. In such scenarios it is well known that the cost of estimating the filter using a PF
can be significantly reduced using a collapsing sum representation of the expectation associated
to the filter with the most precise discretization and estimating each difference independently using
a coupled particle filter.
In other applications, one can approximate differences of expectations of filters with different parameter values as a type of finite difference approximations; see for instance \cite{jacob1,sen}.

The CPF developed in \cite{chopin3} (see also \cite{jacob1,jacob2,mlpf,lee,sen}) is used in several applications as discussed above and various
other contexts. It consists of a particle filter which runs on the product space of the two filters. The
sampling operation often consists of simulating from a coupling of the Markov transitions of the hidden dynamics, which are often
available in applications. Resampling proceeds by sampling the maximal coupling
associated to the probability distributions of particle indices. The use of correlating the PFs is vital, for
instance in ML applications (see \cite{giles,hein,mlpf}), as it is this property which allows a variance reduction relative to a single PF. As has been noted by several authors, e.g.~\cite{sen}, unless the coupling of the Markov transitions is particularly strong, one expects that the maximally coupled resampling operation to ultimately
decorrelate the pairs of particles exponentially fast in time. As a result, the benefits of running such
algorithms may have a minimal effect for long time intervals.

In this article we consider four CPFs. The first is the case where the resampling is independent for each pair (independent resampliing CPF (IRCPF)), the second with the maximally coupled resampling 
(the maximally coupled resampling PF (MCRPF)).
The third algorithm, which to our knowledge is new, is based upon a weak law of large numbers 
for the MCRPF on the product space. This result shows that the limiting coupling on product space
does not correspond to the maximal coupling of the filter (or predictor). This coupling does not seem
to have any optimality properties, so we suggest a new CPF, the MCPF which approximates the maximal coupling of the predictors. This algorithm requires that the Markov transition of the filter
to have a density which is known pointwise and essentially samples from the maximal coupling of
particle filters. The fourth algorithm in \cite{mlpf_new} is based on multinomial resampling which uses the same
uniform random variable for each particle pair; we call this the Wasserstein CPF (WCPF). In general, all four algorithms can be used in each of the examples, with the constraint for the MCPF mentioned above.
We remark that there are CPFs in \cite{gregory,sen}, but they are not considered here
as they require even more mathematical complexity for their analysis.

We prove a CLT for the first three algorithms (the CLT for the WCPF, when the state-dynamics are one dimensional is in \cite{mlpf_new}) associated to the difference of estimates of expectations of the same function w.r.t.~the predictors, which is extended to multiple dimensions. The asymptotic variance expression is directly related to the properties of the limiting
coupling one approximates. Under assumptions (of the type in \cite{whiteley}), for the PODDP with (Euler) discretization $\Delta_l$, we show that the MCPF (resp.~WCPF) has an asymptotic variance that is upper-bounded by an expression that is $\mathcal{O}(\Delta_l)$ (resp.~$\mathcal{O}(\Delta_l^{1-\lambda})$, $\lambda$ arbitrarily close to, but not equal to, zero), uniformly in time, which preserves the so-called forward rate of the diffusion in some scenarios. This is reassuring as it shows that filtering for difference estimation
in the PODDP case can be effectively performed. This time and rate stability is associated to the fact 
that the limiting coupling on product spaces are associated to the optimal $L_0$ and $L_2$ Wasserstein couplings of the predictor/filter. For the IRCPF one does not recover the coupling rate
of the diffusion process and this poor performance is well-known in the literature. In the case
of the MCRPF we show even in a favourable case, that it can have an asymptotic variance, at time $n$, that is $\mathcal{O}(e^n\Delta_l)$
and identify when one can expect the algorithm to work well.  As was seen in the empirical results of \cite{mlpf_new,mlpf} the time and rate behaviour of the MCRPF in general does not seem to be as good as for the MCPF and WCPF.
The assumptions used for our asymptotic variance results are also verified in a real example.
Our CLTS are, to the best of our knowledge, the first results of these types for CPFs and require non-standard proofs. To summarize, the main results of the article are:
\begin{itemize}
\item{Theorem \ref{theo:wlln_max} gives a WLLN for the MCRPF.}
\item{Theorem \ref{theo:clt_ind} gives a CLT for the IRCPF, MRCPF \& MCPF.}
\item{Theorem \ref{theo:av_thm} gives a general bound on the asymptotic variance for each of the methods, IRCPF, MRCPF, MCPF \& WCPF.}
\item{Propositions \ref{prop:diff_max_coup} and \ref{prop:diff_was} give the time uniform bounds on the asymptotic variance for the MCPF and WCPF noted above (i.e.~for PODDPs).}
\end{itemize}

This paper is structured as follows. In Section \ref{sec:mod} we give our notations, models and the motivating example of a PODDP with MLMC is given.
In Section \ref{sec:algo} the algorithms are presented. In Section
\ref{sec:theory} our theoretical results are stated. Our CLTs are given and a general bound on the asymptotic
variance is provided.
In Section \ref{sec:appl} our results
are applied to a pratical model in the context of using coupled particle filters for PODDP with MLMC. 
The article is summarized in Section \ref{sec:summary}.
The appendix contains the proofs of our theoretical results.

\section{Notation and Models}\label{sec:mod}

\subsection{Notations}

Let $(\mathsf{X},\mathcal{X})$ be a measurable space. 
For a given function $v:\mathsf{X}\rightarrow[1,\infty)$ we denote by $\mathcal{L}_v(\mathsf{X})$
the class of functions $\varphi:\mathsf{X}\rightarrow\bbR$ for which
$$
\|\varphi\|_v := \sup_{x\in \mathsf{X}}\frac{|\varphi(x)|}{v(x)} < +\infty\ .
$$
When $v\equiv 1$ we write $\|\varphi\| := \sup_{x\in \mathsf{X}}|\varphi(x)|$.
If $v=1$ we write $\mathcal{B}_b(\mathsf{X})$, $\mathcal{C}_b(\mathsf{X})$ for the bounded measurable and continuous, bounded measurable functions respectively.
$\mathcal{C}^2(\mathsf{X})$ are the collection of twice continuously differentiable real-valued functions on $\mathsf{X}$. Let $\mathsf{d}$ be a metric on $\mathsf{X}$ then
for $\varphi\in\mathcal{L}_v(\mathsf{X})$, we say that $\varphi\in\textrm{Lip}_{v,\mathsf{d}}(\mathsf{X})$ if there exist a $C<+\infty$ such that for every $(x,y)\in\mathsf{X}\times\mathsf{X}$:
$$
|\varphi(x)-\varphi(y)| \leq C\mathsf{d}(x,y)v(x)v(y).
$$
If $v=1$, we write $\varphi\in\textrm{Lip}_{\mathsf{d}}(\mathsf{X})$ and write $\|\varphi\|_{\textrm{Lip}}$ for the Lipschitz constant.
$\mathscr{P}(\mathsf{X})$  denotes the collection of probability measures on $(\mathsf{X},\mathcal{X})$.
We also denote, for $\mu\in\mathscr{P}(\mathsf{X})$, $\|\mu\|_{v}:=\sup_{|\varphi|\leq v}|\mu(\varphi)|$. 
%For $\varphi:\mathsf{X}\rightarrow\mathbb{R}$ we write $\mathcal{B}_b(\mathsf{X})$, $\mathcal{C}_b(\mathsf{X})$
%and $\textrm{Lip}(\mathsf{X})$ as the collection of bounded measurable, continuous, bounded measurable and Lipschitz functions respectively.
%When $\mathsf{X}=\mathbb{R}^d$, $\varphi:\mathsf{X}\rightarrow\mathbb{R}$, if $\varphi\in\textrm{Lip}(\mathsf{X})$, i.e.~we have for every $(x,y)\in\mathsf{X}\times\mathsf{X}$, $|\varphi(x)-\varphi(y)|\leq C|x-y|$ for some $C<+\infty$ the norm on the R.H.S.~
%is understood to be the $L_1-$norm.
%
%For $\varphi\in\mathcal{B}_b(\mathsf{X})$, we write the supremum norm $\|\varphi\|=\sup_{x\in\mathsf{X}}|\varphi(x)|$.
%For $\varphi\in\mathcal{B}_b(\mathsf{X})$, $\textrm{Osc}(\varphi)=\sup_{(x,y)\in\mathsf{X}\times\mathsf{X}}|\varphi(x)-\varphi(y)|$ and %if $\textrm{Osc}(\varphi)=1$, we write $\textrm{Osc}_1(\varphi)$. 
%we write $\textrm{Osc}_1(\mathsf{X})$ for the set of functions $\varphi$ on $\mathsf{X}$ such that $\textrm{Osc}(\varphi)=1$.
%For $\varphi\in\textrm{Lip}(\mathsf{X})$, we write the Lipschitz constant $\|\varphi\|_{\textrm{Lip}}$.
For a measure $\mu$ on $(\mathsf{X},\mathcal{X})$
and a $\mu-$integrable, $\varphi:\mathsf{X}\rightarrow\mathbb{R}$, the notation $\mu(\varphi)=\int_{\mathsf{X}}\varphi(x)\mu(dx)$ is used.
Let $K:\mathsf{X}\times\mathcal{X}\rightarrow(0,\infty)$ be a non-negative kernel and $\mu$ be a measure then we use the notations
$
\mu K(dy) = \int_{\mathsf{X}}\mu(dx) K(x,dy)
$
and for $K(x,\cdot)-$integrable, $\varphi:\mathsf{X}\rightarrow\mathbb{R}$
$
K(\varphi)(x) = \int_{\mathsf{X}} \varphi(y) K(x,dy).
$
%For a sequence of Markov kernels $K_1,\dots,K_n$ we write 
%$$
%K_{1:n}(x_0,dx_n) = \int_{\mathsf{X}^{n-1}}\prod_{p=1}^n K_p(x_{p-1},dx_p) .
%$$
For $\mu,\nu\in\mathscr{P}(\mathsf{X})$, the total variation distance 
is denoted $\|\mu-\nu\|_{\textrm{tv}}=\sup_{B\in\mathcal{X}}|\mu(B)-\nu(B)|$.
For $B\in\mathcal{X}$ the indicator is written $\mathbb{I}_B(x)$ and the dirac measure $\delta_B(dx)$.
For two measures $\mu,\nu$ on $(\mathsf{X},\mathcal{X})$, the product measure is $\mu\otimes \nu$.
For two measurable functions $\varphi$, $\psi$ on  $(\mathsf{X},\mathcal{X})$, the tensor product of functions is $\varphi\otimes\psi$.
%For $\mu\in\mathscr{P}(\mathsf{X})$ such that $\mu(V)<+\infty$ we denote $\mu\in\mathscr{P}_V$.
$\mathcal{U}_A$ denotes the uniform distribution on the set~$A$. $\mathcal{N}_t(a,b)$ is the $t-$dimensional Gaussian distribution of mean $a$ and covariance $b$ (if $t=1$ the subscript is dropped from $\mathcal{N}$).
$\mathbb{P}$ and $\mathbb{E}$ are used to denote probability and expectation w.r.t.~the law of the specified algorithm - the context will be clear in each instant.
$\Rightarrow$ and $\rightarrow_{\mathbb{P}}$ are used to denote convergence in distribution and probability respectively. In the context of the article, this is as $N\rightarrow\infty$.

\subsection{Models}

Let $(\mathsf{X},\mathcal{X})$ be a measurable space and $\{G_n\}_{n\geq 0}$ a sequence of non-negative, bounded and measurable functions such that $G_n:\mathsf{X}\rightarrow\mathbb{R}_+$.
Let $\eta_0^{f},\eta_0^c\in\mathscr{P}(\mathsf{X})$ and $\{M_n^f\}_{n\geq 1}$,   $\{M_n^c\}_{n\geq 1}$ be two sequences of Markov kernels, i.e.~$M_n^f:\mathsf{X}\rightarrow\mathscr{P}(E)$,
$M_n^c:\mathsf{X}\rightarrow\mathscr{P}(E)$. Define, for $s\in\{f,c\}$, $B\in\mathcal{X}$
$$
\gamma_{n}^s(B) = \int_{\mathsf{X}^{n+1}}\mathbb{I}_B(x_n) \Big(\prod_{p=0}^{n-1} G_p^s(x_p)\Big) \eta_0^s(dx_0)\prod_{p=1}^n M_p^s(x_{p-1},dx_p)
$$
and
$$
\eta_n^s(B) = \frac{\gamma_{n}^s(B)}{\gamma_{n}^s(1)}.
$$

The objective is to consider Monte Carlo type algorithms which will approximate quantities, for $\varphi\in\mathcal{B}_b(\mathsf{X})$ any $n\geq 0$, such as 
\begin{equation}\label{eq:pred}
\eta_n^f(\varphi) - \eta_n^c(\varphi)
\end{equation}
or 
\begin{equation}\label{eq:filt}
\frac{\eta_n^f(G_n\varphi)}{\eta_n^f(G_n)} - \frac{\eta_n^c(G_n\varphi)}{\eta_n^c(G_n)}.
\end{equation}
\eqref{eq:pred} corresponds to a predictor of a state-space model and \eqref{eq:filt} the filter.
We focus explicitly on the predictor from herein.

\begin{rem}
We can make the $G_n$ also depend on $\{f,c\}$, which may be of importance in applictations. In the subsequent development, this is not done, but could be at a cost
of slightly longer mathematical arguments and notational complications. 
\end{rem}

The major point is that one
would like to approximate couplings of $(\eta_n^f,\eta_n^c)$, say $\check{\eta}_n\in\mathscr{P}(\mathsf{X}\times\mathsf{X})$ i.e.~that for any $B\in\mathcal{X}$ and every $n\geq 0$
$$
\check{\eta}_n(B\times\mathsf{X}) = \eta_n^f(B) \qquad \check{\eta}_n(\mathsf{X}\times B) = \eta_n^c(B)
$$
and consider
approximating 
$$
\check{\eta}_n(\varphi\otimes 1) - \check{\eta}_n(1\otimes \varphi)
$$
%or 
%$$
%\frac{\check{\eta}_n((G_n\varphi)\otimes 1)}{\check{\eta}_n(G_n\otimes 1)} - \frac{\check{\eta}_n(1\otimes(G_n\varphi))}{\check{\eta}_n(1\otimes G_n)}.
%$$
An explanation of why coupling the pairs is of interest has been given in the introduction and will be further illuminated in Section \ref{sec:diff_ex}.

Throughout the article, it is assumed that there exists $\check{\eta}_0\in\mathscr{P}(\mathsf{X}\times\mathsf{X})$ such that
for any $B\in\mathcal{X}$
$$
\check{\eta}_0(B\times \mathsf{X}) = \eta_0^f(B) \qquad \check{\eta}_0(\mathsf{X}\times B) = \eta_0^c(B)
$$
and moreover for any $n\geq 1$ there exists Markov kernels $\{\check{M}_n\}$, $\check{M}_n:\mathsf{X}\times\mathsf{X}\rightarrow\mathscr{P}(\mathsf{X}\times\mathsf{X})$
such that for any $B\in\mathcal{X}$, $(x,x')\in\mathsf{X}\times\mathsf{X}$:
$$
\check{M}_n(B\times \mathsf{X})(x,x') = M_n^f(B)(x) \qquad \check{M}_n(\mathsf{X}\times B)(x,x') = M_n^c(B)(x').
$$

\subsection{Example}\label{sec:diff_ex}

The following example is from \cite{mlpf} and there is some overlap with the presentation in that article.
We start with a diffusion process:
\begin{eqnarray}
dZ_t & = & a(Z_t)dt + b(Z_t)dW_t
\label{eq:sde}
\end{eqnarray}
with $Z_t\in\mathbb{R}^d=\mathsf{X}$, 
$a:\mathbb{R}^d\rightarrow\mathbb{R}^d$ ($j$th element denoted $a^j$), $b:\mathbb{R}^d\rightarrow\mathbb{R}^{d\times d}$ ($(j,k)$th element denoted $b^{j,k}$), 
$t\geq 0$ and $\{W_t\}_{t\geq 0}$ a $d-$dimensional Brownian motion. 
The following assumptions are made and if referred to as (D). (D) is assumed explicitly. We set $Z_0=x^*\in\mathsf{X}$.

\begin{quote} 
The coefficients $a^j, b^{j,k} \in \mathcal{C}^2(\mathsf{X})$, for $j,k= 1,\ldots, d$. 
Also, $a$ and $b$ satisfy % and the following hold
\begin{itemize}
%\item[(i)] Smoothness: $a^j, b^{j,k} \in C^2$, for $j,k= 1,\ldots, d$;
\item[(i)] {\bf uniform ellipticity}: $b(z)b(z)^T$ is uniformly positive definite;
\item[(ii)] {\bf globally Lipschitz}:
% with linear growth}
there is a $C>0$ such that 
$|a^j(z)-a^j(y)|+|b^{j,k}(z)-b^{j,k}(y)| \leq C |z-y|$ % and $|a(x)|+|b(x))| \leq C (1+|x|)$ 
for all $(z,y)\in \mathsf{X}\times\mathsf{X}$, $(j,k)\in\{1,\dots,d\}^2$. %(Linear growth comes form the former by letting y=0)
%\item[(iii)] {\bf boundedness}: $\mathbb{E}|X_0|^p < \infty$ for all $p \geq 1.$
\end{itemize}
\end{quote}

The data are observed at regular unit time-intervals (i.e.~in discrete time) 
$y_1,y_2,\dots$, $y_k \in \mathsf{Y}$.
It is assumed that conditional on $Z_{k}$, 
$Y_k$ is independent of all other random variables with density $G(z_{k},y_k)$. Let $M(z,dy)$  be the transition of the diffusion process (over unit time)
and consider a discrete-time Markov chain $X_0,X_1,\dots$ with initial distribution $M(x^*,\cdot)$ and transition $M(x,dy)$. Here we are
creating a discrete-time Markov chain that corresponds to the discrete-time skeleton of the diffusion process at a time lag of 1.
Now we write $G_{k}(x_{k})$  instead of $G(x_{k},y_{k+1})$.
Then we define, for $B\in\mathcal{X}$
$$
\gamma_n(B) := \int_{\mathsf{X}^{n+1}}\mathbb{I}_B(x_n)\Big(\prod_{p=0}^{n-1} G_{p}(x_{p})\Big)M(x^*,dx_{0})\prod_{p=1}^n M(x_{p-1},dx_{p}).
$$
The predictor is $\eta_n(B)=\gamma_n(B)/\gamma_n(1)$ which corresponds
to the distribution associated to $Z_{n+1}|y_1,\dots,y_n$.

In many applications, one must time discretize the diffusion to use the model in practice. We suppose an Euler discretization
with discretization $\Delta_l=2^{-l}$, $l\geq 0$ and write the associated transition kernel over unit time as $M^l(x,dy)$.
Note that, in practice, one may not be able to compute the density of the kernel as it is a compositon of $\Delta_l^{-1}-1$ Gaussians, however, one can certainly sample from $M^l$ in most cases.
Hence we
are interested in the Feynman-Kac model
for $B\in\mathcal{X}$
$$
\gamma_n^l(B) := \int_{\mathsf{X}^{n+1}}\mathbb{I}_B(x_n)\Big(\prod_{p=0}^{n-1} G_{p}(x_{p})\Big)M^l(x^*,dx_{0})\prod_{p=1}^n M^l(x_{p-1},dx_{p}).
$$
with associated predictor $\eta_n^l(B)=\gamma_n^l(B)/\gamma_n^l(1)$.  

Below, we will explain why one may wish to compute, for $l\geq 1$, $\eta_n^l(\varphi)-\eta_n^{l-1}(\varphi)$, $\varphi\in\mathcal{B}_b(\mathsf{X})$. That is, $f$ as used above, relates to the predictor
associated to the discretization $\Delta_l$ and  $c$, the predictor with discretization $\Delta_{l-1}$. 
A natural coupling of $M_n^f=M^l$ and $M_n^c=M^{l-1}$ exists (e.g.~\cite{giles}) so one also 
has a given $\check{\eta}_0$ and $\check{M}_n$.

Before continuing, we note some results which will help in our discussion below.
 As established in \cite[eq.~(32)]{mlpf} one has  for $C<+\infty$
\begin{equation}\label{eq:markov_cont}
\sup_{\mathcal{A}}\sup_{x\in\mathsf{X}}|M_n^f(\varphi)(x)-M_n^c(\varphi)(x)| \leq C \Delta_l
\end{equation}
where $\mathcal{A}=\{\varphi\in\mathcal{B}_b(\mathsf{X})\cap\textrm{Lip}_{\mathsf{d}_1}(\mathsf{X}): \|\varphi\|\leq 1\}$, $\mathsf{d}_1$ is the $L_1-$norm.
In addition,  \cite[Proposition D.1.]{mlpf} states
that for $p>0$, $C<+\infty$ and any $(x,y)\in\mathsf{X}\times\mathsf{X}$
\begin{equation}\label{eq:coup_h_cont}
\int_{\mathsf{X}\times\mathsf{X}}\|u-v\|^p \check{M}_n((x,y),d(u,v)) \leq C(|x-y| + \Delta_l^{1/2})^p.
\end{equation}
When $p=2$ ($\|u-v\|$ is the $L_2-$norm), the term $\Delta_l$ is the so-called \emph{forward strong error rate}.
In the proof of \cite[Theorem 4.3]{mlpf}, it is shown that for any $n\geq 0$ there exists a $C<+\infty$ such that for $\varphi\in\mathcal{A}$, $l\geq 0$
one has
\begin{equation}\label{eq:pred_bias_euler}
|\eta_n(\varphi)-\eta_n^l(\varphi)| \leq C\Delta_l.
\end{equation}
Note that this bound can be deduced from \eqref{eq:markov_cont}, but that $C$ may depend on $n$; this latter point is ignored for now.

\subsubsection{Multilevel Monte Carlo}\label{sec:mlmc_intro}

Suppose one can exactly sample from $\eta_n^l$ for any 
$l,n\geq 0$. The Monte Carlo estimate of $\eta_n^l(\varphi)$, set $\varphi\in\mathcal{A}$, is then of course $\frac{1}{N}\sum_{i=1}^N\varphi(x_n^i)$,
$X_n^i$ are i.i.d.~from $\eta_n^l$. One has the mean square error (MSE)
$$
\mathbb{E}\Big[\Big(\frac{1}{N}\sum_{i=1}^N\varphi(x_n^i) - \eta_n(\varphi)\Big)^2\Big] = \frac{\mathbb{V}\textrm{ar}_{\eta_n^l}[\varphi(X)]}{N} + |\eta_n(\varphi)-\eta_n^l(\varphi)|^2
$$
where $\mathbb{V}\textrm{ar}_{\eta_n^l}[\varphi(X)]$ is the variance of $\varphi(X)$ w.r.t.~$\eta_n^l$. 
Then, for $\varepsilon>0$ given, to target a mean square error of $\mathcal{O}(\varepsilon^2)$, by \eqref{eq:pred_bias_euler},
one chooses $l=\mathcal{O}(|\log(\varepsilon)|)$ (as one sets $\Delta_l^2=\varepsilon^2$). Then one must choose $N=\mathcal{O}(\varepsilon^{-2})$ and we suppose the cost
of simulating one sample is $\mathcal{O}(\Delta_l^{-1})$ (see \cite{mlpf} for a justification), again ignoring $n$. Then the cost of
achieving a MSE of $\mathcal{O}(\varepsilon^2)$ is hence $\mathcal{O}(\varepsilon^{-3})$.

Now, one has that for any $L\geq 1$, the multilevel identity \cite{giles,hein}
$$
\eta_n^L(\varphi) = \eta_n^0(\varphi) + \sum_{l=1}^{L}\{[\eta_n^l-\eta_n^{l-1}](\varphi)\}.
$$
Suppose that, for $l\geq 1$, it is possible to exactly sample a coupling of $(\eta_n^l,\eta_n^{l-1})$, denote it $\check{\eta}_n^l$, so that one has 
$$
\int_{\mathsf{X}\times\mathsf{X}} \|x-y\|^2 \check{\eta}_n^l(d(x,y)) \leq C\Delta_l
$$
and the cost of such a simulation is $\mathcal{O}(\Delta_l^{-1})$. The rate $\Delta_l$ has been taken from the strong error rate in \eqref{eq:coup_h_cont}.
Now to estimate $[\eta_n^l-\eta_n^{l-1}](\varphi)$ one simulates $((X_n^{1,l},X_n^{1,l-1}),\dots,$ $(X_n^{N_l,l},X_n^{N_l,l-1}))$ i.i.d.~from $\check{\eta}_n^l$ for some $N_l\geq 1$ and the approximation
is $\frac{1}{N_l}\sum_{l=1}^{N_l}\{\varphi(x_n^{i,l})-\varphi(x_n^{i,l-1})\}$.
For $1\leq l\leq L$ this is repeated independently for estimating $[\eta_n^l-\eta_n^{l-1}](\varphi)$ and the case $l=0$ is performed, independently, using the Monte Carlo method above. One can
then easily show that the MSE of the estimate of $\eta_n^L$ is upper-bounded by (recall $\varphi\in\mathcal{A}$) by
$$
\frac{\mathbb{V}\textrm{ar}_{\eta_n^l}[\varphi(X)]}{N_0} + C \sum_{l=1}^L\frac{\Delta_l}{N_l} + |\eta_n(\varphi)-\eta_n^l(\varphi)|^2.
$$
Then, for $\varepsilon>0$ given, to target a mean square error of $\mathcal{O}(\varepsilon^2)$, set $L=\mathcal{O}(|\log(\varepsilon)|)$, and $N_l=\mathcal{O}(\varepsilon^{-2}|\log(\varepsilon)|\Delta_l)$ (see \cite{giles}),
then the MSE is $\mathcal{O}(\varepsilon^2)$. The overall computational effort is $\mathcal{O}(\sum_{l=0}^L N_l \Delta_l^{-1})=\mathcal{O}(\varepsilon^{-2}|\log(\varepsilon)|^2)$; if $\varepsilon$ is suitably small this is a significant reduction in computational effort relative to the MC method above.

The main point here is that, at present, there are no computational methods to perform the exact simulation mentioned and thus we focus on particle filters, developed in the literature.
The estimates (variance/cost) above typically depend on $n$ and our objective is to consider if the variance of PF estimates can be $\mathcal{O}(\Delta_l)$ uniformly in time and hence that the ML gain is retained uniformly in time.
We focus on the asymptotic variance in the CLT (versus the finite sample variance) as this can be more straightforward to deal with.

\section{Algorithms}\label{sec:algo}

\subsection{Independent Pair Resampling}

The first procedure we consider is as follows. Let $n\geq 1$, $B\in\mathcal{X}\vee\mathcal{X}$ and $\mu\in\mathscr{P}(\mathsf{X}\times\mathsf{X})$ and define the probability measure:
$$
\check{\Phi}_n^I(\mu)(B)  = \frac{\mu([G_{n-1}\otimes G_{n-1}]\check{M}_n(B)) }{\mu(G_{n-1}\otimes G_{n-1})}.
$$
Set $u_n=(x_n^f,x_n^c)\in\mathsf{X}\times\mathsf{X}$
then consider the joint probability measure on $(\mathsf{X}\times\mathsf{X})^{n+1}$
$$
\mathbb{P}(d(u_0,\dots,u_n)) = \check{\eta}_0(du_0) \prod_{p=1}^n \check{\Phi}_p^I(\eta_{p-1}^f\otimes\eta_{p-1}^c)(du_p).
$$
It is easily checked that the marginal of $x_n^f$ (resp.~$x_n^c$) is $\eta_n^f$ (resp.~$\eta_n^c$). Denote by $\check{\eta}_n^I$ the marginal of $u_n$ induced by this joint probability measure.
Thus, if one could sample a trajectory of $(u_0,\dots,u_n)$ from $\mathbb{P}(d(u_0,\dots,u_n))$ one could easily approximate quantities such as \eqref{eq:pred} or \eqref{eq:filt} using Monte Carlo
methods. In most practical applications of interest, this is not possible. 

The particle approximation is taken as:
$$
\mathbb{P}(d(u_0^{1:N},\dots,u_n^{1:N})) = \Big(\prod_{i=1}^N \check{\eta}_0(du_0^i)\Big)\Big(\prod_{p=1}^n\prod_{i=1}^N\check{\Phi}_p^I(\eta_{p-1}^{N,f}\otimes\eta_{p-1}^{N,c})(du_p^i)\Big)
$$
where for $p\geq 1$, $s\in\{f,c\}$
$$
\eta_{p-1}^{N,s}(dx) = \frac{1}{N}\sum_{i=1}^N \delta_{x_{p-1}^{i,s}}(dx).
$$
The key point is that in this algorithm, the resampled indices for a pair of particles are generated (conditionally) independently.
Set
$$
\check{\eta}_{p}^{N,I}(du) = \frac{1}{N}\sum_{i=1}^N \delta_{u_{p}^{i}}(du).
$$
As has been mentioned by many authors (e.g.~\cite{sen}) one does not expect this procedure to effective, in the sense that $\check{\eta}_n^I$ would
not provide an appropriate dependence between $(\eta_n^f,\eta_n^c)$.

\subsection{Maximally Coupled Resampling}

Let $n\geq 1$, $B\in\mathcal{X}\vee\mathcal{X}$ and $\mu\in\mathscr{P}(\mathsf{X}\times\mathsf{X})$ and define the probability measure:
\begin{eqnarray*}
\check{\Phi}_n^M(\mu)(B) & = & \mu\Big(\{F_{n-1,\mu,f} \wedge F_{n-1,\mu,c}\} \check{M}_n(B)\Big)
+ \Big(1-
\mu\Big(\{F_{n-1,\mu,f} \wedge F_{n-1,\mu,c}\}\Big)
\Big) \times \\ & & 
(\mu\otimes\mu)\Big(\Big\{\overline{F}_{n-1,\mu,f}\otimes \overline{F}_{n-1,\mu,c}\Big\}\bar{M}_n(B)\Big)
%\Big\{\frac{F_{n-1,\mu,f}-\{F_{n-1,\mu,f}\wedge F_{n-1,\mu,c}\}}{
%\mu(F_{n-1,\mu,f}-\{F_{n-1,\mu,f}\wedge F_{n-1,\mu,c}\})}\Big\}\otimes \\ & &
%\Big\{\frac{F_{n-1,\mu,c}-\{F_{n-1,\mu,f}\wedge F_{n-1,\mu,c}\}}{
%\mu(F_{n-1,\mu,c}-\{F_{n-1,\mu,f}\wedge F_{n-1,\mu,c}\})}\Big\}
%\bar{M}_n(\varphi)\Big)
\end{eqnarray*}
where for $(x,y)\in\mathsf{X}\times\mathsf{X}$
\begin{eqnarray*}
\overline{F}_{n-1,\mu,f}(x,y) &  = & \frac{F_{n-1,\mu,f}(x,y)-\{F_{n-1,\mu,f}(x,y)\wedge F_{n-1,\mu,c}(x,y)\}}{
\mu(F_{n-1,\mu,f}-\{F_{n-1,\mu,f}\wedge F_{n-1,\mu,c}\})} \\
\overline{F}_{n-1,\mu,c}(x,y) &  = & \frac{F_{n-1,\mu,c}(x,y)-\{F_{n-1,\mu,f}(x,y)\wedge F_{n-1,\mu,c}(x,y)\}}{
\mu(F_{n-1,\mu,c}-\{F_{n-1,\mu,f}\wedge F_{n-1,\mu,c}\})} \\
F_{n-1,\mu,f}(x,y) &  = & \check{G}_{n-1,\mu,f}(x)\otimes1 
\end{eqnarray*}
\begin{eqnarray*}
F_{n-1,\mu,c}(x,y) &  = & 1\otimes \check{G}_{n-1,\mu,c}(y) \\
\check{G}_{n-1,\mu,f}(x) & = & \frac{G_{n-1}(x)}{\mu(G_{n-1}\otimes 1)} \\
\check{G}_{n-1,\mu,c}(y) & = & \frac{G_{n-1}(y)}{\mu(1\otimes G_{n-1})} 
\end{eqnarray*}
and for $((x,y),(u,v))\in\mathsf{X}^2\times \mathsf{X}^2$ and $B\in\mathcal{X}\vee\mathcal{X}$
$$
\bar{M}_n(B)((x,y),(u,v)) = \check{M}_n(B)(x,v).
$$
Now set $\check{\eta}_0^M=\check{\eta}_0$ and we define, recursively, for any $n\geq 1$, $B\in\mathcal{X}\vee\mathcal{X}$
$$
\check{\eta}_n^M(B) = \check{\Phi}_n^M(\check{\eta}_{n-1}^M)(B).
$$
Note that by \cite[Proposition A.1.]{mlpf} we have for $B\in\mathcal{X}$
$$
\check{\eta}_n^M(B\times\mathsf{X}) = \eta_n^f(B)\quad\textrm{and}\quad\check{\eta}_n^M(\mathsf{X}\times B) = \eta_n^c(B).
$$
The algorithm used here is:
$$
\mathbb{P}(d(u_0^{1:N},\dots,u_n^{1:N})) = \Big(\prod_{i=1}^N \check{\eta}_0(du_0^i)\Big)\Big(\prod_{p=1}^n\prod_{i=1}^N\check{\Phi}_p^M(\check{\eta}_{n-1}^{N,M})(du_p^i)\Big)
$$
where for $p\geq 1$, 
$$
\check{\eta}_{p-1}^{N,M}(du) = \frac{1}{N}\sum_{i=1}^N \delta_{u_{p}^{i}}(du).
$$
and we set for  $s\in\{f,c\}$
$$
\eta_{p-1}^{N,s}(dx) = \frac{1}{N}\sum_{i=1}^N \delta_{x_{p-1}^{i,s}}(dx).
$$
This procedure is as in \cite{chopin3} and adopted in, for instance, \cite{jacob2,mlpf}. The idea is to provide a local optimality procedure
w.r.t.~the resampling operation. This is in the sense that for any fixed $N$ and conditional on the information generated so far, one will
maximize the probability that the resampled indices are equal.

For the MCRPF, we present a preliminary result, which will prove to be of interest. The following assumptions are used and note that in (A\ref{hyp:3}) there is a metric $\mathsf{d}$ on $\mathsf{X}\times\mathsf{X}$ implicit in the assumption.
\begin{hypA}\label{hyp:1}
For every $n\geq 0$, $G_n\in \mathcal{C}_b(\mathsf{X})$.
\end{hypA}
\begin{hypA}\label{hyp:2}
For every $n\geq 1$, $\check{M}_n$ is Feller.
\end{hypA}
\begin{hypA}\label{hyp:3}
$\mathsf{X}\times\mathsf{X}$ is a locally compact and separable metric space.
\end{hypA}
The assumptions are adopted due to the complexity of the operator $\check{\Phi}_n^M$. As can be seen in 
Appendix \ref{app:max_wlln}, where the proofs for the following result are given, it is non-trivial to work with $\check{\Phi}_n^M$. To weaken these assumptions would
lead to further calculations, which would essentially confirm the same result.

%The proof of the following result can be found in Appendix \ref{app:max_wlln}.

\begin{theorem}\label{theo:wlln_max}
Assume (A\ref{hyp:1}-\ref{hyp:3}). Then for any $\varphi\in\mathcal{C}_b(\mathsf{X}\times\mathsf{X})$, $n\geq 0$ we have
$$
\check{\eta}_{n}^{N,M}(\varphi) \rightarrow_{\mathbb{P}} \check{\eta}_{n}^{M}(\varphi).
$$
\end{theorem}

What is interesting here, is that on the \emph{product space} $\mathsf{X}\times\mathsf{X}$ the target $\check{\eta}_{n}^{M}$ is a coupling
of $(\eta_n^f,\eta_n^c)$, but the actual maximal coupling of $(\eta_n^f,\eta_n^c)$ is not sampled from. As is well known, the maximal coupling will maximize
the probability that two random variables are equal, with specified marginals and is the optimal coupling of two probability measures w.r.t.~the $L_0-$Wasserstein distance.
If this former property is desirable from a practical perspective, then the above algorithm should not be used. The maximially coupled resampling operation is, for a finite
number of samples (particles), the optimal (in the above sense) way to couple the resampling operation, but, may not lead to large sample `good' couplings. This is manifested in
\cite{mlpf} where the forward error rate \eqref{eq:coup_h_cont} is lost for the diffusion problem in Section \ref{sec:diff_ex}. As mentioned above, the limit is a coupling of $(\eta_n^f,\eta_n^c)$, but in general there is no reason to suspect that it is optimal
in any sense.

\subsection{Maximal Coupling}

We now present an algorithm which can sample (in the limit) from the maximal coupling of $(\eta_n^f,\eta_n^c)$. We will assume that for $s\in\{f,c\}$ the Markov kernels $M_n^s$, as
well as $\eta_0^s,$ admit a density w.r.t.~a $\sigma-$finite measure $dx$. The densities are denoted $M_n^s$ and $\eta_0^s$, $s\in\{f,c\}$ and we assume
that the densities can be evaluated numerically. To remove this latter requirement is left to future work.

Let $n\geq 1$, $B\in\mathcal{X}\vee\mathcal{X}$ and $(\mu,\nu)\in\mathscr{P}(\mathsf{X})\times\mathscr{P}(\mathsf{X})$ and define the probability measure:
\begin{eqnarray*}
\check{\Phi}_n^{C}(\mu,\nu)(\varphi) & = & \int_{\mathsf{X}\times\mathsf{X}}\mathbb{I}_B(x,y) \Big[\int_{\mathsf{X}} 
F_{n-1,\mu,f}(u)\wedge F_{n-1,\nu,c}(u) \delta_{\{u,u\}}(d(x,y)) du + \\ & & \frac{1}{1-\int_{\mathsf{X}} 
F_{n-1,\mu,f}(u)\wedge F_{n-1,\nu,c}(u)du}
\overline{F}_{n-1,\mu,\nu,f}(x)\overline{F}_{n-1,\nu,\mu,c}(y) dxdy\Big]
\end{eqnarray*}
where, for $x\in\mathsf{X}$
\begin{eqnarray*}
\overline{F}_{n-1,\mu,\nu,f}(x) & = & F_{n-1,\mu,f}(x) -  F_{n-1,\mu,f}(x)\wedge F_{n-1,\nu,c}(x)\\
\overline{F}_{n-1,\nu,\mu,c}(x) & = &  F_{n-1,\nu,c}(x) -  F_{n-1,\mu,f}(x)\wedge F_{n-1,\nu,c}(x)\\
F_{n-1,\mu,f}(x) & = & \frac{\mu(G_{n-1}M_n^f(\cdot,x))}{\mu(G_{n-1})}\\
 F_{n-1,\nu,c}(x) & = & \frac{\nu(G_{n-1}M_n^c(\cdot,x))}{\nu(G_{n-1})}.
\end{eqnarray*}
Now set $\check{\eta}_0^C$ as the maximal coupling of $(\eta_0^f,\eta_0^c)$ and  set for $B\in\mathcal{X}\vee\mathcal{X}$
$$
\check{\eta}_n^C(B) = \check{\Phi}_n^{C}(\eta_{n-1}^f,\eta_{n-1}^c)(B).
$$
We have for $B\in\mathcal{X}$
$$
\check{\eta}_n^C(B\times\mathsf{X}) = \eta_n^f(B)\quad\textrm{and}\quad\check{\eta}_n^C(\mathsf{X}\times B) = \eta_n^c(B).
$$
Moreover $\check{\eta}_n^C$ is the maximal coupling of $(\eta_n^f,\eta_n^c)$.

The particle approximation used is:
$$
\mathbb{P}(d(u_0^{1:N},\dots,u_n^{1:N})) = \Big(\prod_{i=1}^N \check{\eta}_0^C(du_0^i)\Big)\Big(\prod_{p=1}^n\prod_{i=1}^N\check{\Phi}_p^C(\eta_{p-1}^{N,f},
\eta_{p-1}^{N,c})(du_p^i)\Big)
$$
where for $p\geq 1$, $s\in\{f,c\}$
$$
\eta_{p-1}^{N,s}(dx) = \frac{1}{N}\sum_{i=1}^N \delta_{x_{p-1}^{i,s}}(dx).
$$
We set for $\mu\in\mathscr{P}(\mathsf{X})$, $s\in\{f,c\}$, $B\in\mathcal{X}$, $n\geq 1$
$$
\Phi_n^s(\mu)(B) = \frac{\mu(G_{n-1}M_n^s(B))}{\mu(G_{n-1})}.
$$
%We abuse the notation and say $X\sim\Phi_n^s(\mu)$ and write for $x\in\mathsf{X}$
%$$
%\Phi_n^s(\mu)(x) = \frac{\int_{\mathsf{X}}G_{n-1}(y)M_n^s(y,x)\mu(dy)}{\mu(G_{n-1})}.
%$$

We remark that as $M_n^s$, $n\geq 1$ and $\eta_0^s$, $s\in\{f,c\}$, can be evaluated numerically we can sample from $\check{\eta}_0^C$ and $\check{\Phi}_n^{C}(\eta_{n-1}^{N,f},\eta_{n-1}^{N,c})$, the latter of which is 
the maximal coupling of $\Phi_n^f(\eta_{n-1}^{N,f})$ and $\Phi_n^c(\eta_{n-1}^{N,c})$, using the algorithm in \cite{thor}. This is as follows:
\begin{enumerate}
\item{Sample $X_n^f\sim\Phi_n^f(\eta_{n-1}^{N,f})$ and $W|x_n^f\sim\mathcal{U}_{[0,F_{n-1,\eta_{n-1}^{N,f},f}(x_n^f)]}$. If $W<F_{n-1,\eta_{n-1}^{N,c},c}(x_n^f)$, output $(x_n^f,x_n^f)$, otherwise go-to 2..}
\item{Sample $X_n^c\sim\Phi_n^c(\eta_{n-1}^{N,c})$ and $\tilde{W}|x_n^c\sim\mathcal{U}_{[0,F_{n-1,\eta_{n-1}^{N,c},c}(x_n^c)]}$, if $\tilde{W}>F_{n-1,\eta_{n-1}^{N,f},f}(x_n^c)$, output $(x_n^f,x_n^c)$, otherwise start 2.~again.}
\end{enumerate}
This algorithm is a rejection sampler, which
would add a random running time element per time-step, which may not be desirable for some applications. In the limiting case (i.e.~sampling the maximal coupling of $(\eta_n^f,\eta_n^c)$) the
expected number of steps to acceptance is 
$$
1 + \frac{\int_{\mathsf{X}}|\eta_n^c(x)-\eta_n^f(x)|dx}{2\int_{\mathsf{X}}[\eta_n^c(x)-\eta_n^f(x)]\mathbb{I}_{\{x\in\mathsf{X}:\eta_n^c(x)\geq\eta_n^f(x)\}}(x)dx} = 1 + R_n.
$$
One might expect that $R_n$ above is $\mathcal{O}(1)$ in applications. We expect $(\eta_n^f,\eta_n^c)$ to be close and possibly that $\eta_n^c$ is more diffuse than $\eta_n^f$, for instance
that $\eta_n^c(x)/\eta_n^f(x)\geq C$, for some constant $C\leq 1$, hence one may find $R_n=\mathcal{O}(1)$. In general, one might want to confirm this conjecture before implementing this approach.

%This approach is sampling the \emph{maximal coupling}
%of the selection-mutation operators $\Phi_n^f(\eta_{n-1}^{N,f})$ and $\Phi_n^c(\eta_{n-1}^{N,c})$.

\subsection{Wasserstein Coupled Resampling}

We describe the WCPF used in \cite{mlpf_new}.
For this case we restrict our attention to the case that $\mathsf{X}=\mathbb{R}$. It is explicitly assumed that the cumulative distribution function (CDF) and its inverse associated to the probability, $s\in\{f,c\}$, $n\geq 0$, $B\in\mathcal{X}$
$$
\overline{\eta}_n^s(B) = \frac{\eta_n^s(G_n\mathbb{I}_B)}{\eta_n^s(G_n)}
$$
exists and are continuous functions. We denote the CDF (resp.~inverse) of $\overline{\eta}_n^s$ as $F_{\overline{\eta}_n^s}$ (resp.~$F_{\overline{\eta}_n^s}^{-1}$ ).
In general we write probability measures  on $\mathsf{X}$ for which the CDF and inverse are well-defined as $\mathscr{P}_F(\mathsf{X})$ with the 
associated CDF $F_{\mu}$.
 Let $n\geq 1$, $B\in\mathcal{X}\vee\mathcal{X}$ and $\mu,\nu\in\mathscr{P}_F(\mathsf{X})$ and define the probability measure:
$$
\check{\Phi}_n^W(\mu,\nu)(B)  = \int_{\mathsf{X}\times\mathsf{X}}  \Big(\int_{0}^1 \delta_{\{F_{\mu}^{-1}(w),F_{\nu}^{-1}(w)\}}(du) dw \Big) \check{M}_n(B)(u).
$$
Consider the joint probability measure on $(\mathsf{X}\times\mathsf{X})^{n+1}$
$$
\mathbb{P}(d(u_0,\dots,u_n)) = \check{\eta}_0(du_0) \prod_{p=1}^n \check{\Phi}_p^W(\overline{\eta}_{p-1}^f,\overline{\eta}_{p-1}^c)(du_p).
$$
It is easily checked that the marginal of $x_n^f$ (resp.~$x_n^c$) is $\eta_n^f$ (resp.~$\eta_n^c$). Denote by $\check{\eta}_n^W$ the marginal of $u_n$ induced by this joint probability measure.

The algorithm used here is:
$$
\mathbb{P}(d(u_0^{1:N},\dots,u_n^{1:N})) = \Big(\prod_{i=1}^N \check{\eta}_0(du_0^i)\Big)\Big(\prod_{p=1}^n\prod_{i=1}^N\check{\Phi}_p^W(\overline{\eta}_{p-1}^{N,f},\overline{\eta}_{p-1}^{N,c})(du_p^i)\Big)
$$
where  for $p\geq 1$, $s\in\{f,c\}$
$$
\overline{\eta}_{p-1}^{N,s}(dx) = \sum_{i=1}^N \frac{G_{p-1}(x_{p-1}^{i,s})}{\sum_{j=1}^N G_{p-1}(x_{p-1}^{j,s})}\delta_{x_{p-1}^{i,s}}(dx).
$$
As before for $p\geq 1$, $s\in\{f,c\}$
$$
\eta_{p-1}^{N,s}(dx) = \frac{1}{N}\sum_{i=1}^N \delta_{x_{p-1}^{i,s}}(dx)
$$
and for $p\geq 0$
$$
\check{\eta}_{p}^{N,W}(du) = \frac{1}{N}\sum_{i=1}^N \delta_{u_{p}^{i}}(du).
$$

\section{Theoretical Results}\label{sec:theory}

This section is split into two. We give our CLTs in Section \ref{sec:clt_state} and bounds on the asymptotic variance in Section \ref{sec:av_bound_disc}.

\subsection{Central Limit Theorems}\label{sec:clt_state}

Denote the sequence of non-negative kernels $\{Q_n^s\}_{n\geq 1}$, $s\in\{f,c\}$, $Q_n^s(x,dy) = G_{n-1}(x) M_n^s(x,dy)$ and for $B\in\mathcal{X}$, $x_p\in\mathsf{X}$
$$
Q_{p,n}^s(B)(x_p) = \int_{\mathsf{X}^{n-p}} \mathbb{I}_B(x_n) \prod_{q=p}^{n-1} Q_{q+1}^s(x_q,dx_{q+1})
$$
$0\leq p<n$ and in the case $p=n$, $Q_{p,n}^s$ is the identity operator.
Now denote for $0\leq p<n$, $s\in\{f,c\}$, $B\in\mathcal{X}$, $x_p\in\mathsf{X}$
$$
D_{p,n}^s(B)(x_p) = \frac{Q_{p,n}^s(\mathbb{I}_B - \eta_n^s(B))}{\eta_p^s(Q_{p,n}^s(1))}
$$
in the case $p=n$, $D_{p,n}^s(B)(x)=\mathbb{I}_B-\eta_n^s(B)$.

We then have our CLTs, where the proofs are given in Appendix \ref{app:clt_prf}. Recall that for the MCPF $M_n^s$ is used as a notation for the kernel and density of $M_n^s$.

\begin{theorem}\label{theo:clt_ind}
We have
\begin{enumerate}
\item{For the IRCPF:
For any $\varphi\in\mathcal{B}_b(\mathsf{X})$, $n\geq 0$ we have
$$
\sqrt{N}[\check{\eta}_{n}^{N,I}-\check{\eta}_{n}^{I}](\varphi\otimes 1 - 1 \otimes \varphi) \Rightarrow \mathcal{N}(0,\sigma_n^{2,I}(\varphi))
$$
where
$$
\sigma_n^{2,I}(\varphi) = \sum_{p=0}^n \check{\eta}_p^I(\{(D_{p,n}^f(\varphi)\otimes 1- 1\otimes D_{p,n}^c(\varphi))\}^2).
$$
}
\item{For the MCRPF: Assume (A\ref{hyp:1}-\ref{hyp:3}), then
for any $\varphi\in\mathcal{C}_b(\mathsf{X})$, $n\geq 0$ we have
$$
\sqrt{N}[\check{\eta}_{n}^{N,M}-\check{\eta}_{n}^{M}](\varphi\otimes 1 - 1 \otimes \varphi) \Rightarrow \mathcal{N}(0,\sigma_n^{2,M}(\varphi))
$$
where
$$
\sigma_n^{2,M}(\varphi) = \sum_{p=0}^n \check{\eta}_p^M(\{(D_{p,n}^f(\varphi)\otimes 1- 1\otimes D_{p,n}^c(\varphi))\}^2).
$$
}
\item{For the MCPF: Suppose that for $s\in\{f,c\}$, %$\eta_0^s\in\mathcal{B}_b(\mathsf{X})$ and
$n\geq 1$, $M_n^s\in\mathcal{B}_b(\mathsf{X}\times\mathsf{X})$. % $n\geq 1$.% and $\eta_0^s\in\mathscal{B}_(\mathsf{X})$, $s\in\{f,c\}$
Then for any $\varphi\in\mathcal{B}_b(\mathsf{X})$, $n\geq 0$ we have
$$
\sqrt{N}[\check{\eta}_{n}^{N,C}-\check{\eta}_{n}^{C}](\varphi\otimes 1 - 1 \otimes \varphi) \Rightarrow \mathcal{N}(0,\sigma_n^{2,C}(\varphi))
$$
where
$$
\sigma_n^{2,C}(\varphi) = \sum_{p=0}^n \check{\eta}_p^C(\{(D_{p,n}^f(\varphi)\otimes 1- 1\otimes D_{p,n}^c(\varphi))\}^2).
$$
}
\end{enumerate}
\end{theorem}

For Wasserstein resampling, the following result is established in \cite{mlpf_new}.

\begin{theorem}\label{theo:clt_was}
Suppose that $\check{M}_n$, $M_n^s$, $s\in\{f,c\}$  are Feller for every $n\geq 1$ and $G_n\in\mathcal{C}_b(\mathsf{X})$ for every $n\geq 0$.
Then for any $\varphi\in\mathcal{C}_b(\mathsf{X})$, $n\geq 0$ we have
$$
\sqrt{N}[\check{\eta}_{n}^{N,W}-\check{\eta}_{n}^{W}](\varphi\otimes 1 - 1 \otimes \varphi) \Rightarrow \mathcal{N}(0,\sigma_n^{2,W}(\varphi))
$$
where
$$
\sigma_n^{2,W}(\varphi) = \sum_{p=0}^n \check{\eta}_p^W(\{(D_{p,n}^f(\varphi)\otimes 1- 1\otimes D_{p,n}^c(\varphi))\}^2).
$$
\end{theorem}

\begin{rem}
One can also prove a mutlivariate CLT using the Cramer-Wold device.  
Consider $1\leq t <+\infty$, $(\varphi_1,\dots,\varphi_t,\psi_1,\dots,\psi_t)\in\mathcal{B}_b(\mathsf{X})^{2t}$ (or if required $\mathcal{C}_b(\mathsf{X})^{2t}$).
Consider the $t\times t$ positive definite and symmetric matrix $\Sigma_{n}^{s}(\varphi_{1:t},\psi_{1:t})$, $s\in\{I,M,C,W\}$, with $(i,j)$th entry denoted $\Sigma_{n,(ij)}^{s}(\varphi_{1:t},\psi_{1:t})$.
Using the Cramer-Wold device under the various assumptions of each the algorithms one can easily deduce that for each $s\in\{I,M,C,W\}$
$$
\sqrt{N}([\check{\eta}_{n}^{N,s}-\check{\eta}_{n}^{s}](\varphi_1\otimes 1 - 1 \otimes \psi_1),\dots, 
[\check{\eta}_{n}^{N,s}-\check{\eta}_{n}^{s}](\varphi_t\otimes 1 - 1 \otimes \psi_t) ) \Rightarrow \mathcal{N}_t(0,\Sigma_n^{s}(\varphi_{1:t},\psi_{1:t}))
$$
where
$$
\Sigma_{n,(ij)}^{s}(\varphi_{1:t},\psi_{1:t}) = \sum_{p=0}^n \check{\eta}_p^s([D_{p,n}^f(\varphi_i)\otimes 1- 1\otimes D_{p,n}^c(\psi_i)]
[D_{p,n}^f(\varphi_j)\otimes 1- 1\otimes D_{p,n}^c(\psi_j)]).
$$
\end{rem}

\subsection{Asymptotic Variance}\label{sec:av_bound_disc}

We now consider bounding the asymptotic variance in the general case. The result below (Theorem \ref{theo:av_thm}), will allow one to understand when one can expect time-uniformly `close' errors in approximations of $[\eta_n^f-\eta_n^c](\varphi)$.
The following assumptions are used, which are essentially those in \cite{whiteley} (see also \cite{jasra}) with some additional assumptions ((H\ref{hypav:6}-\ref{hypav:7}) below) as we are treating a more delicate case than in \cite{whiteley}. The assumptions can hold on unbounded state-spaces as will be the case for our applications.

\begin{hypB}\label{hypav:1}
There exists a $\tilde{V}:\mathsf{X}\rightarrow[1,\infty)$ unbounded and constants $\delta\in(0,1)$ and $\underline{d}\geq 1$ with the following properties.
For each $d\in(\underline{d},+\infty)$ there exists a $b_d<+\infty$ such that $\forall x\in\mathsf{X}$ and any $s\in\{f,c\}$,  
$$
\sup_{n\geq 1} Q_n^s(e^{\tilde{V}})(x) \leq e^{(1-\delta)\tilde{V}(x) + b_d\mathbb{I}_{C_d}(x)}
$$
where $C_d=\{x\in\mathsf{X}:\tilde{V}(x)\leq d\}$.
\end{hypB}
\begin{hypB}\label{hypav:2}
\begin{enumerate}
\item{There exists a $C<+\infty$ such that for any $s\in\{f,c\}$, $\eta_0^s(v)\leq C$, with $v=e^{\tilde{V}}$.}
\item{For any $r>1$ there exists a $C<+\infty$ such that for any $s\in\{f,c\}$, $\eta_0^s(C_r)^{-1}\leq C$.}
\end{enumerate}
\end{hypB}
\begin{hypB}\label{hypav:3}
With $\underline{d}$ as in (H\ref{hypav:1}), for each $d\in[\underline{d},\infty)$, $s\in\{f,c\}$, 
$$
\int_{C_d} G_{n-1}(x)M_n^s(x,dy) > 0~\forall x\in\mathsf{X}, n\geq 1
$$
and there exist $\tilde{\epsilon}_d^{-}>0$, $\nu_d\in\mathcal{P}_v$ such that for each $A\in\mathcal{X}$, $s\in\{f,c\}$, 
$$
\inf_{n\geq 1} \int_{C_d\cap A} Q_n^s(x,dy) \geq \tilde{\epsilon}_d^{-}\nu(C_d\cap A),~\forall x\in C_d.
$$
%In addition $\nu_d(dy):=\lambda(dy)\mathbb{I}_{C_d}(y)/\int_{C_d}\lambda(dy)\in\mathcal{P}_v$.
\end{hypB}
\begin{hypB}\label{hypav:4}
With $\underline{d}$ as in (H\ref{hypav:1}),  and $\tilde{\epsilon}_d^-$, $\nu_d$ as in (H\ref{hypav:2}), for each $d\in[\underline{d},\infty)$ there exist $\tilde{\epsilon}_d^+\in[\tilde{\epsilon}_d^-,\infty)$ such that for each $A\in\mathcal{X}$, $s\in\{f,c\}$, 
$$
\sup_{n\geq 1}\int_{C_d\cap A} Q_n^s(x,dy)  \leq \tilde{\epsilon}_d^+\nu(C_d\cap A), ~\forall x\in C_d
$$
\end{hypB}
\begin{hypB}\label{hypav:5}
\begin{enumerate}
\item{
$$
\sup_{n\geq 0}\sup_{x\in\mathsf{X}}G_n(x) <+\infty.
$$
}
\item{For any $r>1$ there exists a $C<+\infty$ such that $\inf_{x\in C_r}G_0(x)\geq C$.}
\end{enumerate}
\end{hypB}
\begin{hypB}\label{hypav:6}
With $\underline{d}$ as in (H\ref{hypav:1}), 
for each $d\in[\underline{d},\infty)$, we have for each $n\geq 0$, $s\in\{f,c\}$
$$
\frac{1}{G_n M_{n+1}^s(\mathbb{I}_{C_d})}\in\mathcal{L}_{v}(\mathsf{X})
$$
and $\sup_{n\geq 0}\max_{s\in\{f,c\}}\|1/G_n M_{n+1}^s(\mathbb{I}_{C_d})\|_{v}<+\infty$.
\end{hypB}
\begin{hypB}\label{hypav:7}
Let $\mathsf{d}$ be a given metric on $\mathsf{X}$.
For any $\xi\in(0,1]$, there exist a $C<+\infty$ such that for $\varphi\in\mathcal{L}_{v^{\xi}}(\mathsf{X})\cap\textrm{Lip}_{v^{\xi},\mathsf{d}}(\mathsf{X})$, $(x,y)\in\mathsf{X}\times\mathsf{X}$, $s\in\{f,c\}$
$$
\sup_{n\geq 1}|Q_{n}^s(\varphi)(x)-Q_{n}^s(\varphi)(y)| \leq C\|\varphi\|_{v^{\xi}}\mathsf{d}(x,y)[v(x)v(y)]^{\xi}.
$$
\end{hypB}

Let $A=\{(x,y)\in\mathsf{X}\times\mathsf{X}:x\neq y\}$. Define for $n\geq 1$, $0\leq p \leq n$, $s\in\{f,c\}$, $x\in\mathsf{X}$
$$
h_{p,n}^s(x) := \frac{Q_{p,n}^s(1)(x)}{\eta_p^s(Q_{p,n}^s(1))}
$$
and, with $B\in\mathcal{X}$,
$$
S_{p,n}^{f,c}(B)(x) := \frac{Q_{p,n}^f(B)(x)}{Q_{p,n}^f(1)(x)} - \frac{Q_{p,n}^c(B)(x)}{Q_{p,n}^c(1)(x)}.
$$
Set
$$
B(n,f,c,\varphi,\xi) = 
\|\varphi\|\sum_{p=0}^{n-1}\rho^{n-p}\Big\{\|h_{p,n}^fS_{p,n}^{f,c}(\varphi)\|_{v^{\xi}} + |[\eta_n^f-\eta_n^c](\varphi)| + \|\varphi\|\|h_{p,n}^f-h_{p,n}^c\|_{v^{\xi}}\rho^{n-p}\Big\}.
$$
%Let
%$$
%\tilde{\varphi}(x,y) = |x-y|^2
%$$
%where $|\cdot|$ is the $L_1$-norm.

Below is our main result, whose proof can be found in Appendix \ref{app:av_proofs}.

\begin{theorem}\label{theo:av_thm}
Assume (H\ref{hypav:1}-\ref{hypav:6}). Then for any $\xi\in(0,1/4)$ there exists a $\rho<1$ and  $C<+\infty$ depending on the constants in  
(H\ref{hypav:1}-\ref{hypav:6}) such that for any $\varphi\in\mathcal{B}_b(\mathsf{X})$, $n\geq 1$, $s\in\{I,M,C,W\}$ we have
$$
\sigma^{2,s}_n(\varphi) \leq \check{\eta}_n^s\Big(\{\varphi\otimes 1-1\otimes\varphi - ([\eta_n^f-\eta_n^c](\varphi))\}^2\Big) + 
$$
$$
C\Big\{B(n,f,c,\varphi,\xi) + \|\varphi\|^2\sum_{p=0}^{n-1}\rho^{n-p}\Big(\check{\eta}_p^{s}(\mathbb{I}_A(v\otimes v)^{2\xi})(\rho^{n-p}+1)\Big)\Big\}.
$$
Additionally assume (H\ref{hypav:7}). Then if $\mathsf{d}^2\in\mathcal{L}_{(v\otimes v)^{\tilde{\xi}}}$, $\tilde{\xi}\in(0,1/2)$, for 
any $\xi\in(0,(1-2\tilde{\xi})/12))$ there exists a $\rho<1$ and  $C<+\infty$ depending on the constants in  
(H\ref{hypav:1}-\ref{hypav:7}) such that for any $\varphi\in\mathcal{B}_b(\mathsf{X})\cap\textrm{\emph{Lip}}_{\mathsf{d}}(\mathsf{X})$, $n\geq 1$, $s\in\{I,M,C,W\}$ we have
$$
\sigma^{2,s}_n(\varphi) \leq \check{\eta}_n^s\Big(\{\varphi\otimes 1-1\otimes\varphi - ([\eta_n^f-\eta_n^c](\varphi))\}^2\Big) + 
$$
$$
C\Big\{B(n,f,c,\varphi,\xi) + \|\varphi\|^2\sum_{p=0}^{n-1}\rho^{n-p}\Big(\check{\eta}_p^{s}(\mathsf{d}^2(v^{4\xi}\otimes v^{8\xi}))(\rho^{n-p}+1)\Big)\Big\}.
$$
\end{theorem}

\begin{rem}
The range of $\xi$ are such that the upper-bounds are finite. This is because $\check{\eta}_p^s(v)$ is upper-bounded by a constant that does not depend on $p$ (see \cite[Proposition 1]{whiteley}).
\end{rem}

The main conclusion of this Theorem is that one expects that the (variances of) approximations of $[\eta_n^f-\eta_n^c](\varphi)$ are uniformly `close' in time if the following are satisfied:
\begin{enumerate}
\item{That
$$
\check{\eta}_n^s\Big(\{\varphi\otimes 1-1\otimes\varphi - ([\eta_n^f-\eta_n^c](\varphi))\}^2\Big)\quad\textrm{and}\quad|[\eta_n^f-\eta_n^c](\varphi)| 
$$
are small uniformly in time.
}
\item{The differences
$$
 \|h_{p,n}^fS_{p,n}^{f,c}(\varphi)\|_{v^{\xi}}\quad\textrm{and}\quad\|h_{p,n}^f-h_{p,n}^c\|_{v^{\xi}}
$$
are small at a rate which decays more slowly than exponentially in time.
}
\item{That
$$
\check{\eta}_p^{s}(\mathbb{I}_A(v\otimes v)^{2\xi}) \quad\textrm{or}\quad\check{\eta}_p^{s}(\mathsf{d}^2(v^{4\xi}\otimes v^{8\xi}))
$$
are small at a rate which decays more slowly than exponentially in time.}
\end{enumerate}
We note that in 1.~strictly one could have $|[\eta_n^f-\eta_n^c](\varphi)|$ close at a rate which decays more slowly than exponentially in time, but as one requires
the time uniform closeness of $\check{\eta}_n^s(\{\varphi\otimes 1-1\otimes\varphi - ([\eta_n^f-\eta_n^c](\varphi))\}^2)$, one expects this property for the former.
The use of $A$ and $\mathsf{d}$ are linked to
the coupling properties associated to $\check{\eta}_p^{s}$ as we consider in the next section.

\section{Application to Partially Observed Diffusions}\label{sec:appl}

We now return to the PODDP model considered in Section \ref{sec:diff_ex}. 
Throughout $1\leq l\leq L$ is fixed, with $L>1$ fixed.
We will consider the significance of the results in Theorems \ref{theo:clt_ind}, \ref{theo:clt_was} and \ref{theo:av_thm} for each of the four CPFs.
We also present an example in Section \ref{sec:verify} where (H\ref{hypav:1}-\ref{hypav:8}) can be verified.
We begin with a control of the term $B(n,l,l-1,\varphi,\xi)$.

\subsection{A General Result}

Define the set, for $x\in\mathsf{X}$
$$
\mathsf{B}_l(x) := \{y\in\mathsf{X}:\|y-x\|>2\Delta_l\}
$$
where $\|\cdot\|$ is the $L_2-$norm.
We write the transition densities w.r.t.~Lebesgue measure $dy$ of the diffusion as $M(x,y)$ and the Euler approximation as $M^l(x,y)$,
$(x,y)\in\mathsf{X}\times\mathsf{X}$.
We note the following two equations quoted in \cite{delm:01}: there exists a $0<C,C'<+\infty$ (which depend on the drift and diffusion coefficients of \eqref{eq:sde}) such that for any $l\geq 0$
\begin{eqnarray}
M(x,y) + M^l(x,y) & \leq & C\exp\{-C'\|y-x\|^2\}\quad\forall~(x,y)\in\mathsf{X}\times\mathsf{X} \label{eq:diff_1}\\
|M(x,y) -  M^l(x,y)| & \leq & C\Delta_l\exp\{-C'\|y-x\|^2\}\quad\forall~x\in \mathsf{X},y\in \mathsf{B}_l(x). \label{eq:diff_2} 
\end{eqnarray}

We add an additional assumption:
\begin{hypB}\label{hypav:8}
We have, for $C'$ as in \eqref{eq:diff_1}-\eqref{eq:diff_2}:
\begin{enumerate}
\item{For any $\xi\in(0,1/2)$, there exists a $C<+\infty$ such that for any $x\in\mathsf{X}$, $l\geq 0$
$$
\int_{\mathsf{B}_l(x)^c}v(y)^{\xi}dy \leq C v(x)^{2\xi}\int_{\mathsf{B}_l(x)^c}dy.
$$
}
\item{For any $\xi\in(0,1/2)$, there exists a $C<+\infty$ such that for any $x\in\mathsf{X}$, $l\geq 0$
$$
\int_{\mathsf{B}_l(x)}v(y)^{\xi}\exp\{-C'\|y-x\|^2\}dy \leq C v(x)^{2\xi}.
$$
}
\end{enumerate}
\end{hypB}

We have the following result whose proof is in Appendix \ref{app:diff_case}, Section \ref{app:diff_case_b}.

\begin{prop}\label{prop:diff_general_b_bound}
Assume (H\ref{hypav:1}-\ref{hypav:6}), (H\ref{hypav:8}). Then for any $(\xi,\hat{\xi})\in(0,1/8)\times(0,1/2)$  there exists a $C<+\infty$ depending on the constants in (H\ref{hypav:1}-\ref{hypav:6}) (H\ref{hypav:8}) such that for any $\varphi\in\mathcal{B}_b(\mathsf{X})$ $n\geq 1$, $1\leq l \leq L$, we have
$$
B(n,l,l-1,\varphi,\xi) \leq C\|\varphi\|^2\Delta_lv(x^*)^{2\xi+\hat{\xi}}.
$$
\end{prop}

The main point of this result is that now the time-uniform coupling ability of each algorithm will now rely on the properties of the limiting coupling $\check{\eta}_n^s$.

\subsection{IRCPF}

We consider the term
$$
\check{\eta}_n^I\Big(\{\varphi\otimes 1-1\otimes\varphi - ([\eta_n^f-\eta_n^c](\varphi))\}^2\Big)
$$
in the upper-bound of Theorem \ref{theo:av_thm}. 
Let us suppose that $\varphi\in\mathcal{A}$, where $\mathcal{A}$ is defined below \eqref{eq:markov_cont}. 
One has that
$$
\check{\eta}_n^I\Big(\{\varphi\otimes 1-1\otimes\varphi - ([\eta_n^f-\eta_n^c](\varphi))\}^2\Big) = 
$$
$$
\frac{(\eta_{n-1}^f\otimes\eta_{n-1}^c)([G_{n-1}\otimes G_{n-1}]\check{M}_n([\varphi\otimes1-1\otimes\varphi]^2)) }{(\eta_{n-1}^f\otimes\eta_{n-1}^c)(G_{n-1}\otimes G_{n-1})}
- [\eta_n^f-\eta_n^c](\varphi)^2.
$$
Then using $\varphi\in\mathcal{A}$ (along with the $C_2-$inequality) and \cite[Lemma D.2.]{mlpf},
there is a finite constant $C<+\infty$ that depends on $n$ such that
$$
\check{\eta}_n^I\Big(\{\varphi\otimes 1-1\otimes\varphi - ([\eta_n^f-\eta_n^c](\varphi))\}^2\Big) \leq C\Big(\frac{(\eta_{n-1}^f\otimes\eta_{n-1}^c)([G_{n-1}\otimes G_{n-1}]\check{M}_n(\tilde{\varphi})) }{(\eta_{n-1}^f\otimes\eta_{n-1}^c)(G_{n-1}\otimes G_{n-1})}+\Delta_l^2\Big)
$$
where $\tilde{\varphi}(x,y) = \|x-y\|^2$.
Now applying \eqref{eq:coup_h_cont} one has the upper-bound
$$
C\Big(\frac{(\eta_{n-1}^f\otimes\eta_{n-1}^c)([G_{n-1}\otimes G_{n-1}]\tilde{\varphi}) }{(\eta_{n-1}^f\otimes\eta_{n-1}^c)(G_{n-1}\otimes G_{n-1})}+\Delta_l+\Delta_l^2\Big).
$$
In general, there is no reason to expect that $(\eta_{n-1}^f\otimes\eta_{n-1}^c)([G_{n-1}\otimes G_{n-1}]\tilde{\varphi})$ is small as a function of $\Delta_l$. This is unsuprising because one
uses an independent coupling in the resampling operation. As a result, in the sense of Section \ref{sec:mlmc_intro} for ML estimation, the IRCPF is not useful. %This is well-known in the literature, based upon empirical simulations.

\subsection{MCRPF}

To start our discussion, suppose first that $G_n$ is constant for each $n\geq 0$. This represents the most favourable scenario for the MCRPF, because in the resampling operation, the resampled indices for each pair are equal (of course one would not use resampling
in this case).  For the metric in (H\ref{hypav:7}), we take $\mathsf{d}_1$ the $L_1-$norm and set $\varphi\in\mathsf{B}_b(\mathsf{X})\cap\textrm{Lip}_{\mathsf{d}_1}(\mathsf{X})$ ($\mathsf{X}=\mathbb{R}^d$). Then
setting $\tilde{\varphi}(x,y)=\|x-y\|^2$
$$
\check{\eta}_n^M\Big(\{\varphi\otimes 1-1\otimes\varphi - ([\eta_n^f-\eta_n^c](\varphi))\}^2\Big) \leq C\|\varphi\|_{\textrm{Lip}}^2\check{\eta}_n^M(\tilde{\varphi})
$$
where $C<+\infty$ does not depend on $n$. Now as $G_n$ constant, one can deduce (e.g.~by \cite[Theorem 2.7.3]{mao}, recall (D) is assumed) that
\begin{equation}\label{eq:diff_bad_case_mcrpf}
\check{\eta}_n^M\Big(\{\varphi\otimes 1-1\otimes\varphi - ([\eta_n^f-\eta_n^c](\varphi))\}^2\Big) \leq Cn(n+4)\exp\{C'n^2\}\Delta_l
\end{equation}
where $C,C'<+\infty$ does not depend on $n$. Now \eqref{eq:diff_bad_case_mcrpf} is not necessarily tight in $n$ for every diffusion that satisfies (D), for instance
if one has $dZ_t = a dt + b dW_t$, for $b>0$, $a\in(-1,0)$, then there is no dependence on $n$ on the R.H.S.~of \eqref{eq:diff_bad_case_mcrpf}, however, it is worrying that the coupling
can be exponentially bad in time, in this favourable case. The main point here, is that the principal source of coupling in this algorithm is 
$\check{M}$ and if this coupling deteriorates when iterating $\check{M}$, then for the MRCPF one \emph{cannot hope to obtain time uniform couplings}. 

More generally, if $G_n$ is non-constant, we first remark that the upper-bound \eqref{eq:diff_bad_case_mcrpf} is likely to be $\mathcal{O}(\Delta_l^{1/2})$ as it is the resampling operation
where the forward rate is lost (see \cite{mlpf}). Secondly, one expects to find \emph{examples} where $\sigma^{2,s}_n(\varphi)$ is time-uniform or grows slowly as a function of $n$ (due to the empirical results in \cite{mlpf}). However, 
due to the highly non-linear and complex expression for $\check{\Phi}_n^M$, we expect a general result to be particularly arduous to obtain; this is left to future work.
%As noted in the introduction, this has been noted by several authors in the literature.

\subsection{MCPF}

We have the following result which establishes the time-uniform coupling of the MCPF. The proof can be found in Appendix \ref{app:diff_case}, Section \ref{app:diff_max_coup}.

\begin{prop}\label{prop:diff_max_coup}
Assume (H\ref{hypav:1}-\ref{hypav:6}), (H\ref{hypav:8}).  Then for any $(\xi,\hat{\xi})\in(0,1/32)\times(0,1/2)$ there exists a $C<+\infty$ depending on the constants in (H\ref{hypav:1}-\ref{hypav:6}) (H\ref{hypav:8}) such that for any $\varphi\in\mathcal{B}_b(\mathsf{X})$, $n\geq 0$, $1\leq l \leq L$ we have
$$
\sigma^{2,C}_n(\varphi) \leq C\|\varphi\|^2 \Delta_lv(x^*)^{32\xi+2\hat{\xi}}.
$$
\end{prop}

\subsection{WCPF}

Note that $\mathsf{X}=\mathbb{R}$ and, for the metric in (H\ref{hypav:7}), we take $\mathsf{d}_1$ the $L_1-$norm. Let $\tilde{\varphi}(x,y)=(x-y)^2$.
The proof of the following result, which establishes the time-uniform coupling of the WCPF, can be found in Appendix \ref{app:diff_case}, Section \ref{app:diff_was}.

\begin{prop}\label{prop:diff_was}
Assume (H\ref{hypav:1}-\ref{hypav:8}). Then, let $\lambda\in(0,1)$ be given,
if $\tilde{\varphi}\in\mathcal{L}_{(v\otimes v)^{\tilde{\xi}}}(\mathsf{X})$, for any $\tilde{\xi}\in(0,1/(16(1+\lambda)))$ and set $(\xi,\hat{\xi})\in(0,\min\{1/32,\lambda/(16(1+\lambda)),(1-2\tilde{\xi})/12\})\times (0,1/2)$ then there exists a $C<+\infty$ depending on the constants in (H\ref{hypav:1}-\ref{hypav:8}) such that for any 
$\varphi\in\mathcal{B}_b(\mathsf{X})\cap\textrm{\emph{Lip}}_{\mathsf{d}_1}(\mathsf{X})$, $n\geq 0$, $1\leq l \leq L$ we have
$$
\sigma^{2,W}_n(\varphi)  \leq C\|\varphi\|_{\textrm{\emph{Lip}}}^2\|\tilde{\varphi}\|_{v^{\tilde{\xi}}}(\Delta_l)^{1/(1+\lambda)}v(x^*)^{20\xi + (16(\lambda+1)\tilde{\xi}+2\hat{\xi})/(1+\lambda)}.
$$
\end{prop}

As $\lambda > 0$ in Proposition \ref{prop:diff_was} (and can be made close to zero) one almost has the (time-uniform) forward error rate for the WCPF. We believe that $\lambda=0$ is the case and discuss strategies to establish this in Remark \ref{rem:improve_rate} in 
Appendix \ref{app:diff_case}, Section \ref{app:diff_was}.

\subsection{Example}\label{sec:verify}

We consider $\mathsf{X}=\mathbb{R}$, the diffusion $dZ_t = -\frac{3}{2}Z_tdt+dW_t$ and $\mathsf{Y}=\{0,1\}$ and $G(x,y) = \Big(\frac{be^x+a}{(1+e^x)}\Big)^y\Big(1-(\frac{be^x+a}{(1+e^x)})\Big)^{1-y}$, $y\in\mathsf{Y}$ and any $0<a<b<1$.
The metric in (H\ref{hypav:7}) is the $L_1-$norm.
In practice the CPFs can only be run (with currently available computational power) with $L\leq 20$. So we will assume that there is a $L^*>L$ for one which one cannot run the algorithm (say $L^*=50$). This will reduce the complexity of the forthcoming discussion.
For the Euler discretization (although of course it is not required here) one can determine that the transition kernel $M^l(x,y)$ is the density of a $\mathcal{N}(\alpha_lx,\beta_l)$ distribution with $\alpha_l=(1-(3/2)\Delta_l)^{\Delta_l^{-1}}$, $\beta_l=\Delta_l(1-(1-(3/2)\Delta_l)^{2\Delta_l^{-1}})/(1-(1-(3/2)\Delta_l)^2)$.
Due to our assumption concerning $L^*$, one can easily find constants $0<\underline{\alpha}<\overline{\alpha}<1$,  $0<\underline{\beta}<\overline{\beta}\leq 1$, that do not depend on $l$, with $\underline{\alpha}\leq |\alpha_l|\leq \overline{\alpha}$, $\underline{\beta}\leq\beta_l\leq\overline{\beta}$ for each $l\in\{0,\dots,L\}$.

To verify (H\ref{hypav:1}) we use, as in \cite[Example 4.2.1]{whiteley1}, the Lyapunov function $\tilde{V}(x) = \frac{1}{2(1+\delta_0)}x^2 +1$ for some $\delta_0>0$. One can show that for any $1\leq l \leq L$, $n\geq 1$, $y\in\mathsf{X}$
$$
Q_n^l(e^v)(y) \leq \exp\Big\{\frac{\alpha_l^2(1+\delta_0)}{1+\delta_0-\beta_l}(\tilde{V}(y)-1) + C\Big\}
$$
where $[\alpha_l^2(1+\delta_0)]/[1+\delta_0-\beta_l]<1$ for every $0\leq l\leq L$, $\delta_0>0$ and $0<C<+\infty$ is a constant that does not depend on $l,n,\delta_0$ and can be made arbitrarily large. One can easily verify (H\ref{hypav:1}) for any $\delta_0>0$ with $\underline{d}>1$ large enough.  (H\ref{hypav:2}-\ref{hypav:5}) follow by elementary calculations associated to the choice of $\tilde{V}$, $M^l$ and $G_n$ and are omitted.

To verify (H\ref{hypav:6}) as $G_n$ is lower-bounded, we consider $M^l(C_d)(x)$ and suppose $x>0$, $l\geq 1$, $d>1$ (the case $x\leq 0$, $l=0$ can be verified in a similar manner). We have for any $x>0$:
$$
v(x) M^l(C_d)(x) \geq C \exp\Big\{\frac{x^2}{2}\Big(\frac{1}{\delta_0+1}-\frac{\alpha_l^2}{\beta_l}\Big)-\frac{\alpha_l}{\beta_l}x\tilde{d}\Big\}
$$
where $\tilde{d}=\sqrt{(d-1)2(1+\delta_0)}$. As $\delta_0$ is arbitrary one can choose $\delta_0>0$ so that $\frac{1}{\delta_0+1}-\frac{\alpha_l^2}{\beta_l}\geq 0$ for any $1\leq l \leq L$. Checking calculations on a computer for $L^*=50$, yields that $0<\delta_0
\leq 5$ suffices.

For (H\ref{hypav:7}) one has the following result, where the proof is in Appendix \ref{app:diff_case}, Section \ref{app:diff_verify}.

\begin{lem}\label{lem:verify}
(H\ref{hypav:7}) is satisfied in this example.
\end{lem}

To verify (H\ref{hypav:8}), we shall suppose that $C'$ in \eqref{eq:diff_2}  is at least $1+1/(2(1+\delta_0))$ (i.e.~$C'>1/12$). It is noted that it is non-trivial to check the value of $C'$ in \cite{bally}, but such a choice seems reasonable. In the given sitution, it is then straightforward to verify (H\ref{hypav:8}).

\section{Summary}\label{sec:summary}

We have considered CLTs for coupled particle filters and the associated asymptotic variance for applications. 
The main message is that it can be non-trivial to construct CPFs for which one inherits the appropriate `closeness' uniformly in time.
The MCPF and WCPF seem to be the best options, but suffer from the fact that (at least as considered in this paper) one requires
either the density of the transition and have a random running time per time step (MCPF) or is constrained to the case that $\mathsf{X}$ is
one-dimensional (WCPF). None-the-less these are still useful algorithms when they can be implemented. 

There are a number of extensions to this work. First one may extend
these results to the context of normalization constant estimation (e.g.~\cite{mlpf1}) and the associated asymptotic variance. Secondly, is a more in depth analysis and implementation of the MCPF,
which to our knowledge has not been performed in the literature.

\subsubsection*{Acknowledgements}
AJ \& FY were supported by an AcRF tier 2 grant: R-155-000-161-112.
AJ was supported under the
KAUST Competitive Research Grants Program-Round 4 (CRG4) project, Advanced Multi-Level sampling techniques for Bayesian Inverse Problems with applications to subsurface,
ref: 2584. We would like to thank Alexandros Beskos, Sumeetpal Singh and Xin Tong for useful conversations associated to this work.

\appendix

\section{Common Proofs for the CLT}\label{app:clt_prf}

The structure of all the appendices are to give the proof of the main result at the beginning. The associated technical results are given in such a way that they should be read in order to understand the overall proof.

For $p\geq 0$, $\varphi\in\mathcal{B}_b(\mathsf{X})$, $s\in\{f,c\}$
$$
V_p^{N,s}(\varphi) = \sqrt{N}[\eta_p^{N,s}-\Phi_p^s(\eta_{p-1}^{N,s})](\varphi)
$$
with the convention that $V_0^{N,s}(\varphi) = \sqrt{N}[\eta_0^{N,s}-\eta_{0})](\varphi)$.
For $p\geq 0$, $p\leq n$, $n>0$, $\varphi\in\mathcal{B}_b(\mathsf{X})$, $s\in\{f,c\}$,
$$
R_{p+1}^{N,s}(\varphi) = \frac{\eta_p^{N,s}(D_{p,n}^s(\varphi))}{\eta_p^{N,s}(G_p)}[\eta_p^{s}(G_p)-\eta_p^{N,s}(G_p)].
$$

Now we note that, using the calculations in \cite{mlsmc,ddj2012}, we have for $t\in\{I,M,C,W\}$
\begin{eqnarray}
\sqrt{N}[\check{\eta}_{n}^{N,t}-\check{\eta}_{n}^{t}](\varphi\otimes 1 - 1 \otimes \varphi) & = &  \sum_{p=0}^n \{V_p^{N,f}(D_{p,n}^f(\varphi)) + \sqrt{N}R_{p+1}^{N,f}(D_{p,n}^f(\varphi))\} - \nonumber\\ & &
\sum_{p=0}^n \{V_p^{N,c}(D_{p,n}^c(\varphi)) + \sqrt{N}R_{p+1}^{N,c}(D_{p,n}^c(\varphi))\}\label{eq:master_ind}.
\end{eqnarray}

\begin{proof}[Proof of Theorem \ref{theo:clt_ind}]
The proof follows immediately from \eqref{eq:master_ind}, Lemma \ref{lem:r_cont_ind}, and Proposition \ref{prop:field_ind}.
\end{proof}

\subsection{Technical Results}

The following results can be established for all four algorithms, but as the result for the Wasserstein method is given in \cite{mlpf_new}, only the other three cases are considered. When required we specify particular conditions required for a given algorithm. By default proofs are specified for the IRCPF case and the 
MRCPF \& MCPF is mentioned where required.

\begin{lem}\label{lem:r_cont_ind}
For $\varphi\in\mathcal{B}_b(\mathsf{X})$,  $n>0$, $0\leq p \leq n$, $s\in\{f,c\}$ we have
$$
\sqrt{N}R_{p+1}^{N,s}(\varphi) \rightarrow_{\mathbb{P}} 0.
$$
\end{lem}

\begin{proof}
By Proposition \ref{prop:lp_bound_ind} (for MCRPF it is \cite[Proposition C.6]{mlpf}, for MCPF it is Proposition \ref{prop:lp_bound_mcpf}) $\eta_p^{N,s}(G_p)$ converges in probability to a well-defined limit. Hence we need only show that
$$
\sqrt{N}\eta_p^{N,s}(D_{p,n}^s(\varphi))[\eta_p^{s}(G_p)-\eta_p^{N,s}(G_p)]
$$
will converge in probability to zero. By Cauchy-Schwarz:
$$
\sqrt{N}\mathbb{E}[|\eta_p^{N,s}(D_{p,n}^s(\varphi))[\eta_p^{s}(G_p)-\eta_p^{N,s}(G_p)]|] \leq 
$$
$$
\sqrt{N}\mathbb{E}[|\eta_p^{N,s}(D_{p,n}^s(\varphi))|^2]^{1/2}\mathbb{E}[|[\eta_p^{s}(G_p)-\eta_p^{N,s}(G_p)]|^2]^{1/2}.
$$
Applying Proposition \ref{prop:lp_bound_ind} (for MCRPF \cite[Proposition C.6]{mlpf}, for MCPF it is Proposition \ref{prop:lp_bound_mcpf}) it easily follows that there is a finite constant $C<+\infty$ that does not depend upon $N$ such that
$$
\sqrt{N}\mathbb{E}[|\eta_p^{N,s}(D_{p,n}^s(\varphi))[\eta_p^{s}(G_p)-\eta_p^{N,s}(G_p)]|] \leq \frac{C}{\sqrt{N}}.
$$
This bound allows one to easily conclude.
\end{proof}

\begin{lem}\label{lem:ind1}
For $\varphi\in\mathcal{B}_b(\mathsf{X})$, $p\geq 0$, $s\in\{f,c\}$ we have
\begin{eqnarray*}
\mathbb{E}[V_p^{N,s}(\varphi)] & = &  0 \\
\lim_{N\rightarrow+\infty}\mathbb{E}[V_p^{N,s}(\varphi)^2] & = & \eta_p^s((\varphi-\eta_p^s(\varphi))^2).
\end{eqnarray*}
\end{lem}

\begin{proof}
$\mathbb{E}[V_p^{N,s}(\varphi)] =  0$ follows immediately from the expression, so we focus on the second property.
\begin{eqnarray*}
\mathbb{E}[V_p^{N,s}(\varphi)^2] & = &  \frac{1}{N}\sum_{i=1}^N \mathbb{E}[(\varphi(X_p^{i,s})-\Phi^s_p(\eta_{p-1}^{N,s})(\varphi))^2] \\
& = & \mathbb{E}[\varphi(X_p^{1,s})^2] - \mathbb{E}[\Phi^s_p(\eta_{p-1}^{N,s})(\varphi)^2]
\end{eqnarray*}
$\Phi^s_p(\eta_{p-1}^{N,s})(\varphi)$ is a bounded random quantity and moreover by Proposition \ref{prop:lp_bound_ind} (MCRPF \cite[Proposition C.6]{mlpf}, MCPF Proposition \ref{prop:lp_bound_mcpf}) it converges in probability
to $\eta_p^s(\varphi)$. Hence by \cite[Theorem 25.12]{bill}, $\lim_{N\rightarrow+\infty}\mathbb{E}[\Phi^s_p(\eta_{p-1}^{N,s})(\varphi)^2] = \eta_p^s(\varphi)^2$.
Hence we consider 
$$
\mathbb{E}[\varphi(X_p^{1,s})^2] = \mathbb{E}[\frac{1}{N}\sum_{i=1}^N \varphi(X_p^{i,s})^2 - \eta_p^s(\varphi^2)] + \eta_p^s(\varphi^2).
$$
By Jensen
$$
\mathbb{E}[\frac{1}{N}\sum_{i=1}^N \varphi(X_p^{i,s})^2 - \eta_p^s(\varphi^2)] \leq \mathbb{E}[|\frac{1}{N}\sum_{i=1}^N \varphi(X_p^{i,s})^2 - \eta_p^s(\varphi^2)|^2]^{1/2}
$$
and hence we conclude via Proposition \ref{prop:lp_bound_ind} (MCRPF \cite[Proposition C.6]{mlpf}, MCPF Proposition \ref{prop:lp_bound_mcpf}) that 
$$
\lim_{N\rightarrow+\infty}\mathbb{E}[\varphi(X_p^{1,s})^2] = \eta_p^s(\varphi^2)
$$
and the result thus follows.
\end{proof}

\begin{lem}\label{lem:ind2}
We have:
\begin{enumerate}
\item{For the IRCPF: For $(\varphi,\psi)\in\mathcal{B}_b(\mathsf{X})^2$, $p\geq 0$  we have
$$
\lim_{N\rightarrow+\infty}\mathbb{E}[V_p^{N,f}(\varphi)V_p^{N,c}(\psi)] = \check{\eta}_p^I(\varphi\otimes \psi) - \eta_p^f(\varphi)\eta_p^c(\psi).
$$
}
\item{For the MCRPF: Assume (A\ref{hyp:1}-\ref{hyp:3}), 
$(\varphi,\psi)\in\mathcal{C}_b(\mathsf{X})^2$, $p\geq 0$ then we have
$$
\lim_{N\rightarrow+\infty}\mathbb{E}[V_p^{N,f}(\varphi)V_p^{N,c}(\psi)] = \check{\eta}_p^M(\varphi\otimes \psi) - \eta_p^f(\varphi)\eta_p^c(\psi).
$$
}
\item{For the MCPF:  Suppose that for $s\in\{f,c\}$, %$\eta_0^s\in\mathcal{B}_b(\mathsf{X})$ and
$n\geq 1$, $M_n^s\in\mathcal{B}_b(\mathsf{X}\times\mathsf{X})$.
 For $(\varphi,\psi)\in\mathcal{B}_b(\mathsf{X})^2$, $p\geq 0$ we have
$$
\lim_{N\rightarrow+\infty}\mathbb{E}[V_p^{N,f}(\varphi)V_p^{N,c}(\psi)] = \check{\eta}_p^C(\varphi\otimes \psi) - \eta_p^f(\varphi)\eta_p^c(\psi).
$$
}
\end{enumerate}
\end{lem}

\begin{proof}
We have for all methods:
\begin{eqnarray}
\mathbb{E}[V_p^{N,f}(\varphi)V_p^{N,c}(\psi)] & = & N\Big(\mathbb{E}[\eta_p^{N,f}(\varphi)\eta_p^{N,c}(\psi)]-\mathbb{E}[\eta_p^{N,f}(\varphi)\Phi_p^c(\eta_{p-1}^{N,c})(\psi)]-\nonumber\\ & & 
\mathbb{E}[\Phi_p^f(\eta_{p-1}^{N,f})(\varphi)\eta_p^{N,c}(\psi)] + \mathbb{E}[\Phi_p^f(\eta_{p-1}^{N,f})(\varphi)\Phi_p^c(\eta_{p-1}^{N,c})(\psi)]
\Big) \nonumber\\
& = &  N\Big(\mathbb{E}[\eta_p^{N,f}(\varphi)\eta_p^{N,c}(\psi)] - \mathbb{E}[\Phi_p^f(\eta_{p-1}^{N,f})(\varphi)\Phi_p^c(\eta_{p-1}^{N,c})(\psi)]
\Big).\label{eq:max_ind_1}
\end{eqnarray}
We consider calculations for (i)-(iii) associated to \eqref{eq:max_ind_1}.

\textbf{Proof for (i):} Now
\begin{eqnarray*}
N\mathbb{E}[\eta_p^{N,f}(\varphi)\eta_p^{N,c}(\psi)] & = & \frac{1}{N}\sum_{i=1}^N\sum_{j=1}^N\mathbb{E}[\varphi(X_p^{i,f})\psi(X_p^{i,c})] \\
& = & \mathbb{E}[\check{\Phi}_p^I(\eta_{p-1}^{N,f}\otimes \eta_{p-1}^{N,c})(\varphi\otimes\psi)] + (N-1)\mathbb{E}[\Phi_p^f(\eta_{p-1}^{N,f})(\varphi)\Phi_p^c(\eta_{p-1}^{N,c})(\psi)].
\end{eqnarray*}
Thus,
$$
\mathbb{E}[V_p^{N,f}(\varphi)V_p^{N,c}(\psi)] = \mathbb{E}[\check{\Phi}_p^I(\eta_{p-1}^{N,f}\otimes \eta_{p-1}^{N,c})(\varphi\otimes\psi)] - \mathbb{E}[\Phi_p^f(\eta_{p-1}^{N,f})(\varphi)\Phi_p^c(\eta_{p-1}^{N,c})(\psi)].
$$
$\check{\Phi}_p^I(\eta_{p-1}^{N,f}\otimes \eta_{p-1}^{N,c})(\varphi\otimes\psi)$ is a bounded random quantity and by Propositions \ref{prop:lp_bound_ind} and \ref{prop:prop_ind} it converges in probability to
$\check{\eta}_p^I(\varphi\otimes \psi)$ hence  by \cite[Theorem 25.12]{bill}, 
$$
\lim_{N\rightarrow+\infty}\mathbb{E}[\check{\Phi}_p^I(\eta_{p-1}^{N,f}\otimes \eta_{p-1}^{N,c})(\varphi\otimes\psi)] = \check{\eta}_p^I(\varphi\otimes \psi).
$$
Similarly via Proposition \ref{prop:lp_bound_ind} and \cite[Theorem 25.12]{bill}
$$
\lim_{N\rightarrow+\infty}\mathbb{E}[\Phi_p^f(\eta_{p-1}^{N,f})(\varphi)\Phi_p^c(\eta_{p-1}^{N,c})(\psi)] = \eta_p^f(\varphi)\eta_p^c(\psi)
$$
and hence we can conclude the proof of (i).

\textbf{Proof for (ii):}  Using similar calculations as for (i) we have
$$
\mathbb{E}[V_p^{N,f}(\varphi)V_p^{N,c}(\psi)] = \mathbb{E}[\check{\Phi}_p^M(\check{\eta}_{p-1}^{N,M})(\varphi\otimes\psi)] - \mathbb{E}[\Phi_p^f(\eta_{p-1}^{N,f})(\varphi)\Phi_p^c(\eta_{p-1}^{N,c})(\psi)].
$$
The proof of Theorem \ref{theo:wlln_max} establishes that
$$
\check{\Phi}_p^M(\check{\eta}_{p-1}^{N,M})(\varphi\otimes\psi) \rightarrow_{\mathbb{P}} \check{\eta}_p^M(\varphi\otimes \psi).
$$
The proof can then be completed in much the same way as for (i).

\textbf{Proof for (iii):}  Using similar calculations as for (i) we have
$$
\mathbb{E}[V_p^{N,f}(\varphi)V_p^{N,c}(\psi)] = \mathbb{E}[\check{\Phi}_p^C(\eta_{p-1}^{N,f},\eta_{p-1}^{N,c})(\varphi\otimes\psi)] - \mathbb{E}[\Phi_p^f(\eta_{p-1}^{N,f})(\varphi)\Phi_p^c(\eta_{p-1}^{N,c})(\psi)].
$$
The proof of Theorem \ref{theo:wlln_mcpf} establishes that
$$
\check{\Phi}_p^C(\eta_{p-1}^{N,f},\eta_{p-1}^{N,c})(\varphi\otimes\psi) \rightarrow_{\mathbb{P}} \check{\eta}_p^C(\varphi\otimes \psi).
$$
The proof can then be completed in much the same way as for (i).
\end{proof}

Define for $p\geq 0$, $\varphi,\psi\in\mathcal{B}_b(\mathsf{X})$.
$$
V_p^N(\varphi,\psi) = V_p^{N,f}(\varphi) - V_p^{N,c}(\psi).
$$

\begin{prop}\label{prop:field_ind}
We have:
\begin{enumerate}
\item{For the IRCPF:
Let $n\geq 0$, then for any $(\varphi_0,\dots,\varphi_n) \in\mathcal{B}_b(\mathsf{X})^{n+1}$, $(\psi_0,\dots,\psi_n) \in\mathcal{B}_b(\mathsf{X})^{n+1}$, 
 $(V_0^N(\varphi_0,\psi_0),\dots,$ $V_n^N(\varphi_n,\psi_n))$ converges in distribution to a $(n+1)-$dimensional Gaussian random variable with zero mean and
diagonal covariance matrix, with the $p\in\{0,\dots,n\}$ diagonal entry
$$
\check{\eta}_p^I(\{(\varphi_p\otimes 1- 1\otimes \psi_p) - \check{\eta}_p^I(\varphi_p\otimes 1- 1\otimes \psi_p)\}^2).
$$
}
\item{For the MCRPF: Assume (A\ref{hyp:1}-\ref{hyp:3}), then for
$n\geq 0$ and any $(\varphi_0,\dots,\varphi_n) \in\mathcal{C}_b(\mathsf{X})^{n+1}$, $(\psi_0,\dots,\psi_n) \in\mathcal{C}_b(\mathsf{X})^{n+1}$, 
 $(V_0^N(\varphi_0,\psi_0),\dots,V_n^N(\varphi_n,\psi_n))$ converges in distribution to a $(n+1)-$dimensional Gaussian random variable with zero mean and
diagonal covariance matrix, with the $p\in\{0,\dots,n\}$ diagonal entry
$$
\check{\eta}_p^M(\{(\varphi_p\otimes 1- 1\otimes \psi_p) - \check{\eta}_p^M(\varphi_p\otimes 1- 1\otimes \psi_p)\}^2).
$$
}
\item{For the MCPF: Suppose that for $s\in\{f,c\}$, %$\eta_0^s\in\mathcal{B}_b(\mathsf{X})$ and
$n\geq 1$, $M_n^s\in\mathcal{B}_b(\mathsf{X}\times\mathsf{X})$, then for
$n\geq 0$ and any $(\varphi_0,\dots,\varphi_n) \in\mathcal{B}_b(\mathsf{X})^{n+1}$, $(\psi_0,\dots,\psi_n) \in\mathcal{B}_b(\mathsf{X})^{n+1}$, 
 $(V_0^N(\varphi_0,\psi_0),\dots,V_n^N(\varphi_n,\psi_n))$ converges in distribution to a $(n+1)-$ dimensional Gaussian random variable with zero mean and
diagonal covariance matrix, with the $p\in\{0,\dots,n\}$ diagonal entry
$$
\check{\eta}_p^C(\{(\varphi_p\otimes 1- 1\otimes \psi_p) - \check{\eta}_p^C(\varphi_p\otimes 1- 1\otimes \psi_p)\}^2).
$$
}

\end{enumerate}
\end{prop}

\begin{proof}
This follows by using almost the same exposition and proofs as \cite{delm:04}, pp.~293-294, Theorem 9.3.1 and Corollary 9.3.1 and the results (of this paper) Lemmata \ref{lem:ind1}-\ref{lem:ind2}. The proof is thus omitted.
\end{proof}

\section{Technical Results for the IRCPF}

\begin{prop}\label{prop:lp_bound_ind}
For any $n\geq 0$, $s\in\{f,c\}$, $p\geq 1$ there exists a $C<+\infty$
such that for any $\varphi\in\mathcal{B}_b(\mathsf{X})$, $N\geq 1$ we have
$$
\mathbb{E}[|[\eta_n^{N,s}-\eta_n^{s}](\varphi)|^{p}]^{1/p} \leq \frac{C\|\varphi\|}{\sqrt{N}}.
$$
\end{prop}

\begin{proof}
This can be proved easily by induction. For instance using the strategy in \cite[Proposition C.6]{mlpf} and is hence omitted.
\end{proof}

%Below, we write $\rightarrow_{\mathbb{P}}$ as convergence in probability as $N\rightarrow\infty$.

\begin{prop}\label{prop:prop_ind}
For any $\varphi\in\mathcal{B}_b(\mathsf{X}\times\mathsf{X})$, $n\geq 0$ we have
$$
(\eta_{n}^{N,f}\otimes\eta_{n}^{N,c})(\varphi) \rightarrow_{\mathbb{P}} (\eta_{n}^{f}\otimes\eta_{n}^{c})(\varphi).
$$
\end{prop}

\begin{proof}
Our proof is by induction, consider the case $n=0$ and let $\epsilon > 0$ be arbitrary, then
$$
\mathbb{P}(|(\eta_{n}^{N,f}\otimes\eta_{n}^{N,c})-(\eta_{n}^{f}\otimes\eta_{n}^{c})](\varphi)|>\epsilon) \leq 
$$
$$
\mathbb{P}(|[(\eta_{n}^{N,f}\otimes\eta_{n}^{N,c})(\varphi) - \mathbb{E}[(\eta_{n}^{N,f}\otimes\eta_{n}^{N,c})(\varphi)]|>\epsilon/2) + 
\mathbb{P}(|\mathbb{E}[(\eta_{n}^{N,f}\otimes\eta_{n}^{N,c})(\varphi)] - (\eta_{n}^{f}\otimes\eta_{n}^{c})(\varphi)|>\epsilon/2) \leq 
$$
$$
\mathbb{P}(|[(\eta_{n}^{N,f}\otimes\eta_{n}^{N,c})(\varphi) - \mathbb{E}[(\eta_{n}^{N,f}\otimes\eta_{n}^{N,c})(\varphi)]|>\epsilon/2)  
+ \mathcal{O}(N^{-1})
$$
where the last line follows as 
$$
\mathbb{E}[(\eta_{n}^{N,f}\otimes\eta_{n}^{N,c})(\varphi)] = \frac{N-1}{N}(\eta_{n}^{f}\otimes\eta_{n}^{c})(\varphi) + \frac{1}{N}\check{\eta}_n^I(\varphi).
$$
To deal with the term $\mathbb{P}(|[(\eta_{n}^{N,f}\otimes\eta_{n}^{N,c})(\varphi) - \mathbb{E}[(\eta_{n}^{N,f}\otimes\eta_{n}^{N,c})(\varphi)]|>\epsilon/2) $
we will apply the bounded difference inequality (see \cite{bound_diff}). To this end, first note that $(u_0^1,\dots,u_0^N)$ are i.i.d.~and setting 
$f(u_0^1,\dots,u_0^N) = (\eta_{n}^{N,f}\otimes\eta_{n}^{N,c})(\varphi)$ we note that for any $1\leq k \leq N$ (the notational convention is clear if $k=1$ or $k=N$)
and any $(u_0^1,\dots,u_0^N)\in(\mathsf{X}\times\mathsf{X})^{n+1}$, $\tilde{u}_0^k\in\mathsf{X}\times\mathsf{X}$
$$
|f(u_0^1,\dots,u_0^N)-f(u_0^1,\dots,u_0^{k-1},\tilde{u}_0^k,u_0^{k+1},\dots,u_0^N)| = \frac{1}{N^2}|[\varphi(u_0^k)-\varphi(\tilde{u}_0^k)]
+
$$
$$
\sum_{i\in\{1,\dots,N\}\setminus\{k\}}[\varphi(x_{0}^{i,f},x_{0}^{k,c})-\varphi(x_{0}^{i,f},\tilde{x}_{0}^{k,c})] +
\sum_{i\in\{1,\dots,N\}\setminus\{k\}}[\varphi(x_{0}^{k,f},x_{0}^{i,c})-\varphi(\tilde{x}_{0}^{k,f},x_{0}^{i,c})]
| \leq 
$$
\begin{equation}\label{eq:prf1}
\frac{4\|\varphi\|}{N}.
\end{equation}
Thus, by the bounded difference inequality:
$$
\mathbb{P}(|[(\eta_{n}^{N,f}\otimes\eta_{n}^{N,c})(\varphi) - \mathbb{E}[(\eta_{n}^{N,f}\otimes\eta_{n}^{N,c})(\varphi)]|>\epsilon/2)  \leq 2\exp\Big\{-\frac{\epsilon^2 N}{16\|\varphi\|^2}\Big\}.
$$
Hence we can easily conclude the result when $n=0$ as $\epsilon>0$ was arbitrary.

Now we assume the result for $n-1$ and consider $n$, we have
$$
\mathbb{P}(|[(\eta_{n}^{N,f}\otimes\eta_{n}^{N,c})-(\eta_{n}^{f}\otimes\eta_{n}^{c})](\varphi)|>\epsilon) \leq 
\mathbb{P}(|[(\eta_{n}^{N,f}\otimes\eta_{n}^{N,c})(\varphi) - \mathbb{E}[(\eta_{n}^{N,f}\otimes\eta_{n}^{N,c})(\varphi)|\mathcal{F}_{n-1}^N]|>\epsilon/2) +
$$
$$
\mathbb{P}(|\mathbb{E}[(\eta_{n}^{N,f}\otimes\eta_{n}^{N,c})(\varphi)|\mathcal{F}_{n-1}^N] - (\eta_{n}^{f}\otimes\eta_{n}^{c})(\varphi)|>\epsilon/2) 
$$
where $\mathcal{F}_{n-1}^N$ is the $\sigma-$algebra generated by the particle system up-to time $n-1$. For the term
$\mathbb{P}(|[(\eta_{n}^{N,f}\otimes\eta_{n}^{N,c})(\varphi) - \mathbb{E}[(\eta_{n}^{N,f}\otimes\eta_{n}^{N,c})(\varphi)|\mathcal{F}_{n-1}^N]|>\epsilon/2)$, on conditioning upon $\mathcal{F}_{n-1}^N$
one can apply the above bounded difference inequality, as $(u_n^1,\dots,u_n^N)$ are (conditionally) i.i.d.~and that the bound in \eqref{eq:prf1} can also be obtained for any $n$. Hence one has
$$
\mathbb{P}(|[(\eta_{n}^{N,f}\otimes\eta_{n}^{N,c})(\varphi) - \mathbb{E}[(\eta_{n}^{N,f}\otimes\eta_{n}^{N,c})(\varphi)|\mathcal{F}_{n-1}^N]|>\epsilon/2|\mathcal{F}_{n-1}) \leq 2\exp\Big\{-\frac{\epsilon^2 N}{16\|\varphi\|^2}\Big\}
$$
and thus one can conclude the result  if 
$$
\mathbb{E}[(\eta_{n}^{N,f}\otimes\eta_{n}^{N,c})(\varphi)|\mathcal{F}_{n-1}^N] \rightarrow_{\mathbb{P}} (\eta_{n}^{f}\otimes\eta_{n}^{c})(\varphi).
$$
Now
\begin{eqnarray*}
\mathbb{E}[(\eta_{n}^{N,f}\otimes\eta_{n}^{N,c})(\varphi)|\mathcal{F}_{n-1}^N] & = & \frac{1}{N^2}\Big(
\sum_{i=1}^N \mathbb{E}[\varphi(U_n^i)|\mathcal{F}_{n-1}^N] + 
\sum_{i\neq j}\mathbb{E}[\varphi(X_n^{i,f},X_n^{i,c})|\mathcal{F}_{n-1}^N]
\Big) \\
& = & \frac{1}{N}\check{\Phi}_n^I(\eta_{n-1}^{N,f}\otimes\eta_{n-1}^{N,c})(\varphi) + \frac{(N-1)}{N}(\Phi_n^f(\eta_{n-1}^{N,f})\otimes\Phi_n^f(\eta_{n-1}^{N,f}))(\varphi).
\end{eqnarray*}
Now by the induction hypothesis and Proposition \ref{prop:lp_bound_ind} $\check{\Phi}_n^I(\eta_{n-1}^{N,f}\otimes\eta_{n-1}^{N,c})(\varphi)$ converges in probability to
$$
\frac{(\eta_{n-1}^{f}\otimes\eta_{n-1}^{c})((G_{n-1}\otimes G_{n-1})\check{M}_n(\varphi))}{\eta_{n-1}^{f}(G_{n-1})\eta_{n-1}^{c}(G_{n-1})}
$$
hence the term $\frac{1}{N}\check{\Phi}_n^I(\eta_{n-1}^{N,f}\otimes\eta_{n-1}^{N,c})(\varphi)$ goes to zero. Then, again by the induction hypothesis and Proposition \ref{prop:lp_bound_ind}
$(\Phi_n^f(\eta_{n-1}^{N,f})\otimes\Phi_n^f(\eta_{n-1}^{N,f}))(\varphi)$ converges in probability to 
$$
\frac{(\eta_{n-1}^{f}\otimes\eta_{n-1}^{c})((G_{n-1}\otimes G_{n-1})(M_n^f\otimes M_n^c)(\varphi))}{\eta_{n-1}^{f}(G_{n-1})\eta_{n-1}^{c}(G_{n-1})} = (\eta_{n}^{f}\otimes\eta_{n}^{c})(\varphi)
$$
and hence we conclude the result.
\end{proof}

\section{Technical Results for the MCRPF}\label{app:max_wlln}

\begin{proof}[Proof of Theorem \ref{theo:wlln_max}]
The proof is by induction. The case $n=0$ follows by the weak law of large numbers for i.i.d.~random variables, so we assume the result at time $n-1$. We have
$$
\check{\eta}_{n}^{N,M}(\varphi) - \check{\eta}_{n}^{M}(\varphi) = 
\check{\eta}_{n}^{N,M}(\varphi) - \check{\Phi}_n^M(\check{\eta}_{n-1}^{N,M})(\varphi) + 
\check{\Phi}_n^M(\check{\eta}_{n-1}^{N,M})(\varphi) - \check{\Phi}_n^M(\check{\eta}_{n-1}^{M})(\varphi).
$$
One can easily prove that
$$
|\check{\eta}_{n}^{N,M}(\varphi) - \check{\Phi}_n^M(\check{\eta}_{n-1}^{N,M})(\varphi)| \rightarrow_{\mathbb{P}} 0
$$
by using the (conditional) Marcinkiewicz-Zygmund inequality, so we focus on the latter term.
Define
\begin{eqnarray*}
T_1^N & :=& \check{\eta}_{n-1}^{N,M}\Big(\{F_{n-1,\check{\eta}_{n-1}^{N,M},f} \wedge F_{n-1,\check{\eta}_{n-1}^{N,M},c}\}\check{M}_n(\varphi)\Big) - 
\check{\eta}_{n-1}^{M}\Big(\{F_{n-1,\check{\eta}_{n-1}^{M},f} \wedge F_{n-1,\check{\eta}_{n-1}^{M},c}\}\check{M}_n(\varphi)\Big) \\
T_2^N & :=& \Big(1-
\check{\eta}_{n-1}^{N,M}\Big(\{F_{n-1,\check{\eta}_{n-1}^{N,M},f} \wedge F_{n-1,\check{\eta}_{n-1}^{N,M},c}\}\Big)
\Big)(\check{\eta}_{n-1}^{N,M}\otimes\check{\eta}_{n-1}^{N,M})\Big(\Big\{\overline{F}_{n-1,\check{\eta}_{n-1}^{N,M},f}\otimes \\ & & \overline{F}_{n-1,\check{\eta}_{n-1}^{N,M},c}\Big\}\bar{M}_{n}(\varphi)\Big) \\
T_3 & :=& \Big(1-
\check{\eta}_{n-1}^{M}\Big(\{F_{n-1,\check{\eta}_{n-1}^{M},f} \wedge F_{n-1,\check{\eta}_{n-1}^{M},c}\}\Big)
\Big)(\check{\eta}_{n-1}^{M}\otimes\check{\eta}_{n-1}^{M})\Big(\Big\{\overline{F}_{n-1,\check{\eta}_{n-1}^{M},f}\otimes \\ & & \overline{F}_{n-1,\check{\eta}_{n-1}^{M},c}\Big\}\bar{M}_{n}(\varphi)\Big)
\end{eqnarray*}
Then
$$
T_1^N + T_2^N - T_3 = \check{\Phi}_n^M(\check{\eta}_{n-1}^{N,M})(\varphi) - \check{\Phi}_n^M(\check{\eta}_{n-1}^{M})(\varphi).
$$
By Lemma \ref{lem:max_bayes_conv} $T_1^N\rightarrow_{\mathbb{P}}0$ and by Lemma \ref{lem:max_ci_hard_part}, $ T_2^N - T_3\rightarrow_{\mathbb{P}}0$. This
concludes the proof.
\end{proof}

\begin{lem}\label{lem:max_bayes_conv}
Assume (A\ref{hyp:1}). Then if  
for any $\varphi\in\mathcal{C}_b(\mathsf{X}\times\mathsf{X})$, $n\geq 0$
$$
\check{\eta}_{n}^{N,M}(\varphi) \rightarrow_{\mathbb{P}} \check{\eta}_{n}^{M}(\varphi)
$$
we have for any $\psi\in\mathcal{C}_b(\mathsf{X}\times\mathsf{X})$, $n\geq 0$
$$
\check{\eta}_{n}^{N,M}\Big(\{F_{n,\check{\eta}_{n}^{N,M},f} \wedge F_{n,\check{\eta}_{n}^{N,M},c}\}\psi\Big)
\rightarrow_{\mathbb{P}}
\check{\eta}_{n}^{M}\Big(\{F_{n,\check{\eta}_{n}^{M},f} \wedge F_{n,\check{\eta}_{n}^{M},c}\}\psi\Big).
$$
\end{lem}

\begin{proof}
Set
\begin{eqnarray}
T_1^N & := & \check{\eta}_{n}^{N,M}\Big(\{F_{n,\check{\eta}_{n}^{N,M},f} \wedge F_{n,\check{\eta}_{n}^{N,M},c}\}\psi-
\{F_{n,\check{\eta}_{n}^{M},f} \wedge F_{n,\check{\eta}_{n}^{M},c}\}\psi\Big) \label{eq:max_new1}\\
T_2^N & := & \Big(\check{\eta}_{n}^{N,M} - \check{\eta}_{n}^{N}\Big)\Big(\{F_{n,\check{\eta}_{n}^{M},f} \wedge F_{n,\check{\eta}_{n}^{M},c}\}\psi\Big).
\label{eq:max_new2}
\end{eqnarray}
Then we have the decomposition
$$
T_1^N + T_2^N = 
\check{\eta}_{n}^{N,M}\Big(\{F_{n,\check{\eta}_{n}^{N,M},f} \wedge F_{n,\check{\eta}_{n}^{N,M},c}\}\psi\Big) - 
\check{\eta}_{n}^{M}\Big(\{F_{n,\check{\eta}_{n}^{M},f} \wedge F_{n,\check{\eta}_{n}^{M},c}\}\psi\Big) = 
$$
We will show that \eqref{eq:max_new1}-\eqref{eq:max_new2} converge in probability to zero.

\textbf{Term \eqref{eq:max_new1}:} Define
\begin{eqnarray}
T_3^N & := &  \frac{1}{2}\check{\eta}_{n}^{N,M}\Big(\{F_{n,\check{\eta}_{n}^{N,M},f}-F_{n,\check{\eta}_{n}^{M},f}\}\psi\Big) \label{eq:max_new3}\\
T_4^N & := & \frac{1}{2}  \check{\eta}_{n}^{N,M}\Big(\{F_{n,\check{\eta}_{n}^{N,M},c}-F_{n,\check{\eta}_{n}^{M},c}\}\psi\Big) \label{eq:max_new4}\\
T_5^N & :=&  \frac{1}{2}\check{\eta}_{n}^{N,M}\Big(\{|F_{n,\check{\eta}_{n}^{M},f}-F_{n,\check{\eta}_{n}^{M},c}| - |F_{n,\check{\eta}_{n}^{N,M},f}-F_{n,\check{\eta}_{n}^{N,M},c}|\}\psi\Big)\Big) \label{eq:max_new5}
\end{eqnarray}
Then we have
$$
T_1^N = T_3^N + T_4^N + T_5^N.
$$
To show that $T_1^N$ converges in probability to zero we will show that \eqref{eq:max_new3}-\eqref{eq:max_new5} each converge in probability to zero.
For \eqref{eq:max_new3} we have
$$
\check{\eta}_{n}^{N,M}\Big(\{F_{n,\check{\eta}_{n}^{N,M},f}-F_{n,\check{\eta}_{n}^{M},f}\}\psi\Big) = 
\frac{\eta_n^{f}(G_n)-\eta_n^{N,f}(G_n)}{\eta_n^{N,f}(G_n)\eta_n^{f}(G_n)}\check{\eta}_{n}^{N,M}((G_n\otimes 1)\psi).
$$ 
By \cite[Proposition C.6]{mlpf}
$$
\frac{\eta_n^{f}(G_n)-\eta_n^{N,f}(G_n)}{\eta_n^{N,f}(G_n)\eta_n^{f}(G_n)} \rightarrow_{\mathbb{P}} 0
$$
and by hypothesis $\check{\eta}_{n}^{N,M}((G_n\otimes 1)\psi)$ converges in probability, hence \eqref{eq:max_new3} converges in probability to zero.
For \eqref{eq:max_new4} this term converges in probability to zero by an almost identical argument to \eqref{eq:max_new3} and is hence omitted.
For \eqref{eq:max_new5}
$$
\check{\eta}_{n}^{N,M}\Big(\{|F_{n,\check{\eta}_{n}^{M},f}-F_{n,\check{\eta}_{n}^{M},c}| - |F_{n,\check{\eta}_{n}^{N,M},f}-F_{n,\check{\eta}_{n}^{N,M},c}|\}\psi\Big) = 
$$
$$
\frac{\eta_n^{N,f}(G_n)\eta_n^{N,c}(G_n)-\eta_n^{f}(G_n)\eta_n^{c}(G_n)}{\eta_n^{f}(G_n)\eta_n^{c}(G_n)\eta_n^{N,f}(G_n)\eta_n^{N,c}(G_n)}
\check{\eta}_{n}^{N,M}(\{|\eta_n^{c}(G_n)(G_n\otimes 1)-\eta_n^{f}(G_n)(1\otimes G_n)|\}\psi) - 
$$
$$
\frac{1}{\eta_n^{N,f}(G_n)\eta_n^{N,c}(G_n)}
\check{\eta}_{n}^{N,M}\Big(\{
|\eta_n^{N,c}(G_n)(G_n\otimes 1)-\eta_n^{N,f}(G_n)(1\otimes G_n)| -
$$
\begin{equation}\label{eq:max_aux2}
|\eta_n^{c}(G_n)(G_n\otimes 1)-\eta_n^{f}(G_n)(1\otimes G_n)|
\}\psi\Big).
\end{equation}
The first term on the R.H.S.~converges in probability to zero by \cite[Proposition C.6]{mlpf} and hypothesis. Hence as
$$
\frac{1}{\eta_n^{N,f}(G_n)\eta_n^{N,c}(G_n)}
$$
converges in probability, we need only show that 
$$
\check{\eta}_{n}^{N,M}\Big(\{
|\eta_n^{N,c}(G_n)(G_n\otimes 1)-\eta_n^{N,f}(G_n)(1\otimes G_n)| -
|\eta_n^{c}(G_n)(G_n\otimes 1)-\eta_n^{f}(G_n)(1\otimes G_n)|
\}\psi\Big)
$$
converges in probability to zero to conclude. Using standard algebra, we have almost surely, that
$$
\Big|\eta_n^{N,c}(G_n)(G_n\otimes 1)-\eta_n^{N,f}(G_n)(1\otimes G_n)| -
|\eta_n^{c}(G_n)(G_n\otimes 1)-\eta_n^{f}(G_n)(1\otimes G_n)|\Big| \leq
$$
\begin{equation}\label{eq:max_aux_3}
|[\eta_n^{N,c}-\eta_n^{c}](G_n)(G_n\otimes 1)| + 
|[\eta_n^{N,f}-\eta_n^{f}](G_n)(1\otimes G_n)|.
\end{equation}
Hence
$$
\check{\eta}_{n}^{N,M}\Big(\{
|\eta_n^{N,c}(G_n)(G_n\otimes 1)-\eta_n^{N,f}(G_n)(1\otimes G_n)| -
|\eta_n^{c}(G_n)(G_n\otimes 1)-\eta_n^{f}(G_n)(1\otimes G_n)|
\}\psi\Big) \leq
$$
$$
|[\eta_n^{N,c}-\eta_n^{c}](G_n)|\check{\eta}_{n}^{N,M}((G_n\otimes 1)|\psi|) + 
|[\eta_n^{N,f}-\eta_n^{f}](G_n)|\check{\eta}_{n}^{N,M}((1\otimes G_n)|\psi|)
$$
and one can easily conclude by the above arguments.

\textbf{Term \eqref{eq:max_new2}:} As
$$
\{F_{n,\check{\eta}_{n}^{M},f} \wedge F_{n,\check{\eta}_{n}^{M},c}\}\psi = 
\frac{1}{2}\Big(F_{n,\check{\eta}_{n}^{M},f} + F_{n,\check{\eta}_{n}^{M},c} - |F_{n,\check{\eta}_{n}^{M},f}-F_{n,\check{\eta}_{n}^{M},c}|\Big)\psi
$$
we have $\{F_{n,\check{\eta}_{n}^{M},f} \wedge F_{n,\check{\eta}_{n}^{M},c}\}\psi\in\mathcal{C}_b(\mathsf{X}\times\mathsf{X})$, so \eqref{eq:max_new2} converges in probability to zero.
\end{proof}

\begin{lem}\label{lem:max_cond_exp_1}
Assume (A\ref{hyp:1}). Then if  
for any $\varphi\in\mathcal{C}_b(\mathsf{X}\times\mathsf{X})$, $n\geq 0$
$$
\check{\eta}_{n}^{N,M}(\varphi) \rightarrow_{\mathbb{P}} \check{\eta}_{n}^{M}(\varphi)
$$
we have for any $\psi\in\mathcal{C}_b(\mathsf{X}\times\mathsf{X})$, $n\geq 0$
$$
(\check{\eta}_{n}^{N,M}\otimes\check{\eta}_{n}^{N,M})\Big(
\Big\{
\frac{F_{n,\check{\eta}_{n}^{N,M},f}}{\check{\eta}_{n}^{N,M}(F_{n,\check{\eta}_{n}^{N,M},f}-F_{n,\check{\eta}_{n}^{N,M},f}\wedge F_{n,\check{\eta}_{n}^{N,M},c})}
\otimes
\overline{F}_{n,\check{\eta}_{n}^{N,M},c}\Big\}
\bar{M}_{n+1}(\psi)
\Big)
$$
$$
-(\check{\eta}_{n}^{N,M}\otimes\check{\eta}_{n}^{N,M})\Big(
\Big\{
\frac{F_{n,\check{\eta}_{n}^{M},f}}{\check{\eta}_{n}^{M}(F_{n,\check{\eta}_{n}^{M},f}-F_{n,\check{\eta}_{n}^{M},f}\wedge F_{n,\check{\eta}_{n}^{M},c})}
\otimes
\overline{F}_{n,\check{\eta}_{n}^{N,M},c}\Big\}
\bar{M}_{n+1}(\psi)
\Big) 
\rightarrow_{\mathbb{P}} 0.
$$
\end{lem}

\begin{proof}
Define
\begin{eqnarray}
T_1^N & := & 
(\check{\eta}_{n}^{N,M}\otimes\check{\eta}_{n}^{N,M})\Big(
\Big\{
\frac{F_{n,\check{\eta}_{n}^{N,M},f}-F_{n,\check{\eta}_{n}^{M},f}}{\check{\eta}_{n}^{N,M}(F_{n,\check{\eta}_{n}^{N,M},f}-F_{n,\check{\eta}_{n}^{N,M},f}\wedge F_{n,\check{\eta}_{n}^{N,M},c})}
\otimes \nonumber \\ & & 
\overline{F}_{n,\check{\eta}_{n}^{N,M},c}\Big\}
\bar{M}_{n+1}(\psi)
\Big)\label{eq:max_6}\\
T_2^N & := & \frac{\check{\eta}_{n}^{M}(F_{n,\check{\eta}_{n}^{M},f}-F_{n,\check{\eta}_{n}^{M},f}\wedge F_{n,\check{\eta}_{n}^{M},c})-\check{\eta}_{n}^{N,M}(F_{n,\check{\eta}_{n}^{N,M},f}-F_{n,\check{\eta}_{n}^{N,M},f}\wedge F_{n,\check{\eta}_{n}^{N,M},c})}
{\check{\eta}_{n}^{M}(F_{n,\check{\eta}_{n}^{M},f}-F_{n,\check{\eta}_{n}^{M},f}\wedge F_{n,\check{\eta}_{n}^{M},c})\check{\eta}_{n}^{N,M}(F_{n,\check{\eta}_{n}^{N,M},f}-F_{n,\check{\eta}_{n}^{N,M},f}\wedge F_{n,\check{\eta}_{n}^{N,M},c})}\times \nonumber \\ & &
(\check{\eta}_{n}^{N,M}\otimes\check{\eta}_{n}^{N,M})\Big(\Big\{F_{n,\check{\eta}_{n}^{M},f}
\otimes
\overline{F}_{n,\check{\eta}_{n}^{N,M},c}\Big\}
\bar{M}_{n+1}(\psi)\Big)\label{eq:max_7}
\end{eqnarray}
then we have
$$
T_1^N + T_2^N = (\check{\eta}_{n}^{N,M}\otimes\check{\eta}_{n}^{N,M})\Big(
\Big\{
\frac{F_{n,\check{\eta}_{n}^{N,M},f}}{\check{\eta}_{n}^{N,M}(F_{n,\check{\eta}_{n}^{N,M},f}-F_{n,\check{\eta}_{n}^{N,M},f}\wedge F_{n,\check{\eta}_{n}^{N,M},c})}
\otimes
\overline{F}_{n,\check{\eta}_{n}^{N,M},c}\Big\}
\bar{M}_{n+1}(\psi)
\Big) -
$$
$$
-(\check{\eta}_{n}^{N,M}\otimes\check{\eta}_{n}^{N,M})\Big(
\Big\{
\frac{F_{n,\check{\eta}_{n}^{M},f}}{\check{\eta}_{n}^{M}(F_{n,\check{\eta}_{n}^{M},f}-F_{n,\check{\eta}_{n}^{M},f}\wedge F_{n,\check{\eta}_{n}^{M},c})}
\otimes
\overline{F}_{n,\check{\eta}_{n}^{N,M},c}\Big\}
\bar{M}_{n+1}(\psi)
\Big).
$$
We will show that \eqref{eq:max_6} and \eqref{eq:max_7} will converge in probability to zero.

\textbf{Term \eqref{eq:max_6}}:
For \eqref{eq:max_6} we note that
$$
\check{\eta}_{n}^{N,M}(F_{n,\check{\eta}_{n}^{N,M},f}-F_{n,\check{\eta}_{n}^{N,M},f}\wedge F_{n,\check{\eta}_{n}^{N,M},c}) =
1 - \check{\eta}_{n}^{N,M}(F_{n,\check{\eta}_{n}^{N,M},f}\wedge F_{n,\check{\eta}_{n}^{N,M},c}).
$$
By Lemma \ref{lem:max_bayes_conv} this converges in probability to a constant, so we only consider the convergence to zero of
$$
(\check{\eta}_{n}^{N,M}\otimes\check{\eta}_{n}^{N,M})\Big(
\{(F_{n,\check{\eta}_{n}^{N,M},f}-F_{n,\check{\eta}_{n}^{M},f})\otimes \overline{F}_{n,\check{\eta}_{n}^{N,M},c}\}
\bar{M}_{n+1}(\psi)
\Big) = 
$$
$$
\Big(\frac{1}{\eta_n^{N,f}(G_n)}-\frac{1}{\eta_n^{f}(G_n)}\Big)
(\check{\eta}_{n}^{N,M}\otimes\check{\eta}_{n}^{N,M})\Big(\{
(G_n\otimes 1)
\otimes \overline{F}_{n,\check{\eta}_{n}^{N,M},c}\}
\bar{M}_{n+1}(\psi)
\Big).
$$
By \cite[Proposition C.6]{mlpf}
$$
\Big(\frac{1}{\eta_n^{N,f}(G_n)}-\frac{1}{\eta_n^{f}(G_n)}\Big) \rightarrow_{\mathbb{P}} 0
$$
so if we can show that the remaining term on the R.H.S.~is, in absolute value, (almost surely) upper-bounded by a convergent (in probability) random variable, we have shown that
\eqref{eq:max_6} converges in probability to zero. We have
\begin{equation}\label{eq:max_aux1}
|(\check{\eta}_{n}^{N,M}\otimes\check{\eta}_{n}^{N,M})\Big(\{
(G_n\otimes 1)
\otimes \overline{F}_{n,\check{\eta}_{n}^{N,M},c}\}
\bar{M}_{n+1}(\psi)
\Big)| \leq \|\psi\|\|G_n\|\check{\eta}_{n}^{N,M}(|\overline{F}_{n,\check{\eta}_{n}^{N,M},c}|).
\end{equation}
Then almost surely
\begin{equation}\label{eq:max_8}
\check{\eta}_{n}^{N,M}(|\overline{F}_{n,\check{\eta}_{n}^{N,M},c}|) \leq 
\frac{\check{\eta}_{n}^{N,M}(F_{n,\check{\eta}_{n}^{N,M},c})+\check{\eta}_{n}^{N,M}(\{F_{n,\check{\eta}_{n}^{N,M},f}\wedge F_{n,\check{\eta}_{n}^{N,M},c}\})}
{|\check{\eta}_{n}^{N,M}(F_{n,\check{\eta}_{n}^{N,M},c}-\{F_{n,\check{\eta}_{n}^{N,M},f}\wedge F_{n,\check{\eta}_{n}^{N,M},c}\})|}
\end{equation}
By the above arguments, both the denominator and numerator will converge in probability to a finite constant and hence we have shown that \eqref{eq:max_6} converges in probability to zero.

\textbf{Term \eqref{eq:max_7}}: We note
$$
\frac{\check{\eta}_{n}^{M}(F_{n,\check{\eta}_{n}^{M},f}-F_{n,\check{\eta}_{n}^{M},f}\wedge F_{n,\check{\eta}_{n}^{M},c})-\check{\eta}_{n}^{N,M}(F_{n,\check{\eta}_{n}^{N,M},f}-F_{n,\check{\eta}_{n}^{N,M},f}\wedge F_{n,\check{\eta}_{n}^{N,M},c})}
{\check{\eta}_{n}^{M}(F_{n,\check{\eta}_{n}^{M},f}-F_{n,\check{\eta}_{n}^{M},f}\wedge F_{n,\check{\eta}_{n}^{M},c})\check{\eta}_{n}^{N,M}(F_{n,\check{\eta}_{n}^{N,M},f}-F_{n,\check{\eta}_{n}^{N,M},f}\wedge F_{n,\check{\eta}_{n}^{N,M},c})} =
$$
$$
\frac{\check{\eta}_{n}^{N,M}(F_{n,\check{\eta}_{n}^{N,M},f}\wedge F_{n,\check{\eta}_{n}^{N,M},c})-
\check{\eta}_{n}^{M}(F_{n,\check{\eta}_{n}^{M},f}\wedge F_{n,\check{\eta}_{n}^{M},c})
}
{(1-\check{\eta}_{n}^{N,M}(F_{n,\check{\eta}_{n}^{N,M},f}\wedge F_{n,\check{\eta}_{n}^{N,M},c}))(1-\check{\eta}_{n}^{M}(F_{n,\check{\eta}_{n}^{M},f}\wedge F_{n,\check{\eta}_{n}^{M},c}))}.
$$
By Lemma \ref{lem:max_bayes_conv} this converges in probability to zero. Thus, we need only show that $|(\check{\eta}_{n}^{N,M}\otimes\check{\eta}_{n}^{N,M})\Big(\Big\{F_{n,\check{\eta}_{n}^{M},f} \otimes \overline{F}_{n,\check{\eta}_{n}^{N,M},c}\Big\} \bar{M}_{n+1}(\psi)\Big)|$ is (almost surely) upper-bounded by a convergent (in probability) random variable.
Almost surely,
$$
|(\check{\eta}_{n}^{N,M}\otimes\check{\eta}_{n}^{N,M})\Big(\Big\{F_{n,\check{\eta}_{n}^{M},f} \otimes \overline{F}_{n,\check{\eta}_{n}^{N,M},c}\Big\} \bar{M}_{n+1}(\psi)\Big)| \leq
\|\psi\|\check{\eta}_{n}^{N,M}(F_{n,\check{\eta}_{n}^{M},f})\check{\eta}_{n}^{N,M}(|\overline{F}_{n,\check{\eta}_{n}^{N,M},c}|).
$$
As $F_{n,\check{\eta}_{n}^{M},f}\in\mathcal{C}_b(\mathsf{X}\times\mathsf{X})$, $\check{\eta}_{n}^{N,M}(F_{n,\check{\eta}_{n}^{M},f})$ converges in probability to a finite constant
and our proof is concluded by the argument associated to \eqref{eq:max_8}.
\end{proof}

\begin{lem}\label{lem:max_cond_exp_2}
Assume (A\ref{hyp:1}). Then if  
for any $\varphi\in\mathcal{C}_b(\mathsf{X}\times\mathsf{X})$, $n\geq 0$
$$
\check{\eta}_{n}^{N,M}(\varphi) \rightarrow_{\mathbb{P}} \check{\eta}_{n}^{M}(\varphi)
$$
we have for any $\psi\in\mathcal{C}_b(\mathsf{X}\times\mathsf{X})$, $n\geq 0$
$$
(\check{\eta}_{n}^{N,M}\otimes\check{\eta}_{n}^{N,M})\Big(\Big(\Big\{
\frac{F_{n,\check{\eta}_{n}^{M},f}\wedge F_{n,\check{\eta}_{n}^{M},c}-
F_{n,\check{\eta}_{n}^{N,M},f}\wedge F_{n,\check{\eta}_{n}^{N,M},c}}
{\check{\eta}_{n}^{M}(F_{n,\check{\eta}_{n}^{M},f}-F_{n,\check{\eta}_{n}^{M},f}\wedge F_{n,\check{\eta}_{n}^{M},c})}
\Big\}\otimes
\overline{F}_{n,\check{\eta}_{n}^{N,M},c}\Big)
\bar{M}_{n+1}(\psi)
\Big)\rightarrow_{\mathbb{P}}0.
$$
\end{lem}

\begin{proof}
Define
\begin{eqnarray}
C & := & \frac{1}{\check{\eta}_{n}^{M}(F_{n,\check{\eta}_{n}^{M},f}-F_{n,\check{\eta}_{n}^{M},f}\wedge F_{n,\check{\eta}_{n}^{M},c})} \nonumber\\
T_1^N & := & -\frac{C}{2}(\check{\eta}_{n}^{N,M}\otimes\check{\eta}_{n}^{N,M})\Big(\Big(\{F_{n,\check{\eta}_{n}^{N,M},f}-F_{n,\check{\eta}_{n}^{M},f}\}\otimes \overline{F}_{n,\check{\eta}_{n}^{N,M},c}\Big)
\bar{M}_{n+1}(\psi)\Big) \label{eq:max_9}\\
T_2^N & := & -\frac{C}{2}(\check{\eta}_{n}^{N,M}\otimes\check{\eta}_{n}^{N,M})\Big(\Big(\{F_{n,\check{\eta}_{n}^{N,M},c}-F_{n,\check{\eta}_{n}^{M},c}\}\otimes \overline{F}_{n,\check{\eta}_{n}^{N,M},c}\Big)
\bar{M}_{n+1}(\psi)\Big) \label{eq:max_10}\\
T_3^N & := & -\frac{C}{2}(\check{\eta}_{n}^{N,M}\otimes\check{\eta}_{n}^{N,M})\Big(\Big(\{
|F_{n,\check{\eta}_{n}^{M},f}-F_{n,\check{\eta}_{n}^{M},c}| -
|F_{n,\check{\eta}_{n}^{N,M},f}-F_{n,\check{\eta}_{n}^{N,M},c}|
\}
\otimes \nonumber \\ & &  \overline{F}_{n,\check{\eta}_{n}^{N,M},c}\Big)
\bar{M}_{n+1}(\psi)\Big). \label{eq:max_11}
\end{eqnarray}
Then we have
$$
T_1^N + T_2^N + T_3^N = 
$$
$$
(\check{\eta}_{n}^{N,M}\otimes\check{\eta}_{n}^{N,M})\Big(\Big(\Big\{
\frac{F_{n,\check{\eta}_{n}^{M},f}\wedge F_{n,\check{\eta}_{n}^{M},c}-
F_{n,\check{\eta}_{n}^{N,M},f}\wedge F_{n,\check{\eta}_{n}^{N,M},c}}
{\check{\eta}_{n}^{M}(F_{n,\check{\eta}_{n}^{M},f}-F_{n,\check{\eta}_{n}^{M},f}\wedge F_{n,\check{\eta}_{n}^{M},c})}
\Big\}\otimes
\overline{F}_{n,\check{\eta}_{n}^{N,M},c}\Big)
\bar{M}_{n+1}(\psi)
\Big)
$$
We will show that \eqref{eq:max_9},  \eqref{eq:max_10} and \eqref{eq:max_11} will converge in probability to zero to conclude the proof.

\textbf{Term \eqref{eq:max_9}}: We have that \eqref{eq:max_9} is equal to
$$
-\frac{C}{2} \frac{\eta_n^{N,f}(G_n)-\eta_n^{f}(G_n)}{\eta_n^{N,f}(G_n)\eta_n^{f}(G_n)}
(\check{\eta}_{n}^{N,M}\otimes\check{\eta}_{n}^{N,M})\Big(
\Big(\{G_n\otimes 1\}
\otimes \overline{F}_{n,\check{\eta}_{n}^{N,M},c}\Big)\bar{M}_{n+1}(\psi)
\Big).
$$
By \cite[Proposition C.6]{mlpf}
$$
\frac{\eta_n^{N,f}(G_n)-\eta_n^{f}(G_n)}{\eta_n^{N,f}(G_n)\eta_n^{f}(G_n)} \rightarrow_{\mathbb{P}} 0
$$
so if we can show that the remaining term on the R.H.S.~is, in absolute value, (almost surely) upper-bounded by a convergent (in probability) random variable, we have shown that
\eqref{eq:max_9} converges in probability to zero. This can be achieved using the argument for \eqref{eq:max_aux1} in the proof of Lemma \ref{lem:max_cond_exp_1} and hence we have
shown that \eqref{eq:max_9} converges in probability to zero. 

\textbf{Term \eqref{eq:max_10}}:  The argument is almost identical to \eqref{eq:max_9} and is omitted.

\textbf{Term \eqref{eq:max_11}}: Using a decomposition similar to \eqref{eq:max_aux2}, we define:
\begin{eqnarray}
T_4^N & := & -\frac{C}{2}
\frac{\eta_n^{f}(G_n)\eta_n^{c}(G_n)-\eta_n^{N,f}(G_n)\eta_n^{N,c}(G_n)}{\eta_n^{f}(G_n)\eta_n^{c}(G_n)\eta_n^{N,f}(G_n)\eta_n^{N,c}(G_n)}
(\check{\eta}_{n}^{N,M}\otimes\check{\eta}_{n}^{N,M})\Big(
\nonumber \\ & &
\Big(\{
|\eta_n^{c}(G_n)(G_n\otimes 1)-\eta_n^{f}(G_n)(1\otimes G_n)|
\}\otimes\overline{F}_{n,\check{\eta}_{n}^{N,M},c}\Big)
\bar{M}_{n+1}(\psi)\Big) \label{eq:max_12}\\
T_5^N & := & -\frac{C}{2}\frac{1}{\eta_n^{N,f}(G_n)\eta_n^{N,c}(G_n)}
(\check{\eta}_{n}^{N,M}\otimes\check{\eta}_{n}^{N,M})\Big(
\Big(\{
|\eta_n^{N,c}(G_n)(G_n\otimes 1)-\eta_n^{N,f}(G_n)(1\otimes G_n)| - \nonumber\\& &
|\eta_n^{c}(G_n)(G_n\otimes 1)-\eta_n^{f}(G_n)(1\otimes G_n)|
\}\otimes\overline{F}_{n,\check{\eta}_{n}^{N,M},c}\Big)
\bar{M}_{n+1}(\psi)\Big)\label{eq:max_13}
\end{eqnarray}
then we have
$$
T_3^N = -T_4^N - T_5^N.
$$
We will show that \eqref{eq:max_12} and \eqref{eq:max_13} will converge to zero in probability. For
\eqref{eq:max_12}, by \cite[Proposition C.6]{mlpf}
$$
\frac{\eta_n^{f}(G_n)\eta_n^{c}(G_n)-\eta_n^{N,f}(G_n)\eta_n^{N,c}(G_n)}{\eta_n^{f}(G_n)\eta_n^{c}(G_n)\eta_n^{N,f}(G_n)\eta_n^{N,c}(G_n)} \rightarrow_{\mathbb{P}} 0
$$
and 
$$
\Big|(\check{\eta}_{n}^{N,M}\otimes\check{\eta}_{n}^{N,M})\Big(
\Big(\{
|\eta_n^{c}(G_n)(G_n\otimes 1)-\eta_n^{f}(G_n)(1\otimes G_n)|
\}\otimes\overline{F}_{n,\check{\eta}_{n}^{N,M},c}\Big)
\bar{M}_{n+1}(\psi)\Big)
\Big| \leq
$$
$$
\|G_n\|\|\psi\|(\eta_n^{c}(G_n)+\eta_n^{f}(G_n))\check{\eta}_{n}^{N,M}(|\overline{F}_{n,\check{\eta}_{n}^{N,M},c}|)
$$
thus, using the argument for \eqref{eq:max_aux1} in the proof of Lemma \ref{lem:max_cond_exp_1} we have
shown that \eqref{eq:max_12} converges in probability to zero. For  \eqref{eq:max_13} as 
$(\eta_n^{N,f}(G_n)\eta_n^{N,c}(G_n))^{-1}$ converges in probability to finite constant, we need only show that the remaining
term converges in probability to zero. Using \eqref{eq:max_aux_3} we have that
$$
\Big|
(\check{\eta}_{n}^{N,M}\otimes\check{\eta}_{n}^{N,M})\Big(
\Big(\{
|\eta_n^{N,c}(G_n)(G_n\otimes 1)-\eta_n^{N,f}(G_n)(1\otimes G_n)| -
$$
$$
|\eta_n^{c}(G_n)(G_n\otimes 1)-\eta_n^{f}(G_n)(1\otimes G_n)|
\}\otimes\overline{F}_{n,\check{\eta}_{n}^{N,M},c}\Big)
\bar{M}_{n+1}(\psi)\Big)\Big| \leq
$$
$$
\|\psi\|\check{\eta}_{n}^{N,M}\Big(
|[\eta_n^{N,c}-\eta_n^{c}](G_n)(G_n\otimes 1)| + 
|[\eta_n^{N,f}-\eta_n^{f}](G_n)(1\otimes G_n)|
\Big)
\check{\eta}_{n}^{N,M}(|\overline{F}_{n,\check{\eta}_{n}^{N,M},c}|).
$$
Then \eqref{eq:max_13} converges in probability to zero by using the argument for \eqref{eq:max_aux1} in the proof of Lemma \ref{lem:max_cond_exp_1} 
for $\check{\eta}_{n}^{N,M}($ $|\overline{F}_{n,\check{\eta}_{n}^{N,M},c}|)$. For the remaining term one can apply the (last) argument for \eqref{eq:max_new5}
in Lemma \ref{lem:max_bayes_conv}. Hence \eqref{eq:max_11} converges in probability to zero and the proof is concluded.
\end{proof}

\begin{lem}\label{lem:max_cond_exp_3}
Assume (A\ref{hyp:1}). Then if  
for any $\varphi\in\mathcal{C}_b(\mathsf{X}\times\mathsf{X})$, $n\geq 0$
$$
\check{\eta}_{n}^{N,M}(\varphi) \rightarrow_{\mathbb{P}} \check{\eta}_{n}^{M}(\varphi)
$$
we have for any $\psi\in\mathcal{C}_b(\mathsf{X}\times\mathsf{X})$, $n\geq 0$
$$
\frac{\check{\eta}_{n}^{N,M}(F_{n,\check{\eta}_{n}^{N,M},f}\wedge F_{n,\check{\eta}_{n}^{N,M},c})-
\check{\eta}_{n}^{M}(F_{n,\check{\eta}_{n}^{M},f}\wedge F_{n,\check{\eta}_{n}^{M},c})
}
{(1-\check{\eta}_{n}^{N,M}(F_{n,\check{\eta}_{n}^{N,M},f}\wedge F_{n,\check{\eta}_{n}^{N,M},c}))(1-\check{\eta}_{n}^{M}(F_{n,\check{\eta}_{n}^{M},f}\wedge F_{n,\check{\eta}_{n}^{M},c}))}\times
$$
$$
(\check{\eta}_{n}^{N,M}\otimes\check{\eta}_{n}^{N,M})\Big(\Big(\Big\{
F_{n,\check{\eta}_{n}^{N,M},f}\wedge F_{n,\check{\eta}_{n}^{N,M},c}
\Big\}\otimes
\overline{F}_{n,\check{\eta}_{n}^{N,M},c}\Big)
\bar{M}_{n+1}(\psi)
\Big)
\rightarrow_{\mathbb{P}} 0.
$$
\end{lem}

\begin{proof}
The proof concerning \eqref{eq:max_7} of Lemma \ref{lem:max_cond_exp_1} establishes that
$$
\frac{\check{\eta}_{n}^{N,M}(F_{n,\check{\eta}_{n}^{N,M},f}\wedge F_{n,\check{\eta}_{n}^{N,M},c})-
\check{\eta}_{n}^{M}(F_{n,\check{\eta}_{n}^{M},f}\wedge F_{n,\check{\eta}_{n}^{M},c})
}
{(1-\check{\eta}_{n}^{N,M}(F_{n,\check{\eta}_{n}^{N,M},f}\wedge F_{n,\check{\eta}_{n}^{N,M},c}))(1-\check{\eta}_{n}^{M}(F_{n,\check{\eta}_{n}^{M},f}\wedge F_{n,\check{\eta}_{n}^{M},c}))}  \rightarrow_{\mathbb{P}} 0.
$$
Then, almost surely
$$
\Big|(\check{\eta}_{n}^{N,M}\otimes\check{\eta}_{n}^{N,M})\Big(\Big(\Big\{
F_{n,\check{\eta}_{n}^{N,M},f}\wedge F_{n,\check{\eta}_{n}^{N,M},c}
\Big\}\otimes
\overline{F}_{n,\check{\eta}_{n}^{N,M},c}\Big)
\bar{M}_{n+1}(\psi)
\Big)\Big| \leq
$$
$$
\|\psi\|\check{\eta}_{n}^{N,M}(F_{n,\check{\eta}_{n}^{N,M},f}\wedge F_{n,\check{\eta}_{n}^{N,M},c})
\check{\eta}_{n}^{N,M}(|\overline{F}_{n,\check{\eta}_{n}^{N,M},c}|).
$$
By Lemma \ref{lem:max_bayes_conv} $\check{\eta}_{n}^{N,M}(F_{n,\check{\eta}_{n}^{N,M},f}\wedge F_{n,\check{\eta}_{n}^{N,M},c})$ converges in probability to a constant
and  by using the argument for \eqref{eq:max_aux1} in the proof of Lemma \ref{lem:max_cond_exp_1} 
for $\check{\eta}_{n}^{N,M}(|\overline{F}_{n,\check{\eta}_{n}^{N,M},c}|)$ we conclude the proof.
\end{proof}

\begin{lem}\label{lem:max_marginal_product}
Assume (A\ref{hyp:3}). Then if for any $n\geq 0$, $\varphi\in\mathcal{C}_b(\mathsf{X}\times\mathsf{X})$
$$
\check{\eta}_{n}^{N,M}(\varphi) \rightarrow_{\mathbb{P}} \check{\eta}_{n}^{M}(\varphi)
$$
we have for any $\psi\in\mathcal{C}_b(\mathsf{X}^2\times\mathsf{X}^2)$
$$
\Big(\check{\eta}_{n}^{N,M}\otimes \check{\eta}_{n}^{N,M}\Big)(\psi) \rightarrow_{\mathbb{P}} 
\Big(\check{\eta}_{n}^{M}\otimes \check{\eta}_{n}^{M}\Big)(\psi).
$$
\end{lem}

\begin{proof}
We will use a density argument. Define
$$
\mathscr{F} = \{f:\mathsf{X}^2\times\mathsf{X}^2\rightarrow\mathbb{R}:f(x,y)= g(x)\otimes h(y), (g,h)\in\mathcal{C}_b(\mathsf{X}\times\mathsf{X})\times\mathcal{C}_b(\mathsf{X}\times\mathsf{X})\}.
$$
Denote $\mathscr{G}$ as the set of functions which are finite, linear combinations of functions in $\mathscr{F}$. As we have that for some $n\geq 0$, $\varphi\in\mathcal{C}_b(\mathsf{X}\times\mathsf{X})$, $\check{\eta}_{n}^{N,M}(\varphi) \rightarrow_{\mathbb{P}} \check{\eta}_{n}^{M}(\varphi)$, it follows that for $\tilde{\psi}\in\mathscr{G}$
$$
\Big(\check{\eta}_{n}^{N,M}\otimes \check{\eta}_{n}^{N,M}\Big)(\tilde{\psi}) \rightarrow_{\mathbb{P}} 
\Big(\check{\eta}_{n}^{M}\otimes \check{\eta}_{n}^{M}\Big)(\tilde{\psi})
$$
and thus, by \cite[Theorem 25.12]{bill}
\begin{equation}\label{eq:max_1}
\lim_{N\rightarrow\infty}
\mathbb{E}\Big[\Big|\Big\{\Big(\check{\eta}_{n}^{N,M}\otimes \check{\eta}_{n}^{N,M}\Big)-\Big(\check{\eta}_{n}^{M}\otimes \check{\eta}_{n}^{M}\Big)\Big\}(\tilde{\psi})\Big|\Big]= 0.
\end{equation}

Let $\epsilon>0$ be arbitrary and $\psi\in\mathcal{C}_b(\mathsf{X}^2\times\mathsf{X}^2)$. By the Stone-Weierstrass theorem, $\mathscr{G}$ is dense in $\mathcal{C}_b(\mathsf{X}^2\times\mathsf{X}^2)$, hence
there exist a $\tilde{\psi}\in\mathscr{G}$ such that
\begin{equation}\label{eq:max_2}
\|\psi-\tilde{\psi}\| < \epsilon/3.
\end{equation}
Then we have for any $N\geq 1$
$$
\mathbb{E}\Big[\Big|\Big\{\Big(\check{\eta}_{n}^{N,M}\otimes \check{\eta}_{n}^{N,M}\Big)-\Big(\check{\eta}_{n}^{M}\otimes \check{\eta}_{n}^{M}\Big)\Big\}(\psi)\Big|\Big]
\leq 
$$
$$
\mathbb{E}\Big[\Big(\check{\eta}_{n}^{N,M}\otimes \check{\eta}_{n}^{N,M}\Big)(|\psi-\tilde{\psi}|)\Big] + 
\Big(\check{\eta}_{n}^{M}\otimes \check{\eta}_{n}^{M}\Big)(|\psi-\tilde{\psi}|) + 
\mathbb{E}\Big[\Big|\Big\{\Big(\check{\eta}_{n}^{N,M}\otimes \check{\eta}_{n}^{N,M}\Big)-\Big(\check{\eta}_{n}^{M}\otimes \check{\eta}_{n}^{M}\Big)\Big\}(\tilde{\psi})\Big|\Big].
$$
By \eqref{eq:max_1} there exist a $N^*\geq 1$ such that for every $N\geq N^*$
$$
\mathbb{E}\Big[\Big|\Big\{\Big(\check{\eta}_{n}^{N,M}\otimes \check{\eta}_{n}^{N,M}\Big)-\Big(\check{\eta}_{n}^{M}\otimes \check{\eta}_{n}^{M}\Big)\Big\}(\tilde{\psi})\Big|\Big] <\epsilon/3.
$$
Hence applying \eqref{eq:max_2} to the terms
$$
\mathbb{E}\Big[\Big(\check{\eta}_{n}^{N,M}\otimes \check{\eta}_{n}^{N,M}\Big)(|\psi-\tilde{\psi}|)\Big] \quad\textrm{and}\quad
\Big(\check{\eta}_{n}^{M}\otimes \check{\eta}_{n}^{M}\Big)(|\psi-\tilde{\psi}|) 
$$
we have for every $N\geq N^*$ 
$$
\mathbb{E}\Big[\Big|\Big\{\Big(\check{\eta}_{n}^{N,M}\otimes \check{\eta}_{n}^{N,M}\Big)-\Big(\check{\eta}_{n}^{M}\otimes \check{\eta}_{n}^{M}\Big)\Big\}(\psi)\Big|\Big]
<\epsilon
$$
and as $\epsilon>0$ is arbitrary we have 
$$
\Big(\check{\eta}_{n}^{N,M}\otimes \check{\eta}_{n}^{N,M}\Big)(\psi) \rightarrow_{\mathbb{P}} 
\Big(\check{\eta}_{n}^{M}\otimes \check{\eta}_{n}^{M}\Big)(\psi)
$$
as was to be proved.
\end{proof}

\begin{lem}\label{lem:max_ci_hard_part}
Assume (A\ref{hyp:1}-\ref{hyp:3}). Then if  
for any $\varphi\in\mathcal{C}_b(\mathsf{X}\times\mathsf{X})$, $n\geq 0$
$$
\check{\eta}_{n}^{N,M}(\varphi) \rightarrow_{\mathbb{P}} \check{\eta}_{n}^{M}(\varphi)
$$
we have for any $\psi\in\mathcal{C}_b(\mathsf{X}\times\mathsf{X})$, $n\geq 0$
$$
\Big(1-
\check{\eta}_{n}^{N,M}\Big(\{F_{n,\check{\eta}_{n}^{N,M},f} \wedge F_{n,\check{\eta}_{n}^{N,M},c}\}\Big)
\Big)(\check{\eta}_{n}^{N,M}\otimes\check{\eta}_{n}^{N,M})\Big(\Big\{\overline{F}_{n,\check{\eta}_{n}^{N,M},f}\otimes \overline{F}_{n,\check{\eta}_{n}^{N,M},c}\Big\}\bar{M}_{n+1}(\psi)\Big) \rightarrow_{\mathbb{P}}
$$
$$
\Big(1-
\check{\eta}_{n}^{M}\Big(\{F_{n,\check{\eta}_{n}^{M},f} \wedge F_{n,\check{\eta}_{n}^{M},c}\}\Big)
\Big)(\check{\eta}_{n}^{M}\otimes\check{\eta}_{n}^{M})\Big(\Big\{\overline{F}_{n,\check{\eta}_{n}^{M},f}\otimes \overline{F}_{n,\check{\eta}_{n}^{M},c}\Big\}\bar{M}_{n+1}(\psi)\Big).
$$
\end{lem}

\begin{proof}
We make the definitions
\begin{eqnarray}
T_1^N & := & \Big(\check{\eta}_{n}^{M}\Big(\{F_{n,\check{\eta}_{n}^{M},f} \wedge F_{n,\check{\eta}_{n}^{M},c}\}\Big)-
\check{\eta}_{n}^{N,M}\Big(\{F_{n,\check{\eta}_{n}^{N,M},f} \wedge F_{n,\check{\eta}_{n}^{N,M},c}\}\Big)\Big)\times \nonumber \\& &(\check{\eta}_{n}^{N,M}\otimes\check{\eta}_{n}^{N,M})\Big(\Big\{\overline{F}_{n,\check{\eta}_{n}^{N,M},f}\otimes \overline{F}_{n,\check{\eta}_{n}^{N,M},c}\Big\}\bar{M}_{n+1}(\psi)\Big) \label{eq:max_14}\\
T_2^N & := & \Big(1-
\check{\eta}_{n}^{M}\Big(\{F_{n,\check{\eta}_{n}^{M},f} \wedge F_{n,\check{\eta}_{n}^{M},c}\}\Big)\Big)
(\check{\eta}_{n}^{N,M}\otimes\check{\eta}_{n}^{N,M})\Big(\Big\{\overline{F}_{n,\check{\eta}_{n}^{N,M},f}\otimes \overline{F}_{n,\check{\eta}_{n}^{N,M},c}
- \nonumber \\ & &
\overline{F}_{n,\check{\eta}_{n}^{M},f}\otimes \overline{F}_{n,\check{\eta}_{n}^{M},c}
\Big\}\bar{M}_{n+1}(\psi)\Big) \label{eq:max_15}\\
T_3^N & = & \Big[(\check{\eta}_{n}^{N,M}\otimes\check{\eta}_{n}^{N,M})- 
(\check{\eta}_{n}^{M}\otimes\check{\eta}_{n}^{M})
\Big]\Big(\Big\{\overline{F}_{n,\check{\eta}_{n}^{M},f}\otimes \overline{F}_{n,\check{\eta}_{n}^{M},c}\Big\}\bar{M}_{n+1}(\psi)\Big) \label{eq:max_16}
\end{eqnarray}
Then we have
$$
T_1^N + T_2^N + T_3^N = 
\Big(1-
\check{\eta}_{n}^{N,M}\Big(\{F_{n,\check{\eta}_{n}^{N,M},f} \wedge F_{n,\check{\eta}_{n}^{N,M},c}\}\Big)
\Big)(\check{\eta}_{n}^{N,M}\otimes\check{\eta}_{n}^{N,M})\Big(\Big\{\overline{F}_{n,\check{\eta}_{n}^{N,M},f}\otimes \overline{F}_{n,\check{\eta}_{n}^{N,M},c}\Big\}\times
$$
$$
\bar{M}_{n+1}(\psi)\Big) -
\Big(1-
\check{\eta}_{n}^{M}\Big(\{F_{n,\check{\eta}_{n}^{M},f} \wedge F_{n,\check{\eta}_{n}^{M},c}\}\Big)
\Big)(\check{\eta}_{n}^{M}\otimes\check{\eta}_{n}^{M})\Big(\Big\{\overline{F}_{n,\check{\eta}_{n}^{M},f}\otimes \overline{F}_{n,\check{\eta}_{n}^{M},c}\Big\}\bar{M}_{n+1}(\psi)\Big) 
$$
We will show that \eqref{eq:max_14}-\eqref{eq:max_16} will converge in probability to zero.

%$$
%\Big(\check{\eta}_{n}^{M}\Big(\{F_{n,\check{\eta}_{n}^{M},f} \wedge F_{n,\check{\eta}_{n}^{M},c}\}\Big)-
%\check{\eta}_{n}^{N,M}\Big(\{F_{n,\check{\eta}_{n}^{N,M},f} \wedge F_{n,\check{\eta}_{n}^{N,M},c}\}\Big) (\check{\eta}_{n}^{N,M}\otimes\check{\eta}_{n}^{N,M})\Big(\Big\{\overline{F}_{n,\check{\eta}_{n}^{N,M},f}\otimes \overline{F}_{n,\check{\eta}_{n}^{N,M},c}\Big\}\bar{M}_n(\psi)\Big) +
%$$
%$$
%\Big(1-
%\check{\eta}_{n}^{M}\Big(\{F_{n,\check{\eta}_{n}^{M},f} \wedge F_{n,\check{\eta}_{n}^{M},c}\}\Big)\Big)\Big(
%(\check{\eta}_{n}^{N,M}\otimes\check{\eta}_{n}^{N,M})\Big(\Big\{\overline{F}_{n,\check{\eta}_{n}^{N,M},f}\otimes \overline{F}_{n,\check{\eta}_{n}^{N,M},c}
%-
%\overline{F}_{n,\check{\eta}_{n}^{M},f}\otimes \overline{F}_{n,\check{\eta}_{n}^{M},c}
%\Big\}\bar{M}_n(\psi)\Big) +
%$$
%\begin{equation}\label{eq:max_5}
%\Big[(\check{\eta}_{n}^{N,M}\otimes\check{\eta}_{n}^{N,M})-
%(\check{\eta}_{n}^{M}\otimes\check{\eta}_{n}^{M})
%\Big]\Big(\Big\{\overline{F}_{n,\check{\eta}_{n}^{M},f}\otimes \overline{F}_{n,\check{\eta}_{n}^{M},c}\Big\}\bar{M}_n(\psi)\Big)\Big).
%\end{equation}

\textbf{Term \eqref{eq:max_14}:} By Lemma \ref{lem:max_bayes_conv}
$$
\Big|\check{\eta}_{n}^{M}\Big(\{F_{n,\check{\eta}_{n}^{M},f} \wedge F_{n,\check{\eta}_{n}^{M},c}\}\Big)-
\check{\eta}_{n}^{N,M}\Big(\{F_{n,\check{\eta}_{n}^{N,M},f} \wedge F_{n,\check{\eta}_{n}^{N,M},c}\}\Big)\Big|\rightarrow_{\mathbb{P}} 0
$$
so if we can show that
$$
T^N_1 := \Big|(\check{\eta}_{n}^{N,M}\otimes\check{\eta}_{n}^{N,M})\Big(\Big\{\overline{F}_{n,\check{\eta}_{n}^{N,M},f}\otimes \overline{F}_{n,\check{\eta}_{n}^{N,M},c}\Big\}\bar{M}_{n+1}(\psi)\Big)\Big|
$$
is (almost-surely) upper-bounded by a term which converges in probability to a positive constant, we have shown that that \eqref{eq:max_14} converges in probability to zero. Clearly, almost-surely
$$
T^N_1 \leq \|\psi\|\check{\eta}_{n}^{N,M}(|\overline{F}_{n,\check{\eta}_{n}^{N,M},f}|)\check{\eta}_{n}^{N,M}(|\overline{F}_{n,\check{\eta}_{n}^{N,M},c}|)
$$
we focus on showing that $\check{\eta}_{n}^{N,M}(|\overline{F}_{n,\check{\eta}_{n}^{N,M},c}|)$ is upper-bounded by a term which converges in probability to a finite
constant. This can be verified in an almost identical manner to the approach for \eqref{eq:max_8} in the proof of Lemma \ref{lem:max_cond_exp_1}; hence we have verified that  \eqref{eq:max_14} converges in probability to zero.

\textbf{Term \eqref{eq:max_15}:} Set
\begin{eqnarray}
C & := & \check{\eta}_{n}^{M}\Big(\{F_{n,\check{\eta}_{n}^{M},f} \wedge F_{n,\check{\eta}_{n}^{M},c}\}\Big) \nonumber \\
T_4^N & := &  C
(\check{\eta}_{n}^{N,M}\otimes\check{\eta}_{n}^{N,M})\Big(\Big\{[\overline{F}_{n,\check{\eta}_{n}^{N,M},f}-\overline{F}_{n,\check{\eta}_{n}^{M},f}]\otimes \overline{F}_{n,\check{\eta}_{n}^{N,M},c}\Big\}\bar{M}_{n+1}(\psi)\Big) \label{eq:max_17}\\
T_5^N & := & C
(\check{\eta}_{n}^{N,M}\otimes\check{\eta}_{n}^{N,M})\Big(\Big\{
\overline{F}_{n,\check{\eta}_{n}^{M},f}\otimes
[\overline{F}_{n,\check{\eta}_{n}^{N,M},c}-\overline{F}_{n,\check{\eta}_{n}^{M},c}]\Big\}\bar{M}_{n+1}(\psi)\Big) \label{eq:max_18}
\end{eqnarray}
then
$$
T_2^N = T_4^N + T_5^N.
$$
We need to show that \eqref{eq:max_17} and \eqref{eq:max_18} converge in probability to zero. As the proof for \eqref{eq:max_18} is similar and easier than that
for \eqref{eq:max_17}, we focus on the latter. Define
\begin{eqnarray*}
T_6^N & := & C
(\check{\eta}_{n}^{N,M}\otimes\check{\eta}_{n}^{N,M})\Big(\Big\{
\Big(
\frac{F_{n,\check{\eta}_{n}^{N,M},f}}{\check{\eta}_{n}^{N,M}(F_{n,\check{\eta}_{n}^{N,M},f}-F_{n,\check{\eta}_{n}^{N,M},f}\wedge F_{n,\check{\eta}_{n}^{N,M},c})}
\otimes \nonumber \\ & & 
\overline{F}_{n,\check{\eta}_{n}^{N,M},c}\Big)\Big\}
\bar{M}_{n+1}(\psi)
\Big) \\
T_7^N & := & C (\check{\eta}_{n}^{N,M}\otimes\check{\eta}_{n}^{N,M})\Big(
\Big\{\Big(
\frac{F_{n,\check{\eta}_{n}^{M},f}}{\check{\eta}_{n}^{M}(F_{n,\check{\eta}_{n}^{M},f}-F_{n,\check{\eta}_{n}^{M},f}\wedge F_{n,\check{\eta}_{n}^{M},c})}
\Big)\otimes \nonumber \\ & &
\overline{F}_{n,\check{\eta}_{n}^{N,M},c}\Big\}
\bar{M}_{n+1}(\psi)
\Big)  \\
T_8^N & := &  C (\check{\eta}_{n}^{N,M}\otimes\check{\eta}_{n}^{N,M})\Big(\Big\{\Big(
\frac{F_{n,\check{\eta}_{n}^{M},f}\wedge F_{n,\check{\eta}_{n}^{M},c}-
F_{n,\check{\eta}_{n}^{N,M},f}\wedge F_{n,\check{\eta}_{n}^{N,M},c}}
{\check{\eta}_{n}^{M}(F_{n,\check{\eta}_{n}^{M},f}-F_{n,\check{\eta}_{n}^{M},f}\wedge F_{n,\check{\eta}_{n}^{M},c})}
\Big)\otimes \nonumber\\ & & 
\overline{F}_{n,\check{\eta}_{n}^{N,M},c}\Big\}
\bar{M}_{n+1}(\psi)
\Big) \\
T_9^N & := & \frac{\check{\eta}_{n}^{N,M}(F_{n,\check{\eta}_{n}^{N,M},f}\wedge F_{n,\check{\eta}_{n}^{N,M},c})-
\check{\eta}_{n}^{M}(F_{n,\check{\eta}_{n}^{M},f}\wedge F_{n,\check{\eta}_{n}^{M},c})
}
{(1-\check{\eta}_{n}^{N,M}(F_{n,\check{\eta}_{n}^{N,M},f}\wedge F_{n,\check{\eta}_{n}^{N,M},c}))(1-\check{\eta}_{n}^{M}(F_{n,\check{\eta}_{n}^{M},f}\wedge F_{n,\check{\eta}_{n}^{M},c}))}\times \nonumber \\ & &
C (\check{\eta}_{n}^{N,M}\otimes\check{\eta}_{n}^{N,M})\Big(\Big\{
\Big(F_{n,\check{\eta}_{n}^{N,M},f}\wedge F_{n,\check{\eta}_{n}^{N,M},c}\Big)
\otimes
\overline{F}_{n,\check{\eta}_{n}^{N,M},c}\Big\}
\bar{M}_{n+1}(\psi)
\Big)
\end{eqnarray*}
Then we have that
$$
T_4^N = T_6^N - T_7^N + T_8^N - T_9^N.
$$
By Lemma \ref{lem:max_cond_exp_1} $T_6^N - T_7^N\rightarrow_{\mathbb{P}}0$, by Lemma \ref{lem:max_cond_exp_2} $ T_8^N\rightarrow_{\mathbb{P}}0$
and by Lemma \ref{lem:max_cond_exp_3} $ T_9^N\rightarrow_{\mathbb{P}}0$, hence we conclude that $T_4^N\rightarrow_{\mathbb{P}} 0$.

\textbf{Term \eqref{eq:max_16}:} As 
$$
\Big\{\overline{F}_{n,\check{\eta}_{n}^{M},f}\otimes \overline{F}_{n,\check{\eta}_{n}^{M},c}\Big\}\bar{M}_{n+1}(\psi) \in \mathcal{C}_b(\mathsf{X}^2\times\mathsf{X}^2)
$$
it easily follows that by Lemma \ref{lem:max_marginal_product}
$$
\Big[(\check{\eta}_{n}^{N,M}\otimes\check{\eta}_{n}^{N,M})-
(\check{\eta}_{n}^{M}\otimes\check{\eta}_{n}^{M})
\Big]\Big(\Big\{\overline{F}_{n,\check{\eta}_{n}^{M},f}\otimes \overline{F}_{n,\check{\eta}_{n}^{M},c}\Big\}\bar{M}_{n+1}(\psi)\Big)\Big) \rightarrow_{\mathbb{P}} 0.
$$

%
%\textbf{Term 1}
%
%
%%The proof for $\check{\eta}_{n}^{N,M}(|\overline{F}_{n,\check{\eta}_{n}^{N,M},c}|)$ is almost identical and hence omitted. We have, almost-surely
%%\begin{eqnarray*}
%%\check{\eta}_{n}^{N,M}(|\overline{F}_{n,\check{\eta}_{n}^{N,M},f}|) & \leq & \frac{\check{\eta}_{n}^{N,M}(|F_{n,\check{\eta}_{n}^{N,M},f}-\{F_{n,\check{\eta}_{n}^{N,M},f}\wedge F_{n,\check{\eta}_{n}^{N,M},c}\}|)}
%%{|\check{\eta}_{n}^{N,M}(F_{n,\check{\eta}_{n}^{N,M},f}-\{F_{n,\check{\eta}_{n}^{N,M},f}\wedge F_{n,\check{\eta}_{n}^{N,M},c}\})|} \\ & \leq & 
%%\frac{\check{\eta}_{n}^{N,M}(F_{n,\check{\eta}_{n}^{N,M},f})+\check{\eta}_{n}^{N,M}(\{F_{n,\check{\eta}_{n}^{N,M},f}\wedge F_{n,\check{\eta}_{n}^{N,M},c}\})}
%%{|\check{\eta}_{n}^{N,M}(F_{n,\check{\eta}_{n}^{N,M},f}-\{F_{n,\check{\eta}_{n}^{N,M},f}\wedge F_{n,\check{\eta}_{n}^{N,M},c}\})|}
%%\end{eqnarray*}
%%Now
%%$$
%%\check{\eta}_{n}^{N,M}(F_{n,\check{\eta}_{n}^{N,M},f}) = \frac{\eta_n^{N,f}(G_n)}{\eta_n^{N,f}(G_n)} = 1
%%$$
%%and by Lemma \ref{lem:max_bayes_conv} $\check{\eta}_{n}^{N,M}(\{F_{n,\check{\eta}_{n}^{N,M},f}\wedge F_{n,\check{\eta}_{n}^{N,M},c}\})$ converges in probability to a
%%constant; hence we have verified that  term 1 converges in probability to zero.
%
%\textbf{Term 2}:
\end{proof}

\section{Technical Results for the MCPF}

\begin{prop}\label{prop:lp_bound_mcpf}
For any $n\geq 0$, $s\in\{f,c\}$, $p\geq 1$ there exists a $C<+\infty$
such that for any $\varphi\in\mathcal{B}_b(\mathsf{X})$, $N\geq 1$ we have
$$
\mathbb{E}[|[\eta_n^{N,s}-\eta_n^{s}](\varphi)|^{p}]^{1/p} \leq \frac{C\|\varphi\|}{\sqrt{N}}.
$$
\end{prop}

\begin{proof}
As for Proposition \ref{prop:lp_bound_ind}.
\end{proof}

Recall that we are assuming that $M_n^s$ has a density and we are denoting the density also as $M_n^s$.

\begin{theorem}\label{theo:wlln_mcpf}
Suppose that for $n\geq 1$, $s\in\{f,c\}$, $M_n^s\in\mathcal{B}_b(\mathsf{X}\times\mathsf{X})$. % $n\geq 1$.% and $\eta_0^s\in\mathscal{B}_(\mathsf{X})$, $s\in\{f,c\}$
Then for any $\varphi\in\mathcal{B}_b(\mathsf{X}\times\mathsf{X})$, $n\geq 0$ we have
$$
\check{\eta}_{n}^{N,C}(\varphi) \rightarrow_{\mathbb{P}} \check{\eta}_{n}^{C}(\varphi).
$$
\end{theorem}

\begin{proof}
The proof is by induction, the initialization following by the WLLN for i.i.d.~random variables. The result is assumed at time $n-1$ and we have the decomposition
$$
\check{\eta}_{n}^{N,C}(\varphi) - \check{\eta}_{n}^{C}(\varphi) = \check{\eta}_{n}^{N,C}(\varphi) - 
\check{\Phi}_p^C(\eta_{p-1}^{N,f},
\eta_{p-1}^{N,c})(\varphi)
+ \check{\Phi}_p^C(\eta_{p-1}^{N,f},
\eta_{p-1}^{N,c})(\varphi) -  \check{\eta}_{n}^{C}(\varphi).
$$
As in the proof of Theorem \ref{theo:wlln_max}, one can easily prove that
$$
|\check{\eta}_{n}^{N,C}(\varphi) - 
\check{\Phi}_p^C(\eta_{p-1}^{N,f},
\eta_{p-1}^{N,c})(\varphi)| \rightarrow_{\mathbb{P}} 0
$$
by using the (conditional) Marcinkiewicz-Zygmund inequality, so we focus on
$$
\check{\Phi}_p^C(\eta_{p-1}^{N,f},\eta_{p-1}^{N,c})(\varphi) -  \check{\eta}_{n}^{C}(\varphi).
$$
Define:
\begin{eqnarray*}
T_1^N & := & \int_{\mathsf{X}} \varphi(x,x) F_{n-1,\eta_{n-1}^{N,f},f}(x)\wedge F_{n,\eta_{n-1}^{N,c},c}(x) dx - 
\int_{\mathsf{X}} \varphi(x,x) F_{n-1,\eta_{n-1}^{f},f}(x)\wedge F_{n-1,\eta_{n-1}^{c},c}(x) dx \\
T_2^N & := & \Big(1-\int_{\mathsf{X}} 
F_{n-1,\eta_{n-1}^{N,f},f}(x)\wedge F_{n-1,\eta_{n-1}^{N,c},c}(x)dx\Big)^{-1}
\int_{\mathsf{X}\times\mathsf{X}} \varphi(x,y)\overline{F}_{n-1,\eta_{n-1}^{N,f},\eta_{n-1}^{N,c},f}(x) \times \\ & &\overline{F}_{n-1,\eta_{n-1}^{N,c},\eta_{n-1}^{N,f},c}(x) dxdy -
\Big(1-\int_{\mathsf{X}} 
F_{n-1,\eta_{n-1}^{f},f}(x)\wedge F_{n-1,\eta_{n-1}^{c},c}(x)dx\Big)^{-1} \times \\ & &
\int_{\mathsf{X}\times\mathsf{X}} \varphi(x,y)\overline{F}_{n-1,\eta_{n-1}^{f},\eta_{n-1}^{c},f}(x) \overline{F}_{n-1,\eta_{n-1}^{c},\eta_{n-1}^{f},c}(x) dxdy
\end{eqnarray*}
Recalling that $\check{\eta}_{n}^{C}(\varphi) = \check{\Phi}_n^C(\eta_{n-1}^{f},\eta_{n-1}^{c})(\varphi)$ we have 
that
$$
T_1^N + T_2^N = \check{\Phi}_p^C(\eta_{p-1}^{N,f},\eta_{p-1}^{N,c})(\varphi) -  \check{\eta}_{n}^{C}(\varphi).
$$
By Lemma \ref{lem:mc_tech_lem} 1.~$T_1^N\rightarrow_{\mathbb{P}}0$ and by Lemma \ref{lem:mc_tech_lem} 2.~$T_2^N\rightarrow_{\mathbb{P}}0$ which concludes the proof.
\end{proof}

\begin{lem}\label{lem:mc_tech_lem}
Suppose that for $n\geq 1$, $s\in\{f,c\}$, %$\eta_0^s\in\mathcal{B}_b(\mathsf{X})$ and
$M_n^s\in\mathcal{B}_b(\mathsf{X}\times\mathsf{X})$ and that for $\varphi\in\mathcal{B}_b(\mathsf{X}\times\mathsf{X})$, $n\geq 0$
$$
\check{\eta}_{n}^{N,C}(\varphi) \rightarrow_{\mathbb{P}} \check{\eta}_{n}^{C}(\varphi).
$$
Then for any $\psi\in\mathcal{B}_b(\mathsf{X}\times\mathsf{X})$, $n\geq 0$ we have
\begin{enumerate}
\item{
$$
\int_{\mathsf{X}} \psi(x,x) F_{n,\eta_{n}^{N,f},f}(x)\wedge F_{n,\eta_{n}^{N,c},c}(x) dx \rightarrow_\mathbb{P}
\int_{\mathsf{X}} \psi(x,x) F_{n,\eta_{n}^{f},f}(x)\wedge F_{n,\eta_{n}^{c},c}(x) dx.
$$
}
\item{
$$
\Big(1-\int_{\mathsf{X}} 
F_{n,\eta_{n}^{N,f},f}(x)\wedge F_{n,\eta_{n}^{N,c},c}(x)dx\Big)^{-1}
\int_{\mathsf{X}\times\mathsf{X}} \psi(x,y)\overline{F}_{n,\eta_{n}^{N,f},\eta_{n}^{N,c},f}(x) \overline{F}_{n,\eta_{n}^{N,c},\eta_{n}^{N,f},c}(x) dxdy \rightarrow_\mathbb{P}
$$
$$
\Big(1-\int_{\mathsf{X}} 
F_{n,\eta_{n}^{f},f}(x)\wedge F_{n,\eta_{n}^{c},c}(x)dx\Big)^{-1}
\int_{\mathsf{X}\times\mathsf{X}} \psi(x,y)\overline{F}_{n,\eta_{n}^{f},\eta_{n}^{c},f}(x) \overline{F}_{n,\eta_{n}^{c},\eta_{n}^{f},c}(x) dxdy.
$$
}
\end{enumerate}
\end{lem}

\begin{proof}
We begin by noting that for any $x\in\mathsf{X}$, by the assumption that $\check{\eta}_{n}^{N,C}(\varphi) \rightarrow_{\mathbb{P}} \check{\eta}_{n}^{C}(\varphi)$, we have
\begin{eqnarray*}
F_{n,\eta_{n}^{N,f},f}(x) & \rightarrow_\mathbb{P} &  F_{n,\eta_{n}^{f},f}(x) \\ 
F_{n,\eta_{n}^{N,c},c}(x) & \rightarrow_\mathbb{P} &  F_{n,\eta_{n}^{c},c}(x)
\end{eqnarray*}
and hence that
\begin{eqnarray}
F_{n,\eta_{n}^{N,f},f}(x)\wedge F_{n,\eta_{n}^{N,c},c}(x) & \rightarrow_\mathbb{P} &
F_{n,\eta_{n}^{f},f}(x)\wedge F_{n,\eta_{n}^{c},c}(x) \label{eq:mc_eq1}\\
\overline{F}_{n,\eta_{n}^{N,f},\eta_{n}^{N,c},f}(x) & \rightarrow_\mathbb{P} &  \overline{F}_{n,\eta_{n}^{f},\eta_{n}^{c},f}(x) \label{eq:mc_eq2}\\ 
\overline{F}_{n,\eta_{n}^{N,c},\eta_{n}^{N,f},c}(x) & \rightarrow_\mathbb{P} &  \overline{F}_{n,\eta_{n}^{c},\eta_{n}^{f},c}(x). \label{eq:mc_eq3}
\end{eqnarray}

For 1.~we have 
$$
\mathbb{E}\Big[\Big|\int_{\mathsf{X}}\psi(x,x)\Big\{F_{n,\eta_{n}^{N,f},f}(x)\wedge F_{n,\eta_{n}^{N,c},c}(x) -
F_{n,\eta_{n}^{f},f}(x)\wedge F_{n,\eta_{n}^{c},c}(x) 
\Big\}dx\Big|\Big] \leq 
$$
$$
\|\psi\|\int_{\mathsf{X}}\Big\{ 
\mathbb{E}\Big[\Big|F_{n,\eta_{n}^{N,f},f}(x)\wedge F_{n,\eta_{n}^{N,c},c}(x) -
F_{n,\eta_{n}^{f},f}(x)\wedge F_{n,\eta_{n-1}^{c},c}(x) 
\Big|\Big]\Big\}dx.
$$
Then as $F_{n,\eta_{n}^{N,f},f}(x)\wedge F_{n,\eta_{n}^{N,c},c}(x)\leq \|M_{n+1}^f\|\wedge\|M_{n+1}^c\|$,
by \eqref{eq:mc_eq1} and \cite[Theorem 25.12]{bill} for any $x\in\mathsf{X}$
$$
\lim_{N\rightarrow\infty}\mathbb{E}\Big[\Big|F_{n,\eta_{n-1}^{N,f},f}(x)\wedge F_{n,\eta_{n}^{N,c},c}(x) -
F_{n,\eta_{n}^{f},f}(x)\wedge F_{n-1,\eta_{n-1}^{c},c}(x) 
\Big|\Big] = 0.
$$
Hence by the bounded convergence theorem
$$
\lim_{N\rightarrow\infty}
\mathbb{E}\Big[\Big|\int_{\mathsf{X}}\psi(x,x)\Big\{F_{n,\eta_{n}^{N,f},f}(x)\wedge F_{n,\eta_{n}^{N,c},c}(x) -
F_{n,\eta_{n}^{f},f}(x)\wedge F_{n-1,\eta_{n}^{c},c}(x) 
\Big\}dx\Big|\Big] = 0
$$
and thus the result follows.

For 2.~using almost the same proof as in 1.~except using \eqref{eq:mc_eq2}-\eqref{eq:mc_eq3} in place of using \eqref{eq:mc_eq1}, we have
$$
\int_{\mathsf{X}\times\mathsf{X}} \psi(x,y)\overline{F}_{n,\eta_{n}^{N,f},\eta_{n}^{N,c},f}(x) \overline{F}_{n,\eta_{n}^{N,c},\eta_{n}^{N,f},c}(x) dxdy\rightarrow_\mathbb{P}
$$
\begin{equation}\label{eq:mc_eq4}
\int_{\mathsf{X}\times\mathsf{X}} \psi(x,y)\overline{F}_{n,\eta_{n}^{f},\eta_{n}^{c},f}(x) \overline{F}_{n,\eta_{n}^{c},\eta_{n}^{f},c}(x) dxdy.
\end{equation}
Then define
\begin{eqnarray*}
T_1^N & := & \Big(1-\int_{\mathsf{X}} 
F_{n,\eta_{n}^{N,f},f}(x)\wedge F_{n,\eta_{n}^{N,c},c}(x)dx\Big)^{-1} - \Big(1-\int_{\mathsf{X}} 
F_{n,\eta_{n}^{f},f}(x)\wedge F_{n,\eta_{n}^{c},c}(x)dx\Big)^{-1}\Big) \times \\ & & 
\int_{\mathsf{X}\times\mathsf{X}} \psi(x,y)\overline{F}_{n,\eta_{n}^{N,f},\eta_{n}^{N,c},f}(x) \overline{F}_{n,\eta_{n}^{N,c},\eta_{n}^{N,f},c}(x) dxdy \\
T_2^N & := & \Big(1-\int_{\mathsf{X}} 
F_{n,\eta_{n}^{f},f}(x)\wedge F_{n,\eta_{n}^{c},c}(x)dx\Big)^{-1}\Big(\int_{\mathsf{X}\times\mathsf{X}} \psi(x,y)\overline{F}_{n,\eta_{n}^{N,f},\eta_{n}^{N,c},f}(x)\times \\ & &
\overline{F}_{n,\eta_{n}^{N,c},\eta_{n}^{N,f},c}(x) dxdy - 
\int_{\mathsf{X}\times\mathsf{X}} \psi(x,y)\overline{F}_{n,\eta_{n}^{f},\eta_{n}^{c},f}(x)\overline{F}_{n,\eta_{n}^{c},\eta_{n}^{f},c}(x) dxdy
\Big).
\end{eqnarray*}
Then we have
$$
T_1^N + T_2^N = 
$$
$$
\Big(1-\int_{\mathsf{X}} 
F_{n,\eta_{n}^{N,f},f}(x)\wedge F_{n,\eta_{n}^{N,c},c}(x)dx\Big)^{-1}
\int_{\mathsf{X}\times\mathsf{X}} \psi(x,y)\overline{F}_{n,\eta_{n}^{N,f},\eta_{n}^{N,c},f}(x) F_{n,\eta_{n}^{N,c},\eta_{n}^{N,f},c}(x) dxdy -
$$
$$
\Big(1-\int_{\mathsf{X}} 
F_{n,\eta_{n}^{f},f}(x)\wedge F_{n,\eta_{n}^{c},c}(x)dx\Big)^{-1}
\int_{\mathsf{X}\times\mathsf{X}} \psi(x,y)\overline{F}_{n,\eta_{n}^{f},\eta_{n}^{c},f}(x) F_{n,\eta_{n}^{c},\eta_{n}^{f},c}(x) dxdy.
$$
By 1.~and \eqref{eq:mc_eq4}, $T_1^N\rightarrow_{\mathbb{P}}0$ and by \eqref{eq:mc_eq4} $T_2^N\rightarrow_{\mathbb{P}}0$ and hence the proof is complete.
\end{proof}

\section{Proofs for the Asymptotic Variance}\label{app:av_proofs}

Throughout $v$ is as in (H\ref{hypav:1}). We will frequently quote results from \cite{whiteley} which are proved when
one considers $v$ not $v^{\xi}$, $\xi\in(0,1]$. The case $\xi\in(0,1)$ for these quoted results apply using \cite[Lemma 3]{whiteley}.
Note that \cite[Lemma 3]{whiteley} implies that (H\ref{hypav:1}-\ref{hypav:4}) apply in the case for $\xi\in(0,1)$, with Lyapunov function $v^{\xi}$,
with constants depending upon $\xi$.
In our proofs $C$ is a finite and positive constant whose value may change from line-to-line and does not depend on $n,p,f,c$. If there
are any important dependencies of $C$ they will be stated.

\begin{proof}[Proof of Theorem \ref{theo:av_thm}]
We have 
$$
\sigma^{2,s}_n(\varphi) = \check{\eta}_n^s\Big((\varphi\otimes 1-1\otimes\varphi - ([\eta_n^f-\eta_n^c](\varphi)))^2\Big) + 
\sum_{p=0}^{n-1} \check{\eta}_p^s(\{(D_{p,n}^f(\varphi)\otimes 1- 1\otimes D_{p,n}^c(\varphi))\}^2).
$$
Define
\begin{eqnarray}
T_1 & := & 2\eta_p^f\Big((D_{p,n}^f(\varphi)-D_{p,n}^c(\varphi))^2\Big)) \label{eq:av_main1}\\
T_2 & := & 2\check{\eta}_p^s\Big(
(D_{p,n}^c(\varphi)\otimes 1 - 1\otimes D_{p,n}^c(\varphi))^2
\Big) \label{eq:av_main2}
\end{eqnarray}

We focus on the summand and note that by the $C_2-$inequality, one has
\begin{equation}\label{eq:av_main3}
T_1 + T_2 \geq  \check{\eta}_p^s(\{(D_{p,n}^f(\varphi)\otimes 1- 1\otimes D_{p,n}^c(\varphi))\}^2) 
\end{equation}
Our proof for the case when only (H\ref{hypav:1}-\ref{hypav:6}) holds, versus when (H\ref{hypav:1}-\ref{hypav:7}) holds is the same when considering
$T_1$ in \eqref{eq:av_main1}, versus $T_2$ in \eqref{eq:av_main2}. Thus our proof is split into three parts of controlling $T_1$ and then $T_2$ first under 
(H\ref{hypav:1}-\ref{hypav:6}) and then under (H\ref{hypav:1}-\ref{hypav:7}). The proof is concluded by adding and then summing the bounds (and then summing over $p$) for the relevant case.

{\bf Term} \eqref{eq:av_main1}: We have for any $x\in\mathsf{X}$
$$
(D_{p,n}^f(\varphi)(x)-D_{p,n}^c(\varphi)(x))^2 = 
$$
$$
D_{p,n}^f(\varphi)(x)[D_{p,n}^f(\varphi)(x)-D_{p,n}^c(\varphi)(x)]
+ D_{p,n}^c(\varphi)(x)[D_{p,n}^c(\varphi)(x)-D_{p,n}^f(\varphi)(x)].
$$
Applying Lemmata \ref{lem:av_1} and \ref{lem:av_3} yields the upper-bound:
$$
(D_{p,n}^f(\varphi)(x)-D_{p,n}^c(\varphi)(x))^2 \leq 
$$
$$
C\|\varphi\|\rho^{n-p}
\Big(
\|h_{p,n}^fS_{p,n}^{f,c}(\varphi)\|_{v^{\xi}} + |[\eta_n^f-\eta_n^c](\varphi)| + \|\varphi\|\|h_{p,n}^f-h_{p,n}^c\|_{v^{\xi}}\rho^{n-p}
\Big)v(x)^{4\xi}.
$$
Hence, as $\eta_p^f(v^{4\xi})\leq C$ \cite[Proposition 1]{whiteley} (the bound in that Proposition does not depend on $\eta_0^f$ by  (H\ref{hypav:2}) and (H\ref{hypav:5}) 2.), it follows that 
\begin{equation}\label{eq:av_main4}
|T_1| \leq C\|\varphi\| \rho^{n-p}\Big\{\|h_{p,n}^fS_{p,n}^{f,c}(\varphi)\|_{v^{\xi}} + |[\eta_n^f-\eta_n^c](\varphi)| + \|\varphi\|\|h_{p,n}^f-h_{p,n}^c\|_{v^{\xi}}\rho^{n-p}\Big\}.
\end{equation}

{\bf Term} \eqref{eq:av_main2}, case only (H\ref{hypav:1}-\ref{hypav:6}) hold: We have for any $(x,y)\in A$ (if $(x,y)\in A^c$ the term is zero)
$$
(D_{p,n}^c(\varphi)(x)-D_{p,n}^c(\varphi)(y))^2 = 
$$
$$
D_{p,n}^c(\varphi)(x)[D_{p,n}^c(\varphi)(x)-D_{p,n}^c(\varphi)(y)]
+ D_{p,n}^c(\varphi)(y)[D_{p,n}^c(\varphi)(y)-D_{p,n}^c(\varphi)(x)].
$$
Applying Lemmata \ref{lem:av_1} and \ref{lem:av_5} yields the upper-bound for any $(x,y)\in A$:
$$
(D_{p,n}^c(\varphi)(x)-D_{p,n}^c(\varphi)(y))^2 \leq 
C\|\varphi\|^2\rho^{n-p}(v(x)v(y))^{2\xi}(\rho^{n-p}+1)
$$
Hence it follows that 
\begin{equation}\label{eq:av_main5}
|T_2| \leq C\|\varphi\|^2 \rho^{n-p}\Big(\check{\eta}_p^{s}(\mathbb{I}_A(v\otimes v)^{2\xi})(\rho^{n-p}+1)\Big).
\end{equation}
Noting \eqref{eq:av_main3}, \eqref{eq:av_main4} and \eqref{eq:av_main5} the proof can be concluded in the case that only (H\ref{hypav:1}-\ref{hypav:6}) hold.

{\bf Term} \eqref{eq:av_main2}, case (H\ref{hypav:1}-\ref{hypav:7}) hold: We have by Lemma \ref{lem:av_6}
$$
(D_{p,n}^c(\varphi)(x)-D_{p,n}^c(\varphi)(y))^2 \leq C\|\varphi\|^2(\rho)^{n-p}\mathsf{d}(x,y)^2v(x)^{4\xi}v(y)^{8\xi}
$$
where $\rho$ is the square of the $\rho$ in Lemma \ref{lem:av_6}.
Then 
\begin{equation}\label{eq:av_main6}
|T_2| \leq C\|\varphi\|^2 \rho^{n-p}\Big(\check{\eta}_p^{s}(\mathsf{d}^2(v^{4\xi}\otimes v^{8\xi}))(\rho^{n-p}+1)\Big).
\end{equation}
Noting \eqref{eq:av_main3}, \eqref{eq:av_main4} and \eqref{eq:av_main6} the proof can be concluded in the case that (H\ref{hypav:1}-\ref{hypav:7}) hold.
\end{proof}

The following result is essentially \cite[Theorem 1]{whiteley}.

\begin{lem}\label{lem:av_1}
Assume (H\ref{hypav:1}-\ref{hypav:5}). Then for any $\xi\in(0,1]$ there exists $\rho<1$ and $C<+\infty$ also depending on the constants in  
(H\ref{hypav:1}-\ref{hypav:5}) such that for any $\varphi\in\mathcal{L}_{v^{\xi}}(\mathsf{X})$, $n\geq 1$ and $0\leq p <n$, $x\in\mathsf{X}$, $s\in\{f,c\}$ we have
$$
|D_{p,n}^s(\varphi)(x)| \leq C\|\varphi\|_{v^{\xi}}\rho^{n-p}v(x)^{\xi}.
$$
\end{lem}
\begin{proof}
The proof for $\xi=1$ is exactly as in \cite[Theorem 1]{whiteley}, except noting that (H\ref{hypav:2}) and (H\ref{hypav:5}) 2.~establish that the constants
in \cite[Theorem 1]{whiteley} do not depend on $\eta_0^s$ ($\mu$ in that paper). 
\end{proof}

\begin{lem}\label{lem:av_2}
Assume (H\ref{hypav:1}-\ref{hypav:6}). Then for any $\xi\in(0,1]$, $s\in\{f,c\}$ we have
$$
\sup_{n\geq 1}\sup_{0\leq p \leq n}\max_{s\in\{f,c\}}\|1/h_{p,n}^s\|_{v^{\xi}} < + \infty.
$$
\end{lem}
\begin{proof}
This is \cite[Lemma B.2.]{delm_sa}, except with the fix of using (H\ref{hypav:5}) 1.~to go to the second line of that proof.
\end{proof}

\begin{lem}\label{lem:av_3}
Assume (H\ref{hypav:1}-\ref{hypav:6}). Then for any $\xi\in(0,1]$ there exists a 
$\rho<1$ and $C<+\infty$ also depending on the constants in  
(H\ref{hypav:1}-\ref{hypav:6}) such that for any $\varphi\in\mathcal{B}_b(\mathsf{X})$, 
$n\geq 1$ and $0\leq p <n$ $x\in\mathsf{X}$ we have
$$
|D_{p,n}^f(\varphi)(x)-D_{p,n}^c(\varphi)(x)| \leq 
$$
$$
C \Big(
\|h_{p,n}^fS_{p,n}^{f,c}(\varphi)\|_{v^{\xi}} + |[\eta_n^f-\eta_n^c](\varphi)| + \|\varphi\|\|h_{p,n}^f-h_{p,n}^c\|_{v^{\xi}}\rho^{n-p}
\Big)v(x)^{3\xi}.
$$
\end{lem}

\begin{proof}
Define
\begin{eqnarray}
T_1 & := & \Big(S_{p,n}^{f,c}(\varphi)(x) + [\eta_n^f-\eta_n^c](\varphi)\Big)h_{p,n}^f(x)\label{eq_av:1}\\
T_2 & := & D_{p,n}^c(\varphi)(x)\frac{1}{h_{p,n}^c(x)}[h_{p,n}^f(x)-h_{p,n}^c(x)]\label{eq_av:2}
\end{eqnarray}
Then we have 
\begin{equation}\label{eq_av:5}
T_1 + T_2 = D_{p,n}^f(\varphi)(x)-D_{p,n}^c(\varphi)(x).
\end{equation}
We will bound $|T_1|$ and $|T_2|$ and then sum the bounds to conclude.

{\bf Term} \eqref{eq_av:1}: We have
$$
\Big(S_{p,n}^{f,c}(\varphi)(x) + [\eta_n^f-\eta_n^c](\varphi)\Big)h_{p,n}^f(x) = 
\Big(\frac{S_{p,n}^{f,c}(\varphi)(x)}{v(x)^{\xi}}v(x)^{\xi} + [\eta_n^f-\eta_n^c](\varphi)\Big)\frac{h_{p,n}^f(x)}{v(x)^{\xi}}v(x)^{\xi}.
$$
Then by \cite[Proposition 2]{whiteley} $\|h_{p,n}^f\|_{v^{\xi}}$ is upper-bounded by a finite constant that does not depend on $p,n,f$, so we easily have that
\begin{equation}\label{eq_av:3}
|T_1| \leq  C \Big(
\|h_{p,n}^fS_{p,n}^{f,c}(\varphi)\|_{v^{\xi}} + |[\eta_n^f-\eta_n^c](\varphi)| \Big)v(x)^{2\xi}.
\end{equation}

{\bf Term} \eqref{eq_av:2}: We have
$$
 D_{p,n}^c(\varphi)(x)\frac{1}{h_{p,n}^c(x)}[h_{p,n}^f(x)-h_{p,n}^c(x)] = 
$$
$$
\Big(D_{p,n}^c(\varphi)(x)\frac{1}{h_{p,n}^c(x)v(x)^{\xi}}\frac{[h_{p,n}^f(x)-h_{p,n}^c(x)]}{v(x)^{\xi}}\Big)v(x)^{2\xi}.
$$
Applying Lemma \ref{lem:av_1} to $D_{p,n}^c(\varphi)(x)$ and Lemma \ref{lem:av_2} to $1/(h_{p,n}^c(x)v(x)^{\xi})$ yields that
\begin{equation}\label{eq_av:4}
|T_2| \leq C\|\varphi\|\|h_{p,n}^f-h_{p,n}^c\|_{v^{\xi}}\rho^{n-p}v(x)^{3\xi}.
\end{equation}
Noting \eqref{eq_av:5}, \eqref{eq_av:3} and \eqref{eq_av:4} the proof can be concluded.
\end{proof}

Define for $n\geq 1$, $0\leq p < n$, $x\in\mathsf{X}$, $s\in\{f,c\}$
\begin{equation}\label{eq:av_rdef}
R_{p+1}^{(n),s}(x,dy) := \frac{1}{Q_{p,n}^s(1)(x)} Q_{p+1,n}^s(1)(y) Q_{p+1}^s(x,dy).
\end{equation}

\begin{lem}\label{lem:new_av_1}
Assume (H\ref{hypav:1}-\ref{hypav:5}), (H\ref{hypav:7}). Then for any $\xi\in(0,1]$ there exists a $C<+\infty$
depending on the constants in  (H\ref{hypav:1}-\ref{hypav:5}), (H\ref{hypav:7}), such that for any $\varphi\in\mathcal{B}_b(\mathsf{X})\cap\textrm{\emph{Lip}}_{\mathsf{d}}(\mathsf{X})$, 
$(x,y)\in\mathsf{X}\times\mathsf{X}$, $s\in\{f,c\}$ we have
$$
\sup_{n\geq 1}\sup_{0\leq p\leq n}\Big|
\frac{Q_{p,n}^s(\varphi)(x)}{\eta_p^s(Q_{p,n}^s(1))}
-\frac{Q_{p,n}^s(\varphi)(y)}{\eta_p^s(Q_{p,n}^s(1))}\Big| \leq C\|\varphi\|\mathsf{d}(x,y)v(x)^{\xi}v(y)^{\xi}.
$$
\end{lem}

\begin{proof}
We proceed by backward induction, starting with the case $p=n-1$ (the case $n=p$ is trivial). We have
$$
\frac{Q_{n-1,n}^s(\varphi)(x)}{\eta_{n-1}^s(Q_{n-1,n}^s(1))}
-\frac{Q_{n-1,n}^s(\varphi)(y)}{\eta_{n-1}^s(Q_{n-1,n}^s(1))} = 
\frac{Q_{n}^s(\varphi)(x)}{\eta_{n-1}^s(G_{n-1})}
-\frac{Q_{n}^s(\varphi)(y)}{\eta_{n-1}^s(G_{n-1})} .
$$
By \cite[Proposition 2 (2)]{whiteley}, 
\begin{equation}\label{eq_av:10}
\inf_{p\geq 0}\min_{s\in\{f,c\}}\eta_p^s(G_p) \geq C>0
\end{equation}
for some constant $C$. So applying \eqref{eq_av:10} along with (H\ref{hypav:7}) we have
$$
\Big|\frac{Q_{n-1,n}^s(\varphi)(x)}{\eta_{n-1}^s(Q_{n-1,n}^s(1))}
-\frac{Q_{n-1,n}^s(\varphi)(y)}{\eta_{n-1}^s(Q_{n-1,n}^s(1))}\Big| \leq C\|\varphi\|\mathsf{d}(x,y)v(x)^{\xi}v(y)^{\xi}.
$$
Now assuming the result for some $p+1<n$, we have
$$
\frac{Q_{p,n}^s(\varphi)(x)}{\eta_p^s(Q_{p,n}^s(1))}
-\frac{Q_{p,n}^s(\varphi)(y)}{\eta_p^s(Q_{p,n}^s(1))} = \frac{1}{\eta_p^s(G_p)}\Big(
\frac{Q_{p+1}^s(Q_{p+1,n}^s(\varphi))(x)}{\eta_{p+1}^s(Q_{p+1,n}^s(1))} - 
\frac{Q_{p+1}^s(Q_{p+1,n}^s(\varphi))(y)}{\eta_{p+1}^s(Q_{p+1,n}^s(1))}.
\Big)
$$
Now we have, for any $x\in\mathsf{X}$
$$
\Big|\frac{Q_{p+1,n}^s(\varphi))(x)}{\eta_{p+1}^s(Q_{p+1,n}^s(1))}\Big|\frac{1}{v(x)^{\xi}} \leq \|\varphi\|\|h_{p+1,n}^s\|_{v^{\xi}} \leq C\|\varphi\|
$$
where we have used \cite[Proposition 2 (3)]{whiteley} to get to the second inequality.
Then applying \eqref{eq_av:10}, the induction hypothesis and  (H\ref{hypav:7}) we have
\begin{eqnarray*}
\Big|
\frac{Q_{p,n}^s(\varphi)(x)}{\eta_p^s(Q_{p,n}^s(1))}
-\frac{Q_{p,n}^s(\varphi)(y)}{\eta_p^s(Q_{p,n}^s(1))}\Big| & \leq & C\Big\|\frac{Q_{p+1,n}^s(\varphi))}{\eta_{p+1}^s(Q_{p+1,n}^s(1))}\Big\|_{v^{\xi}}\mathsf{d}(x,y)v(x)^{\xi}v(y)^{\xi} \\
& \leq & C\|\varphi\|\mathsf{d}(x,y)v(x)^{\xi}v(y)^{\xi}.
\end{eqnarray*}
\end{proof}

\begin{lem}\label{lem:av_4}
Assume (H\ref{hypav:1}-\ref{hypav:7}). Then for any $\xi\in(0,1/2)$  there exists a 
$C<+\infty$ depending on the constants in  
(H\ref{hypav:1}-\ref{hypav:7}) such that for any $\varphi\in\mathcal{L}_{v^{\xi}}(\mathsf{X})\cap\textrm{\emph{Lip}}_{v^{\xi},\mathsf{d}}(\mathsf{X})$, 
$n\geq 1$ and $0\leq p <n$ $(x,y)\in\mathsf{X}\times \mathsf{X}$, $s\in\{f,c\}$ we have
$$
\Big|h_{p,n}^s(x)\Big(R_{p+1}^{(n),s}(\varphi)(x)-R_{p+1}^{(n),s}(\varphi)(y)\Big)\Big| \leq C\|\varphi\|_{v^{\xi}}
\mathsf{d}(x,y)v(x)^{2\xi}v(y)^{4\xi}.
$$
\end{lem}
\begin{proof}
Define
\begin{eqnarray}
T_1 & := & \frac{h_{p,n}^s(x)}{Q_{p,n}^s(1)(x)}\Big(Q_{p+1}^s(\varphi Q_{p+1,n}^s(1))(x)-Q_{p+1}^s(\varphi Q_{p+1,n}^s(1))(y)\Big) \label{eq_av:6}\\
T_2 & := & \frac{h_{p,n}^s(x)Q_{p+1}^s(\varphi Q_{p+1,n}^s(1))(y)}{Q_{p,n}^s(1)(x)Q_{p,n}^s(1)(y)}\Big(Q_{p,n}^s(1)(y)-Q_{p,n}^s(1)(x)\Big)\label{eq_av:7}
\end{eqnarray}
Then we have 
\begin{equation}\label{eq_av:8}
T_1 + T_2 =h_{p,n}^s(x)\Big(R_{p+1}^{(n),s}(\varphi)(x)-R_{p+1}^{(n),s}(\varphi)(y)\Big).
\end{equation}
We will bound $|T_1|$ and $|T_2|$ and then sum the bounds to conclude.

{\bf Term} \eqref{eq_av:6}: We have
$$
\frac{h_{p,n}^s(x)}{Q_{p,n}^s(1)(x)}\Big(Q_{p+1}^s(\varphi Q_{p+1,n}^s(1))(x)-Q_{p+1}^s(\varphi Q_{p+1,n}^s(1))(y)\Big)
= 
$$
$$
\frac{1}{\eta_p^s(G_p)}\Big(Q_{p+1}^s(\varphi h_{p+1,n}^s)(x)-Q_{p+1}^s(\varphi h_{p+1,n}^s)(y)\Big).
$$
Then by Lemma \ref{lem:new_av_1} and \cite[Proposition 2 (3)]{whiteley}, 
\begin{equation}\label{eq:av_new1}
h_{p+1,n}^s\in\mathcal{L}_{v^{\tilde{\xi}}}(\mathsf{X})\cap\textrm{Lip}_{v^{\tilde{\xi}},\mathsf{d}}(\mathsf{X})~\quad\forall\tilde{\xi}\in(0,1]
\end{equation}
so, applying \eqref{eq_av:10} and (H\ref{hypav:7})
we have
$$
|T_1| \leq C \|\varphi h_{p+1,n}^s\|_{v^{2\xi}}\mathsf{d}(x,y)[v(x)v(y)]^{2\xi}
$$
By \cite[Proposition 2 (3)]{whiteley}
$$
\|\varphi h_{p+1,n}^s\|_{v^{2\xi}} \leq C\|\varphi\|_{v^{\xi}}
$$
so
\begin{equation}\label{eq_av:9}
|T_1| \leq C \|\varphi\|_{v^{\xi}}\mathsf{d}(x,y)[v(x)v(y)]^{2\xi}.
\end{equation}

{\bf Term} \eqref{eq_av:7}:
We have
$$
\frac{h_{p,n}^s(x)Q_{p+1}^s(\varphi Q_{p+1,n}^s(1))(y)}{Q_{p,n}^s(1)(x)Q_{p,n}^s(1)(y)}\Big(Q_{p,n}^s(1)(y)-Q_{p,n}^s(1)(x)\Big) = 
$$
$$
\frac{Q_{p+1}^s(\varphi Q_{p+1,n}^s(1))(y)}{Q_{p,n}^s(1)(y)}\Big(Q_{p+1}^s(h_{p+1,n}^s)(y)-Q_{p+1}^s(h_{p+1,n}^s)(x)\Big) = 
$$
$$
\frac{Q_{p+1}^s(\varphi h_{p+1,n}^s)(y)v(y)^{\xi}}{\eta_p^s(G_p) h_{p,n}^s(y)v(y)^{\xi}}\Big(Q_{p+1}^s(h_{p+1,n}^s)(y)-Q_{p+1}^s(h_{p+1,n}^s)(x)\Big).
$$
Applying \eqref{eq_av:10} and Lemma \ref{lem:av_2} gives the upper-bound
$$
|T_2| \leq C Q_{p+1}^s(|\varphi| h_{p+1,n}^s)(y)v(y)^{\xi}\Big|Q_{p+1}^s(h_{p+1,n}^s)(y)-Q_{p+1}^s(h_{p+1,n}^s)(x)\Big|.
$$
Now $\varphi\in\mathcal{L}_{v^{\xi}}(\mathsf{X})$ and by \cite[Proposition 2 (3)]{whiteley} $\sup_{n\geq 1}\sup_{0\leq p \leq n}\max_{s\in\{f,c\}}\|h_{p+1,n}^s\|_{v^{\xi}}<+\infty$,
then, noting \eqref{eq:av_new1} and thus applying (H\ref{hypav:7})
$$
|T_2| \leq C \|\varphi\|_{v^{\xi}} Q_{p+1}^s(v^{2\xi})(y)v(y)^{\xi}\|h_{p+1,n}^s\|_{v^{\xi}}|x-y|[v(x)v(y)]^{\xi}.
$$
Again applying \cite[Proposition 2 (3)]{whiteley} and (H\ref{hypav:1}) we have finally that
\begin{equation}\label{eq_av:11}
|T_2| \leq C\|\varphi\|_{v^{\xi}}\mathsf{d}(x,y)v(x)^{\xi}v(y)^{4\xi}.
\end{equation}
Noting \eqref{eq_av:8}, \eqref{eq_av:9} and \eqref{eq_av:11} the proof can be concluded.
\end{proof}

Recall $A=\{(x,y)\in\mathsf{X}\times\mathsf{X}:x\neq y\}$.

\begin{lem}\label{lem:av_5}
Assume (H\ref{hypav:1}-\ref{hypav:6}). Then for any $\xi\in(0,1]$ there exists a $\rho<1$ and  $C<+\infty$ depending on the constants in  
(H\ref{hypav:1}-\ref{hypav:6}) such that for any $\varphi\in\mathcal{B}_b(\mathsf{X})$, $n\geq 1$ and $0\leq p <n$, $(x,y)\in\mathsf{X}\times\mathsf{X}$ we have
$$
| D_{p,n}^c(\varphi)(x)- D_{p,n}^c(\varphi)(y)| \leq 
C\mathbb{I}_A(x,y)\|\varphi\|\Big(\rho^{n-p}+1\Big)v(x)^{\xi}v(y)^{2\xi}.
$$
\end{lem}

\begin{proof}
Throughout, we assume that $x\neq y$. Define
\begin{eqnarray}
T_1 & := & \Big(\frac{Q_{p,n}^c(\varphi)(x)}{Q_{p,n}^c(1)(x)}-\frac{Q_{p,n}^c(\varphi)(y)}{Q_{p,n}^c(1)(y)}\Big)h_{p,n}^c(x)\label{eq_av:12}\\
T_2 & := & D_{p,n}^c(\varphi)(y)\Big(\frac{h_{p.n}^c(x)}{h_{p.n}^c(y)}-1\Big)\label{eq_av:13}
\end{eqnarray}
then we have 
\begin{equation}\label{eq_av:14}
T_1 + T_2 = D_{p,n}^c(\varphi)(x)- D_{p,n}^c(\varphi)(y).
\end{equation}
We will bound $|T_1|$ and $|T_2|$ and then sum the bounds to conclude.

{\bf Term} \eqref{eq_av:12}: We have
$$
 \Big(\frac{Q_{p,n}^c(\varphi)(x)}{Q_{p,n}^c(1)(x)}-\frac{Q_{p,n}^c(\varphi)(y)}{Q_{p,n}^c(1)(y)}\Big)h_{p,n}^c(x) = 
$$
$$
\frac{1}{\eta_p^c(G_p)}\Big(Q_{p+1}^c\Big(
\frac{Q_{p+1,n}^c(\varphi)}{\eta_{p+1}^c(Q_{p+1,n}^c(1))}
\Big))(x)-
Q_{p+1}^c\Big(\frac{Q_{p+1,n}^c(\varphi)}{\eta_{p+1}^c(Q_{p+1,n}^c(1))}
\Big)(y)\Big) +
$$
$$
\frac{Q_{p,n}^c(\varphi)(y)}{Q_{p,n}^c(1)(y)}\frac{1}{\eta_p^c(G_p)}\Big(Q_{p+1}^c(h_{p+1,n}^c)(y)-Q_{p+1}^c(h_{p+1,n}^c)(x)\Big).
$$
Applying \eqref{eq_av:10} and $\varphi\in\mathcal{B}_b(\mathsf{X})$ we have
$$
|T_1|  \leq C\|\varphi\|(|Q_{p+1}^c(h_{p+1,n}^c)(x)|+|Q_{p+1}^c(h_{p+1,n}^c)(y)|).
$$
By \cite[Proposition 2 (3)]{whiteley} $\sup_{n\geq 1}\sup_{0\leq p \leq n}\max_{s\in\{f,c\}}\|h_{p+1,n}^s\|_{v^{\xi}}<+\infty$,
then in addition applying (H\ref{hypav:1}) we have
\begin{equation}\label{eq_av:15}
|T_1|  \leq C\|\varphi\|(v(x)^{\xi}+v(y)^{\xi}).
\end{equation}

{\bf Term} \eqref{eq_av:13}: We have
$$
D_{p,n}^c(\varphi)(y)\Big(\frac{h_{p.n}^c(x)}{h_{p.n}^c(y)}-1\Big) = D_{p,n}^c(\varphi)(y)\frac{1}{h_{p.n}^c(y)v(y)^{\xi}}
\frac{h_{p.n}^c(x)-h_{p.n}^c(y)}{v(x)^{\xi}v(y)^{\xi}}v(x)^{\xi}v(y)^{2\xi}.
$$
Then applying Lemma \ref{lem:av_1} for $D_{p,n}^c(\varphi)(y)$, Lemma \ref{lem:av_2} for $1/(h_{p.n}^c(y)v(y)^{\xi})$
and \cite[Proposition 2 (3)]{whiteley} as above, we have
\begin{equation}\label{eq_av:16}
|T_2| \leq C\|\varphi\|\rho^{n-p}v(x)^{\xi}v(y)^{3\xi}.
\end{equation}
Noting \eqref{eq_av:14}, \eqref{eq_av:15} and \eqref{eq_av:16} the proof can be concluded.
\end{proof}

Define for $B\in\mathcal{X}$, $x\in\mathsf{X}$, $s\in\{f,c\}$
$$
P_{p,n}^s(B)(x) = \frac{Q_{p,n}^s(B)(x)}{Q_{p,n}^s(1)(x)}.
$$

\begin{lem}\label{lem:new_av_2}
Assume (H\ref{hypav:1}-\ref{hypav:7}). Then for any $\xi\in(0,1)$ there exists a $C<+\infty$
depending on the constants in  (H\ref{hypav:1}-\ref{hypav:7}) such that for any $\varphi\in\mathcal{B}_b(\mathsf{X})\cap\textrm{\emph{Lip}}_{\mathsf{d}}(\mathsf{X})$, 
$(x,y)\in\mathsf{X}\times\mathsf{X}$, $s\in\{f,c\}$ we have
$$
\sup_{n\geq 1}\sup_{0\leq p\leq n}|P_{p,n}^s(\varphi)(x)-P_{p,n}^s(\varphi)(y)| \leq C\|\varphi\|\mathsf{d}(x,y)v(x)^{\xi}v(y)^{\xi}.
$$
\end{lem}

\begin{proof}
We assume $p<n$ as the case $p=n$ is trivial. Define
\begin{eqnarray}
T_1 & = & \frac{1}{h_{p,n}^s(x)}\Big(\frac{Q_{p,n}^s(\varphi)(x)}{\eta_p^s(Q_{p,n}^s(1))}
-\frac{Q_{p,n}^s(\varphi)(y)}{\eta_p^s(Q_{p,n}^s(1))}\Big) \label{eq:av_new_11}\\
T_2 & = & \frac{Q_{p,n}^s(\varphi)(y)}{Q_{p,n}^s(1)(y)}\frac{1}{h_{p,n}^s(x)}
\Big(\frac{Q_{p,n}^s(1)(y)}{\eta_p^s(Q_{p,n}^s(1))}
-\frac{Q_{p,n}^s(1)(x)}{\eta_p^s(Q_{p,n}^s(1))}\Big) \label{eq:av_new_12}
\end{eqnarray}
then we have 
\begin{equation}\label{eq:av_new_13}
T_1 + T_2 = P_{p,n}^s(\varphi)(x)-P_{p,n}^s(\varphi)(y) 
\end{equation}
We will bound $|T_1|$ and $|T_2|$ and then sum the bounds to conclude.

{\bf Term} \eqref{eq:av_new_11}: We have, for any $\tilde{\xi}\in(0,1]$
$$
\frac{1}{h_{p,n}^s(x)}\Big(\frac{Q_{p,n}^s(\varphi)(x)}{\eta_p^s(Q_{p,n}^s(1))}
-\frac{Q_{p,n}^s(\varphi)(y)}{\eta_p^s(Q_{p,n}^s(1))}\Big)
=
\frac{v(x)^{\tilde{\xi}}}{h_{p,n}^s(x)v(x)^{\tilde{\xi}}}\Big(\frac{Q_{p,n}^s(\varphi)(x)}{\eta_p^s(Q_{p,n}^s(1))}
-\frac{Q_{p,n}^s(\varphi)(y)}{\eta_p^s(Q_{p,n}^s(1))}\Big).
$$
Then applying Lemmata \ref{lem:av_2} and \ref{lem:new_av_2}, we have the upper-bound for any $\tilde{\xi}\in(0,1]$
\begin{equation}\label{eq:av_new_14}
|T_1| \leq C \|\varphi\|\mathsf{d}(x,y)v(x)^{2\tilde{\xi}}v(y)^{\tilde{\xi}}.
\end{equation}

{\bf Term} \eqref{eq:av_new_12}: We have, for any $\tilde{\xi}\in(0,1]$
$$
\frac{Q_{p,n}^s(\varphi)(y)}{Q_{p,n}^s(1)(y)}\frac{1}{h_{p,n}^s(x)}
\Big(\frac{Q_{p,n}^s(1)(y)}{\eta_p^s(Q_{p,n}^s(1))}
-\frac{Q_{p,n}^s(1)(x)}{\eta_p^s(Q_{p,n}^s(1))}\Big) = 
$$
$$
\frac{Q_{p,n}^s(\varphi)(y)}{Q_{p,n}^s(1)(y)}\frac{v(x)^{\tilde{\xi}}}{h_{p,n}^s(x)v(x)^{\tilde{\xi}}}
\Big(\frac{Q_{p,n}^s(1)(y)}{\eta_p^s(Q_{p,n}^s(1))}
-\frac{Q_{p,n}^s(1)(x)}{\eta_p^s(Q_{p,n}^s(1))}\Big) 
$$
Then applying, $\varphi\in\mathcal{B}_b(\mathsf{X})$, Lemmata \ref{lem:av_2} and \ref{lem:new_av_2}, we have the upper-bound for any $\tilde{\xi}\in(0,1]$
\begin{equation}\label{eq:av_new_15}
|T_2| \leq C \|\varphi\|\mathsf{d}(x,y)v(x)^{2\tilde{\xi}}v(y)^{\tilde{\xi}}.
\end{equation}
Noting \eqref{eq:av_new_13}, \eqref{eq:av_new_14} and \eqref{eq:av_new_15} (the latter two with $\tilde{\xi}=\xi/2$) the proof can be concluded.
\end{proof}

\begin{lem}\label{lem:av_6}
Assume (H\ref{hypav:1}-\ref{hypav:7}). Then for any $\xi\in(0,1/2)$ there exists a $\rho<1$ and  $C<+\infty$ depending on the constants in  
(H\ref{hypav:1}-\ref{hypav:7}) such that for any $\varphi\in\mathcal{B}_b(\mathsf{X})\cap\textrm{\emph{Lip}}_{\mathsf{d}}(\mathsf{X})$, $n\geq 1$ and $0\leq p <n$, $(x,y)\in\mathsf{X}\times\mathsf{X}$ we have
$$
| D_{p,n}^c(\varphi)(x)- D_{p,n}^c(\varphi)(y)| \leq C\|\varphi\|\rho^{n-p}\mathsf{d}(x,y)v(x)^{2\xi}v(y)^{4\xi}.
$$
\end{lem}
\begin{proof}
The proof uses the same decomposition \eqref{eq_av:14} as in the proof of Lemma \ref{lem:av_5}. We use the same bound \eqref{eq_av:16} as 
in the proof of Lemma \ref{lem:av_5} and focus on the term $T_1$ as defined in \eqref{eq_av:12}. 
We note by \cite[pp.~133]{delm:04}
$$
P_{p,n}^c(\varphi)(x) = R_{p+1}^{(n),c}(P_{p+1,n}^c(\varphi))(x).
$$
where $R_{p+1}^{(n),c}$ is defined in \eqref{eq:av_rdef}.

Now $T_1$ as in the proof of Lemma \ref{lem:av_5}, can be written as
\begin{eqnarray*}
T_1 & = & h_{p,n}^c(x)\Big(P_{p,n}^c(\varphi)(x)-P_{p,n}^c(\varphi)(y)\Big) \\
& = & h_{p,n}^c(x)\Big(P_{p,n}^c(\varphi-\eta_n^c(\varphi))(x)-P_{p,n}^c(\varphi-\eta_n^c(\varphi))(y)\Big) \\
& = & h_{p,n}^c(x)\Big(R_{p+1}^{(n),s}(P_{p+1,n}^c(\varphi-\eta_n^c(\varphi)))(x)-
R_{p+1}^{(n),s}(P_{p+1,n}^c(\varphi-\eta_n^c(\varphi)))(y)
\Big).
\end{eqnarray*}
Now, noting Lemma \ref{lem:new_av_2}, we apply Lemma \ref{lem:av_4} to yield the upper-bound
\begin{equation}\label{eq_av:17}
|T_1| \leq C\Big\|(P_{p+1,n}^c(\varphi-\eta_n^c(\varphi))\Big\|_{v^{\xi}}
\mathsf{d}(x,y)v(x)^{2\xi}v(y)^{4\xi}.
\end{equation}
Now
$$
\frac{(P_{p+1,n}^c(\varphi-\eta_n^c(\varphi))(x)}{v(x)^{\xi}} = \frac{D_{p+1,n}^c(\varphi)(x)}{v(x)^{\xi/2}}\frac{1}{h_{p+1,n}^c(x)v(x)^{\xi/2}}.
$$
Applying  Lemma \ref{lem:av_1} and Lemma \ref{lem:av_2} yields that
\begin{equation}\label{eq:diff_need}
\Big\|(P_{p+1,n}^c(\varphi-\eta_n^c(\varphi))\Big\|_{v^{\xi}} \leq
C\|\varphi\|\rho^{n-p}.
\end{equation}
Then using the above inequality in \eqref{eq_av:17} and noting \eqref{eq_av:15} and \eqref{eq_av:16} the proof can be concluded.
\end{proof}

\section{Proofs for the Diffusion Case}\label{app:diff_case}

The statements at the beginning of Appendix \ref{app:av_proofs} also apply here. This appendix should be read after 
Appendix \ref{app:av_proofs}.

\subsection{Proofs for Proposition \ref{prop:diff_general_b_bound}}\label{app:diff_case_b}

\begin{proof}[Proof of Proposition \ref{prop:diff_general_b_bound}]
Follows easily from Propositions \ref{prop:diff_1}, \ref{prop:diff_2} and \ref{prop:diff_3}.
\end{proof}

\begin{lem}\label{lem:diff_1}
Assume (H\ref{hypav:8}). Then for any $\xi\in(0,1/2)$ there exists a $C<+\infty$ depending on the constants in (H\ref{hypav:8}) such that
for any $\varphi\in\mathcal{L}_{v^{\xi}}(\mathsf{X})$, $1\leq l \leq L$, $x\in\mathsf{X}$ we have
$$
|M^l(\varphi)(x)-M^{l-1}(\varphi)(x)|  \leq  C\|\varphi\|_{v^{\xi}}\Delta_lv(x)^{2\xi}.
$$
\end{lem}

\begin{proof}
Define
\begin{eqnarray}
T_1 & := & \int_{\mathsf{B}_l(x)}\varphi(y)[M^l(x,y)-M(x,y)-\{M^{l-1}(x,y)-M(x,y)\}]dy \label{eq:diff_3} \\
T_2 & := & \int_{\mathsf{B}_l(x)^c}\varphi(y)[M^l(x,y)+M(x,y)-\{M^{l-1}(x,y)+M(x,y)\}]dy. \label{eq:diff_4}
\end{eqnarray}
Then we have
\begin{equation}\label{eq:diff_5}
T_1 + T_2 = M^l(\varphi)(x)-M^{l-1}(\varphi)(x)
\end{equation}
We will bound $|T_1|$ and $|T_2|$ and then sum the bounds to conclude.

{\bf Term} \eqref{eq:diff_3}: We have, by applying \eqref{eq:diff_2},
$$
|\int_{\mathsf{B}_l(x)}\varphi(y)[M^l(x,y)-M(x,y)-\{M^{l-1}(x,y)-M(x,y)\}]dy| \leq 
$$
$$
C\Delta_l \int_{\mathsf{B}_l(x)} |\varphi(y)|\exp\{-C'\|y-x\|^2\}dy.
$$
Using $\varphi\in\mathcal{L}_{v^{\xi}}(\mathsf{X})$ and (H\ref{hypav:8}) 2.~we have
\begin{equation}\label{eq:diff_6}
|T_1| \leq C\|\varphi\|_{v^{\xi}}\Delta_lv(x)^{2\xi}.
\end{equation}

{\bf Term} \eqref{eq:diff_4}: We have, by applying \eqref{eq:diff_1}
$$
|\int_{\mathsf{B}_l(x)^c}\varphi(y)[M^l(x,y)+M(x,y)-\{M^{l-1}(x,y)+M(x,y)\}]dy| \leq 
$$
$$
C \int_{\mathsf{B}_l(x)^c} |\varphi(y)| \exp\{-C'\|y-x\|^2\}dy
$$
Using $\varphi\in\mathcal{L}_{v^{\xi}}(\mathsf{X})$, the fact 
$$
\int_{\mathsf{B}_l(x)^c}v(y)^{\xi}\exp\{-C'\|y-x\|^2\}dy \leq C \int_{B_l(x)^c}v(y)^{\xi}dy
$$
along with (H\ref{hypav:8}) and that $\int_{\mathsf{B}_l(x)^c}dy\leq C\Delta_l$, we have
\begin{equation}\label{eq:diff_7}
|T_2| \leq C\|\varphi\|_{v^{\xi}}\Delta_lv(x)^{2\xi}.
\end{equation}
Noting \eqref{eq:diff_5}, \eqref{eq:diff_6} and \eqref{eq:diff_7} the proof can be concluded.
\end{proof}

Let $\xi\in(0,1]$ and denote $\mathscr{P}_{v^{\xi}}(\mathsf{X})$ as the collection of probabiity measures for which $\mu(v^{\xi})<+\infty$.

\begin{lem}\label{lem:diff_2}
Assume (H\ref{hypav:5}) 1.~(H\ref{hypav:8}). Then for any $\xi\in(0,1/2)$ there exists a $C<+\infty$ depending on the constants in (H\ref{hypav:5}) 1.~(H\ref{hypav:8})
such that for any $\mu\in\mathscr{P}_{v^{\xi}}(\mathsf{X})$, $\varphi\in\mathcal{L}_{v^{\xi}}(\mathsf{X})$, $p\geq 1$, $1\leq l \leq L$ we have
$$
|[\Phi_p^l(\mu)-\Phi_p^{l-1}(\mu)](\varphi)| \leq \frac{C\|\varphi\|_{v^{\xi}}\Delta_l\mu(v^{2\xi})}{\mu(G_{p-1})}.
$$
\end{lem}

\begin{proof}
We have for $\varphi\in\mathcal{L}_{v^{\xi}}(\mathsf{X})$
$$
[\Phi_p^l(\mu)-\Phi_p^{l-1}(\mu)](\varphi) = \frac{1}{\mu(G_{p-1})}\mu(G_{p-1}\{M^l(\varphi)-M^{l-1}(\varphi)\}).
$$
Applying (H\ref{hypav:5}) 1.~along with Lemma \ref{lem:diff_1} allows one to conclude.
\end{proof}

\begin{lem}\label{lem:diff_3}
Assume (H\ref{hypav:1}-\ref{hypav:5}). Then for any $\xi\in(0,1]$ there exists a $C<+\infty$ depending on the constants in (H\ref{hypav:1}-\ref{hypav:5}) such that
for any $1\leq l \leq L$ we have
\begin{eqnarray}
\sup_{n\geq 1}\sup_{0\leq p \leq n}\Big\|\frac{Q_{p,n}^{l-1}(1)}{\eta_p^{l}(Q_{p,n}^{l-1}(1))}\Big\|_{v^{\xi}}  & \leq & C \label{eq:diff_8}\\
\sup_{n\geq 1}\sup_{0\leq p \leq n}\Big\|\frac{Q_{p,n}^{l-1}(1)}{\Phi^{l-1}_p(\eta_{p-1}^l)(Q_{p,n}^{l-1}(1))}\Big\|_{v^{\xi}}  & \leq & C \label{eq:diff_9}\\
\sup_{n\geq 1}\sup_{0\leq p \leq n}\Big\|\frac{Q_{p,n}^{l-1}(1)}{\Phi^{l}_p(\eta_{p-1}^l)(Q_{p,n}^{l-1}(1))}\Big\|_{v^{\xi}}  & \leq & C \label{eq:diff_17}
\end{eqnarray}
and if additionally $\xi\in(0,1/2)$ there exists a $C<+\infty$ depending on the constants in (H\ref{hypav:1}-\ref{hypav:5}) such that
for any $\varphi\in \mathcal{L}_{v^{\xi}}(\mathsf{X})$, $l\geq 1$ we have
\begin{equation}
\sup_{n\geq 1}\sup_{0\leq p \leq n}\Big\|\frac{Q_{p,n}^{l-1}(\varphi)}{\eta_p^{l}(Q_{p,n}^{l-1}(1))}\Big\|_{v^{2\xi}}  \leq  C \|\varphi\|_{v^{\xi}}.\label{eq:diff_10}
\end{equation}
\end{lem}

\begin{proof}
The proof of \eqref{eq:diff_8} (resp.~\eqref{eq:diff_9}) and \eqref{eq:diff_10} (resp.~\eqref{eq:diff_17}) are very similar, so we only give the proofs of \eqref{eq:diff_9} and \eqref{eq:diff_10}.

{\bf Proof of} \eqref{eq:diff_9}: We have for any $x\in\mathsf{X}$
$$
\frac{Q_{p,n}^{l-1}(1)(x)}{\Phi^{l-1}_p(\eta_{p-1}^l)(Q_{p,n}^{l-1}(1))} = \frac{h_{p,n}^{l-1}(x)}{\Phi^{l-1}_p(\eta_{p-1}^l)(h_{p,n}^{l-1})}.
$$
Now
\begin{eqnarray}
\Phi^{l-1}_p(\eta_{p-1}^l)(h_{p,n}^{l-1}) & = & \frac{\eta_{p-1}^l(Q_p^{l-1}(h_{p,n}^{l-1}))}{\eta_{p-1}^l(G_{p-1})} \nonumber\\
& \geq & \frac{\eta_{p-1}^l(\mathbb{I}_{C_d}Q_p^{l-1}(\mathbb{I}_{C_d}h_{p,n}^{l-1}))}{\sup_{n\geq 0}\|G_n\|} \nonumber\\
& \geq & C\eta_{p-1}^l(C_d)\tilde{\epsilon}_d^{-}\nu(\mathbb{I}_{C_d}h_{p,n}^{l-1}) \nonumber\\
& \geq & C \nu(C_d)\inf_{n\geq 1}\inf_{0\leq p \leq n}\inf_{x\in C_d}h_{p,n}^{l-1} \nonumber\\
& \geq & C\label{eq:diff_11}
\end{eqnarray}
where we have applied (H\ref{hypav:5}) 1.~and (H\ref{hypav:3}) to go to line 3, \cite[pp.~2527]{whiteley} (last displayed equation) to go line 4 and
then \cite[Lemma 10]{whiteley} to go to the final line. Then using \cite[Proposition 2]{whiteley}, it thus follows that for any $x\in\mathsf{X}$
$$
\Big|\frac{Q_{p,n}^{l-1}(1)(x)}{\Phi^{l-1}_p(\eta_{p-1}^l)(Q_{p,n}^{l-1}(1))}\Big| \leq C v(x)^{\xi}
$$
which completes the proof of \eqref{eq:diff_9}.

{\bf Proof of} \eqref{eq:diff_10}: We have for any $x\in\mathsf{X}$
$$
\frac{Q_{p,n}^{l-1}(\varphi)(x)}{\eta_p^{l}(Q_{p,n}^{l-1}(1))} = \frac{Q_{p,n}^{l-1}(\varphi)(x)}{Q_{p,n}^{l-1}(1)(x)}\frac{Q_{p,n}^{l-1}(1)(x)}{\eta_p^{l-1}(Q_{p,n}^{l-1}(1))}
\frac{1}{\eta_p^{l}(h_{p,n}^{l-1})}
$$
By using a very similar proof to \eqref{eq:diff_11}, one can deduce that $\eta_p^{l}(h_{p,n}^{l-1})\geq C$, as frequently noted $\|h_{p,n}^{l-1}\|_{v^{\xi}}\leq C$ and by a similar proof to \cite[Proposition 1]{whiteley} $|Q_{p,n}^{l-1}(\varphi)(x)|/Q_{p,n}^{l-1}(1)(x) \leq C \|\varphi\|_{v^{\xi}}v(x)^{\xi}$. Hence one can complete the proof of \eqref{eq:diff_10}.
\end{proof}

Let $\mu\in\mathscr{P}(\mathsf{X})$, $s\in\{l,l-1\}$, $l\geq 1$, $n\geq 1$, $0\leq p \leq n$, and denote $\Phi_{(p,n)}^s(\mu) = \Phi_n^s\circ\cdots\circ \Phi_{p+1}^s(\mu)$, with the convention that if $n=p$, $\Phi_{(p,n)}^s(\mu)=\mu$ and $\Phi_0^s(\mu)=\mu$.

\begin{lem}\label{lem:diff_4}
Assume (H\ref{hypav:1}-\ref{hypav:5}), (H\ref{hypav:8}). Then for any $\xi\in(0,1/4)$ there exists a $C<+\infty$ depending on the constants in (H\ref{hypav:1}-\ref{hypav:5}) (H\ref{hypav:8}) such that for any $\varphi\in \mathcal{L}_{v^{\xi}}(\mathsf{X})$, $n\geq 1$, $1\leq l \leq L$ we have
$$
|[\Phi_{(0,n)}^{l-1}(\eta_0^{l})-\Phi_{(0,n)}^{l-1}(\eta_0^{l-1})](\varphi)| \leq C\|\varphi\|_{v^{\xi}}\Delta_l v(x^*)^{4\xi}.
$$
\end{lem}

\begin{proof}
Define
\begin{eqnarray}
T_1 & := & [\eta_0^{l}-\eta_0^{l-1}]\Big(\frac{Q_{0,n}^{l-1}(\varphi)}{\eta_0^l(Q_{0,n}^{l-1}(1))}\Big) \label{eq:diff_12} \\
T_2 & := &  \Big(\frac{\eta_0^{l-1}(Q_{0,n}^{l-1}(\varphi))}{\eta_0^{l-1}(Q_{0,n}^{l-1}(1))}\Big)[\eta_0^{l-1}-\eta_0^{l}]\Big(\frac{Q_{0,n}^{l-1}(1)}{\eta_0^l(Q_{0,n}^{l-1}(1))}\Big). \label{eq:diff_13}
\end{eqnarray}
Then we have
\begin{equation}\label{eq:diff_14}
T_1 + T_2 = \Phi_{(0,n)}^{l-1}(\eta_0^{l})-\Phi_{(0,n)}^{l-1}(\eta_0^{l-1})](\varphi).
\end{equation}
We will bound $|T_1|$ and $|T_2|$ and then sum the bounds to conclude.

{\bf Term} \eqref{eq:diff_12}: We have by Lemma \ref{lem:diff_3} \eqref{eq:diff_10} and Lemma \ref{lem:diff_1} that
\begin{equation}\label{eq:diff_15}
|T_1| \leq C\|\varphi\|_{v^{\xi}}\Delta_lv(x^*)^{4\xi}.
\end{equation}

{\bf Term} \eqref{eq:diff_13}: We have
$$
\Big(\frac{\eta_0^{l-1}(Q_{0,n}^{l-1}(\varphi))}{\eta_0^{l-1}(Q_{0,n}^{l-1}(1))}\Big)[\eta_0^{l-1}-\eta_0^{l}]\Big(\frac{Q_{0,n}^l(1)}{\eta_0^l(Q_{0,n}^l(1))}\Big) = 
\eta_n^{l-1}(\varphi)[\eta_0^{l-1}-\eta_0^{l}]\Big(\frac{Q_{0,n}^l(1)}{\eta_0^l(Q_{0,n}^l(1))}\Big).
$$
Then using $\varphi\in \mathcal{L}_{v^{\xi}}(\mathsf{X})$ along with \cite[Proposition 1]{whiteley} for $\eta_n^{l-1}(\varphi)$ and Lemma \ref{lem:diff_3} \eqref{eq:diff_8} and Lemma \ref{lem:diff_1} for the other term on the R.H.S.~we have
\begin{equation}\label{eq:diff_16}
|T_2| \leq C\|\varphi\|_{v^{\xi}}\Delta_lv(x^*)^{2\xi}.
\end{equation}
Noting \eqref{eq:diff_14}, \eqref{eq:diff_15} and \eqref{eq:diff_16} the proof can be concluded.
\end{proof}

\begin{lem}\label{lem:diff_5}
Assume (H\ref{hypav:1}-\ref{hypav:6}), (H\ref{hypav:8}). Then for any $(\xi,\hat{\xi})\in(0,1/8)\times(0,1/2)$ there exists a $\rho<1$ and $C<+\infty$ depending on the constants in (H\ref{hypav:1}-\ref{hypav:6}) (H\ref{hypav:8}) such that for any $\varphi\in \mathcal{L}_{v^{\xi}}(\mathsf{X})$, $n\geq 1$, $0\leq p \leq n$, $1\leq l \leq L$ we have
$$
|\Phi_{(p,n)}^{l-1}(\Phi_p^l(\eta_{p-1}^l))-\Phi_{(p,n)}^{l-1}(\Phi_p^{l-1}(\eta_{p-1}^l))](\varphi)| \leq C\|\varphi\|_{v^{\xi}}\Delta_l\rho^{n-p}v(x^*)^{8\xi+2\hat{\xi}}.
$$
\end{lem}

\begin{proof}
Define
\begin{eqnarray}
T_1 & := & [\Phi_p^l(\eta_{p-1}^l)-\Phi_p^{l-1}(\eta_{p-1}^l)]\Big(
\frac{Q_{p.n}^{l-1}(1)P_{p.n}^{l-1}(\bar{\varphi}_n)}
{\Phi_p^l(\eta_{p-1}^l)(Q_{p.n}^{l-1}(1))}
\Big) \label{eq:diff_18} \\
T_2 & := &   
\Big(
\frac{\Phi_p^{l-1}(\eta_{p-1}^l)(Q_{p.n}^{l-1}(1)P_{p.n}^{l-1}(\bar{\varphi}_n))}{\Phi_p^{l-1}(\eta_{p-1}^l)(Q_{p.n}^{l-1}(1))}
\Big)\times \nonumber \\ & & 
[\Phi_p^{l-1}(\eta_{p-1}^l)-\Phi_p^l(\eta_{p-1}^l)]\Big(
\frac{Q_{p,n}^{l-1}(1)}{\Phi^{l}_p(\eta_{p-1}^l)(Q_{p,n}^{l-1}(1))}
\Big)
\label{eq:diff_19}
\end{eqnarray}
where $\bar{\varphi}_n=\varphi-\eta_n^{l-1}(\varphi)$.
Then we have
\begin{equation}\label{eq:diff_20}
T_1 + T_2 = [\Phi_{(p,n)}^{l-1}(\Phi_p^l(\eta_{p-1}^l))-\Phi_{(p,n)}^{l-1}(\Phi_p^{l-1}(\eta_{p-1}^l))](\varphi).
\end{equation}
We will bound $|T_1|$ and $|T_2|$ and then sum the bounds to conclude.

{\bf Term} \eqref{eq:diff_18}: We have by (a similar proof to) \eqref{eq:diff_need} that 
\begin{equation}\label{eq:diff_21}
\|P_{p.n}^{l-1}(\bar{\varphi}_n)\|_{v^{2\xi}}\leq C\|\varphi\|_{v^{\xi}}\rho^{n-p}
\end{equation}
and Lemma \ref{lem:diff_3} \eqref{eq:diff_17} $\|Q_{p.n}^{l-1}(1)/\Phi_p^l(\eta_{p-1}^l)(Q_{p.n}^{l-1}(1))\|_{2\xi}\leq C$, so by Lemma \ref{lem:diff_2}
$$
\Big| [\Phi_p^l(\eta_{p-1}^l)-\Phi_p^{l-1}(\eta_{p-1}^l)]\Big(
\frac{Q_{p.n}^{l-1}(1)P_{p.n}^{l-1}(\bar{\varphi}_n)}
{\Phi_p^l(\eta_{p-1}^l)(Q_{p.n}^{l-1}(1))}
\Big)\Big| \leq C \|\varphi\|_{v^{\xi}}\Delta_l\rho^{n-p} \frac{\eta_{p-1}^l(v^{8\xi})}{\eta_{p-1}^l(G_{p-1})}.
$$
Applying \cite[Proposition 1, Proposition 2 (2)]{whiteley} yields
\begin{equation}\label{eq:diff_22}
|T_1| \leq C \|\varphi\|_{v^{\xi}}\Delta_l\rho^{n-p}v(x^*)^{8\xi}.
\end{equation}

{\bf Term} \eqref{eq:diff_19}: We have by Lemma \ref{lem:diff_3} \eqref{eq:diff_9} $\|Q_{p.n}^{l-1}(1)/\Phi_p^{l-1}(\eta_{p-1}^l)(Q_{p.n}^{l-1}(1))\|_{2\xi}\leq C$, so by
\eqref{eq:diff_21} 
$$
|T_2| \leq C \|\varphi\|_{v^{\xi}}\rho^{n-p}
\Big(\frac{\eta_{p-1}^l(Q_p^{l-1}(v^{4\xi}))}{\eta_{p-1}^l(G_{p-1})}\Big)
\Big|[\Phi_p^{l-1}(\eta_{p-1}^l)-\Phi_p^l(\eta_{p-1}^l)]\Big(
\frac{Q_{p,n}^{l-1}(1)}{\Phi^{l}_p(\eta_{p-1}^l)(Q_{p,n}^{l-1}(1))}
\Big)\Big|
$$
Applying (H\ref{hypav:1}), \cite[Proposition 1, Proposition 2 (2)]{whiteley} to $\eta_{p-1}^l(Q_p^{l-1}(v^{4\xi}))/\eta_{p-1}^l(G_{p-1})$ and 
Lemma \ref{lem:diff_2} along with Lemma \ref{lem:diff_3} \eqref{eq:diff_17} and \cite[Proposition 1, Proposition 2 (2)]{whiteley} to the remaining term gives
\begin{equation}\label{eq:diff_23}
|T_2| \leq C \|\varphi\|_{v^{\xi}}\Delta_l\rho^{n-p}v(x^*)^{4\xi+2\hat{\xi}}.
\end{equation}
Noting \eqref{eq:diff_20}, \eqref{eq:diff_22} and \eqref{eq:diff_23} the proof can be concluded.
\end{proof}

\begin{prop}\label{prop:diff_1}
Assume (H\ref{hypav:1}-\ref{hypav:6}), (H\ref{hypav:8}). Then for any $(\xi,\hat{\xi})\in(0,1/8)\times(0,1/2)$ there exists a $C<+\infty$ depending on the constants in (H\ref{hypav:1}-\ref{hypav:6}) (H\ref{hypav:8}) such that for any $\varphi\in \mathcal{L}_{v^{\xi}}(\mathsf{X})$, $n\geq 0$, $1\leq l \leq L$ we have
$$
|[\eta_n^l-\eta_n^{l-1}](\varphi)| \leq C\|\varphi\|_{v^{\xi}}\Delta_lv(x^*)^{8\xi+2\hat{\xi}}.
$$
\end{prop}

\begin{proof}
The case $n=0$ follows by Lemma \ref{lem:diff_1}, so we suppose $n\geq 1$.
Define
\begin{eqnarray}
T_1 & := & [\Phi_{(0,n)}^l(\eta_0^l)-\Phi_{(0,n)}^{l-1}(\eta_0^l)](\varphi) \label{eq:diff_24} \\
T_2 & := & [\Phi_{(0,n)}^{l-1}(\eta_0^{l})-\Phi_{(0,n)}^{l-1}(\eta_0^{l-1})](\varphi). \label{eq:diff_25}
\end{eqnarray}
Then we have
\begin{equation}\label{eq:diff_26}
T_1 + T_2 = [\eta_n^l-\eta_n^{l-1}](\varphi).
\end{equation}
We will bound $|T_1|$ and $|T_2|$ and then sum the bounds to conclude.

{\bf Term} \eqref{eq:diff_24}: We have the standard telescoping sum identity
\begin{equation}\label{eq:diff_29}
[\Phi_{(0,n)}^l(\eta_0^l)-\Phi_{(0,n)}^{l-1}(\eta_0^l)](\varphi) = \sum_{p=1}^n\Big\{[\Phi_{(p,n)}^{l-1}(\Phi_p^l(\eta_{p-1}^l))-\Phi_{(p,n)}^{l-1}(\Phi_p^{l-1}(\eta_{p-1}^l))](\varphi)\Big\}.
\end{equation}
Applying Lemma \ref{lem:diff_5} gives
\begin{equation}\label{eq:diff_27}
|T_1| \leq C\|\varphi\|_{v^{\xi}}\Delta_lv(x^*)^{8\xi+2\hat{\xi}}.
\end{equation}

{\bf Term} \eqref{eq:diff_25}: We have  by Lemma \ref{lem:diff_4}
\begin{equation}\label{eq:diff_28}
|T_2| \leq C\|\varphi\|_{v^{\xi}}\Delta_lv(x^*)^{4\xi}.
\end{equation}
Noting \eqref{eq:diff_26}, \eqref{eq:diff_27} and \eqref{eq:diff_28} the proof can be concluded.
\end{proof}

\begin{rem}\label{rem:diff_1}
If $\varphi\in\mathcal{B}_b(\mathsf{X})$ in Lemma \ref{lem:diff_5} and Proposition \ref{prop:diff_1}, then one can take $(\xi,\hat{\xi})\in(0,1/2)\times(0,1]$ and the upper-bound
in Lemma \ref{lem:diff_5} is $C\|\varphi\|\Delta_l\rho^{n-p}v(x^*)^{2\xi+\hat{\xi}}$ and in Proposition \ref{prop:diff_1} it is $C\|\varphi\|\Delta_lv(x^*)^{2\xi+\hat{\xi}}$.
\end{rem}

\begin{prop}\label{prop:diff_2}
Assume (H\ref{hypav:1}-\ref{hypav:6}), (H\ref{hypav:8}). Then for any $\xi\in(0,1/4)$  there exists a $C<+\infty$ depending on the constants in (H\ref{hypav:1}-\ref{hypav:6}) (H\ref{hypav:8}) such that for any $\varphi\in \mathcal{B}_{b}(\mathsf{X})$, $n\geq 1$, $0\leq p \leq n$, $1\leq l \leq L$ we have
$$
\|h_{p,n}^lS_{p,n}^{l,l-1}(\varphi)\|_{v^{\xi}} \leq C\|\varphi\|\Delta_l.
$$
\end{prop}

\begin{proof}
For any $x\in\mathsf{X}$, $0<\kappa<\xi$ one has
$$
\frac{h_{p,n}^l(x)S_{p,n}^{l,l-1}(\varphi)(x)}{v(x)^{\xi}} = \frac{h_{p,n}^l(x)}{v(x)^{\kappa}}\frac{1}{v(x)^{\xi-\kappa}}\Big(P_{p,n}^l(\varphi)(x)-P_{p,n}^{l-1}(\varphi)(x)\Big)
$$
and
$$
P_{p,n}^l(\varphi)(x)-P_{p,n}^{l-1}(\varphi)(x) = \sum_{q=p+1}^n
\Big\{[\Phi_{(q,n)}^{l-1}(\Phi_{(p,q)}^l(\delta_x))-\Phi_{(q,n)}^{l-1}(\Phi_q^{l-1}\{\Phi_{(p,q-1)}^l(\delta_x)\})](\varphi)\Big\}.
$$
Noting \eqref{eq:diff_29}, one can repeat the proofs of Lemmata \ref{lem:diff_3} and \ref{lem:diff_5} in almost the same way for this case, to deduce that
for any $(\tilde{\kappa},\hat{\kappa})\in(0,1/2)\times(0,1]$
$$
\Big|\frac{1}{v(x)^{\xi-\kappa}}\Big(P_{p,n}^l(\varphi)(x)-P_{p,n}^{l-1}(\varphi)(x)\Big)\Big| 
\leq C\|\varphi\|\Delta_l\frac{v(x)^{2\tilde{\kappa}+\hat{\kappa}}}{v(x)^{\xi-\kappa}}
$$
where we have noted Remark \ref{rem:diff_1}. Choose $0<\tilde{\kappa}=\hat{\kappa}<1/(24)$,
$\kappa=\xi/2$, then one can set $\xi=6\tilde{\kappa}$ and $0<\xi<1/4$. Hence
\begin{eqnarray*}
\Big|\frac{1}{v(x)^{\xi-\kappa}}\Big(P_{p,n}^l(\varphi)(x)-P_{p,n}^{l-1}(\varphi)(x)\Big)\Big| & \leq & C\|\varphi\|\Delta_l \\
\frac{h_{p,n}^l(x)}{v(x)^{\kappa}} & \leq & C
\end{eqnarray*}
and the proof is concluded.
\end{proof}

\begin{lem}\label{lem:diff_6}
Assume (H\ref{hypav:1}-\ref{hypav:5}). Then for any $\xi\in(0,1]$ there exists a $C<+\infty$ depending on the constants in (H\ref{hypav:1}-\ref{hypav:5}) such that
for any $1\leq l \leq L$ we have
$$
\sup_{n\geq 1}\sup_{0\leq p \leq n}\Big\|\frac{Q_{p,n}^{l-1}(1)}{\eta_p^{l}(Q_{p,n}^{l}(1))}\Big\|_{v^{\xi}}  \leq C.
$$
\end{lem}

\begin{proof}
The proof is the same as \cite[Lemma 8]{whiteley} as the constants in (H\ref{hypav:1}-\ref{hypav:5}) do not depend upon $l$.
\end{proof}

\begin{prop}\label{prop:diff_3}
Assume (H\ref{hypav:1}-\ref{hypav:6}), (H\ref{hypav:8}). Then for any $(\xi,\hat{\xi})\in(0,1/8)\times(0,1/2)$  there exists a $C<+\infty$ depending on the constants in (H\ref{hypav:1}-\ref{hypav:6}) (H\ref{hypav:8}) such that for any $n\geq 1$, $0\leq p \leq n$, $1\leq l \leq L$, we have
$$
\|h_{p,n}^l-h_{p,n}^{l-1}\|_{v^{\xi}} \leq C(n-p)\Delta_lv(x^*)^{\hat{\xi}}.
$$
\end{prop}

\begin{proof}
Define for any $x\in\mathsf{X}$
\begin{eqnarray}
T_1 & := & \frac{1}{\eta_{p}^l(Q_{p,n}^{l}(1))}(Q_{p,n}^{l}(1)(x)-Q_{p,n}^{l-1}(1)(x)) \label{eq:diff_29} \\
T_2 & := & \frac{Q_{p,n}^{l-1}(1)(x)}{\eta_{p}^l(Q_{p,n}^{l}(1))} [\eta_p^{l-1}-\eta_p^l](h_{p,n}^{l-1}) \label{eq:diff_30} \\
T_3 & := &  h_{p,n}^{l-1}(x) \eta_p^l\Big(\frac{1}{\eta_{p}^l(Q_{p,n}^{l}(1))}(Q_{p,n}^{l}(1)-Q_{p,n}^{l-1}(1))\Big). \label{eq:diff_31} 
\end{eqnarray}
Then we have
\begin{equation}\label{eq:diff_32}
T_1 + T_2 -T_3 = h_{p,n}^l(x)-h_{p,n}^{l-1}(x).
\end{equation}
We will bound $|T_1|$, $|T_2|$ and $|T_3|$ and then sum the bounds to conclude.

{\bf Term} \eqref{eq:diff_29}: We have
$$
T_1 = 
\frac{1}{\eta_{p}^l(Q_{p,n}^{l}(1))}
\sum_{q=p+1}^n\Big\{
Q_{p,q}^l(Q_{q,n}^{l-1}(1))(x) - Q_{p,q-1}^l(Q_{q-1,n}^{l-1}(1))(x)
\Big\}
$$
Considering the summand we have
$$
Q_{p,q}^l(Q_{q,n}^{l-1}(1))(x) - Q_{p,q-1}^l(Q_{q-1,n}^{l-1}(1))(x) = \eta_{q+1}^{l-1}(Q_{q+1,n}^{l-1}(1))\Big(Q_{p,q-1}^l(G_{q-1}[M_q^l-M_q^{l-1}](h_{q+1,n}^{l-1}))\Big).
$$
Now applying Lemma \ref{lem:diff_1}, \cite[Proposition 2 (3)]{whiteley} (with $v^{\xi/4}$) and (H\ref{hypav:5}) 1.~we have
$$
|Q_{p,q}^l(Q_{q,n}^{l-1}(1))(x) - Q_{p,q-1}^l(Q_{q-1,n}^{l-1}(1))(x)| \leq C\Delta_l \eta_{q+1}^{l-1}(Q_{q+1,n}^{l-1}(1))Q_{p,q-1}^l(v^{\xi/2})(x).
$$
Thus
\begin{equation}\label{eq:diff_33}
|T_1| \leq \frac{C\Delta_l}{\eta_{p}^l(Q_{p,n}^{l}(1))}\sum_{q=p+1}^n\Big\{\eta_{q+1}^{l-1}(Q_{q+1,n}^{l-1}(1))Q_{p,q-1}^l(v^{\xi/2})(x)\Big\}.
\end{equation}
Now, 
\begin{equation}\label{eq:diff_34}
\frac{\eta_{q+1}^{l-1}(Q_{q+1,n}^{l-1}(1))Q_{p,q-1}^l(v^{\xi/2})(x)}{\eta_{p}^l(Q_{p,n}^{l}(1))}
= 
\eta_{q+1}^{l-1}\Big(\frac{Q_{q+1,n}^{l-1}(1)}{\eta_q^l(G_q)\eta_{q+1}^l(Q_{q+1,n}^l(1))}\Big) 
\frac{Q_{p,q-1}^l(v^{\xi/2})(x)}{\eta_{p}^l(Q_{p,q-1}^{l}(1))}\frac{1}{\eta^l_{q-1}(G_{q-1})}.
\end{equation}
Applying Lemma \ref{lem:diff_6}, \cite[Proposition 2 (3)]{whiteley} and then \cite[Proposition 1]{whiteley} gives for any $\hat{\xi}\in(0,1]$
\begin{equation}\label{eq:diff_35}
\eta_{q+1}^{l-1}\Big(\frac{Q_{q+1,n}^{l-1}(1)}{\eta_q^l(G_q)\eta_{q+1}^l(Q_{q+1,n}^l(1))}\Big) \leq C v(x^*)^{\hat{\xi}}.
\end{equation}
In addition
$$
\frac{Q_{p,q-1}^l(v^{\xi/2})(x)}{\eta_{p}^l(Q_{p,q-1}^{l}(1))} = 
\frac{Q_{p,q-1}^l(v^{\xi/2})(x)}{Q_{p,q-1}^l(1)(x)}
h_{p,q-1}^l(x)
$$
By a similar proof to \cite[Proposition 1]{whiteley} $Q_{p,q-1}^l(v^{\xi/2})(x)/Q_{p,q-1}^l(1)(x)\leq Cv(x)^{\xi/2}$ and thus by \cite[Proposition 2 (3)]{whiteley}
$h_{p,q-1}^l(x)\leq Cv(x)^{\xi/2}$, so 
\begin{equation}\label{eq:diff_36}
\frac{Q_{p,q-1}^l(v^{\xi/2})(x)}{\eta_{p}^l(Q_{p,q-1}^{l}(1))} \leq Cv(x)^{\xi}.
\end{equation}
Noting \eqref{eq:diff_33}-\eqref{eq:diff_36}, we have shown that for any $(\xi,\hat{\xi})\in(0,1/8)\times(0,1/2)$
\begin{equation}\label{eq:diff_37}
|T_1| \leq C(n-p)\Delta_lv(x)^{\xi}v(x^*)^{\hat{\xi}}.
\end{equation}

{\bf Term} \eqref{eq:diff_30}: We have, by Lemma \ref{lem:diff_6} and Proposition \ref{prop:diff_1}, for any $(\xi,\kappa,\hat{\kappa})\in(0,1/8)^2\times(0,1/2)$
$$
|T_2| \leq Cv(x)^{\xi}\|h_{p,n}^{l-1}\|_{v^{\kappa}}\Delta_lv(x^*)^{8\kappa+2\hat{\kappa}}.
$$
One can choose $0<\kappa=\hat{\kappa}<1/20$, $\hat{\xi}=10\kappa$ and applying \cite[Proposition 2 (3)]{whiteley}, we have
\begin{equation}\label{eq:diff_38}
|T_2| \leq C\Delta_lv(x)^{\xi}v(x^*)^{\hat{\xi}}.
\end{equation}

{\bf Term} \eqref{eq:diff_31}: We have by \cite[Proposition 2 (3)]{whiteley} and \eqref{eq:diff_37}, for any $(\kappa,\hat{\kappa})\in(0,1/8)\times(0,1/2)$
$$
|T_3| \leq C v(x)^{\xi}\times(n-p)\Delta_l\eta_p^l(v^{\kappa})v(x^*)^{\hat{\kappa}}.
$$
Applying \cite[Proposition 1]{whiteley} and choosing $(\kappa,\hat{\kappa})$, so that $\hat{\xi}=\kappa+\hat{\kappa}<1/2$, we have
\begin{equation}\label{eq:diff_39}
|T_3| \leq C(n-p)\Delta_lv(x)^{\xi}v(x^*)^{\hat{\xi}}.
\end{equation}
Noting \eqref{eq:diff_32}, \eqref{eq:diff_37}, \eqref{eq:diff_38} and \eqref{eq:diff_39} the proof can be concluded.
\end{proof}

\subsection{Proofs for Proposition \ref{prop:diff_max_coup}}\label{app:diff_max_coup}

\begin{proof}[Proof of Proposition \ref{prop:diff_max_coup}]
We focus on the first bound in Theorem \ref{theo:av_thm}. Clearly
$$
\check{\eta}^C_n\Big((\varphi\otimes 1-1\otimes\varphi - ([\eta_n^l-\eta_n^{l-1}](\varphi)))^2\Big) \leq 2\|\varphi\|^2\check{\eta}^C_n(\mathbb{I}_A).
$$
As $\check{\eta}^C_n$ is the maximal coupling of $(\eta_n^{l},\eta_n^{l-1})$ we have $\check{\eta}^C_n(\mathbb{I}_A)=\|\eta_n^{l}-\eta_n^{l-1}\|_{\textrm{tv}}$. Hence on applying 
Proposition \ref{prop:diff_1} (noting Remark \ref{rem:diff_1}), we have that
$$
\check{\eta}^C_n\Big((\varphi\otimes 1-1\otimes\varphi - ([\eta_n^l-\eta_n^{l-1}](\varphi)))^2\Big) \leq C\|\varphi\|^2\Delta_lv(x^*)^{2\xi+\hat{\xi}}.
$$
The proof is completed by using Proposition \ref{prop:diff_general_b_bound} and Proposition \ref{prop:diff_4}.
\end{proof}

\begin{prop}\label{prop:diff_4}
Assume (H\ref{hypav:1}-\ref{hypav:6}), (H\ref{hypav:8}). Then for any $(\xi,\hat{\xi})\in(0,1/32)\times(0,1/2)$ there exists a $C<+\infty$ depending on the constants in (H\ref{hypav:1}-\ref{hypav:6}) (H\ref{hypav:8}) such that for any $n\geq 0$, $1\leq l \leq L$ we have
$$
\check{\eta}^C_n(\mathbb{I}_A(v\otimes v)^{2\xi}) \leq C\Delta_lv(x^*)^{32\xi+2\hat{\xi}}.
$$
\end{prop}

\begin{proof}
We will use $(\eta_n^l,\eta_n^{l-1})$ to denote both the probability measure and density associated to $(\eta_n^l,\eta_n^{l-1})$.
We have for any $n\geq 0$
\begin{eqnarray*}
\check{\eta}^C_n(\mathbb{I}_A(v\otimes v)^{2\xi}) & = & \Big(1-\int_{\mathsf{X}}\eta_n^l(x)\wedge\eta_n^{l-1}(x)dx\Big)\int_{\mathsf{X}\times\mathsf{X}}\mathbb{I}_A(x,y) v(x)^{2\xi}v(y)^{2\xi}\times \\
& &  \frac{(\eta_n^l(x)-\eta_n^l(x)\wedge\eta_n^{l-1}(x))}
{1-\int_{\mathsf{X}}\eta_n^l(x)\wedge\eta_n^{l-1}(x)dx} \frac{(\eta_n^{l-1}(y)-\eta_n^l(y)\wedge\eta_n^{l-1}(y))}{1-\int_{\mathsf{X}}\eta_n^l(x)\wedge\eta_n^{l-1}(x)dx} dxdy.
\end{eqnarray*}
Then by upper-bounding the indicator by 1 and applying  Cauchy-Schwarz:
\begin{equation}\label{eq:diff_40}
\check{\eta}^C_n(\mathbb{I}_A(v\otimes v)^{2\xi})  \leq T_1T_2
\end{equation}
where
\begin{eqnarray*}
T_1& := & \Big(\int_{\mathsf{X}}v(x)^{4\xi}(\eta_n^l(x)-\eta_n^l(x)\wedge\eta_n^{l-1}(x))dx\Big)^{1/2} \\
T_2 & := & \Big(\int_{\mathsf{X}}v(x)^{4\xi}(\eta_n^{l-1}(x)-\eta_n^l(x)\wedge\eta_n^{l-1}(x))dx\Big)^{1/2}.
\end{eqnarray*}
We now just deal with $T_1$ as the proof for $T_2$ is almost identical.

We have, letting $D_n^l=\{x:\eta_n^l(x)\geq\eta_n^{l-1}(x)\}$, $\tilde{\varphi}(x)=v(x)^{4\xi}\mathbb{I}_{D_n^l}(x)$ and $\hat{\varphi}(x)=v(x)^{4\xi}\mathbb{I}_{(D_n^l)^c}(x)$
\begin{eqnarray*}
T_1^2 &= &  \frac{1}{2}\Big(\int_{\mathsf{X}}v(x)^{4\xi}[\eta_n^l(x)-\eta_n^{l-1}(x)]dx + \int_{\mathsf{X}}v(x)^{4\xi}|\eta_n^l(x)-\eta_n^{l-1}(x)|dx\Big) \\
& = &  \frac{1}{2}\Big(\int_{\mathsf{X}}v(x)^{4\xi}[\eta_n^l(x)-\eta_n^{l-1}]dx + \int_{\mathsf{X}}\tilde{\varphi}(x)(\eta_n^l(x)-\eta_n^{l-1}(x))dx + \\ & &
\int_{\mathsf{X}}\hat{\varphi}(x)(\eta_n^{l-1}(x)-\eta_n^{l}(x))dx\Big) \\
& \leq & \frac{1}{2}(|[\eta_n^l-\eta_n^{l-1}](v^{4\xi})| + |[\eta_n^l-\eta_n^{l-1}](\tilde{\varphi})|+|[\eta_n^l-\eta_n^{l-1}](\hat{\varphi})|)
\end{eqnarray*}
As $\xi\in(0,1/32)$, $(v^{4\xi},\tilde{\varphi},\tilde{\varphi})\in\mathcal{L}_{v^{\kappa}}(\mathsf{X})^3$, with $0<\kappa<1/8$, so applying Proposition \ref{prop:diff_1}
we have
$$
T_1 \leq C (\Delta_l)^{1/2} v(x^*)^{16\xi+\hat{\xi}}.
$$
Similar calculations give $T_2 \leq C (\Delta_l)^{1/2} v(x^*)^{16\xi+\hat{\xi}}$ and hence noting \eqref{eq:diff_40} one can conclude.
\end{proof}

\subsection{Proofs for Proposition \ref{prop:diff_was}}\label{app:diff_was}

Recall here that $\mathsf{X}=\mathbb{R}$ and the metric in (H\ref{hypav:7}) is $\mathsf{d}_1$ the $L_1-$norm. Let $\tilde{\varphi}(x,y)=(x-y)^2$.

\begin{proof}[Proof of Proposition \ref{prop:diff_was}]
Considering the second bound in Theorem \ref{theo:av_thm}, the result follows by Proposition \ref{prop:diff_general_b_bound} and Lemmata \ref{lem:diff_7} and \ref{lem:diff_8}.
\end{proof}

\begin{prop}\label{prop:diff_5}
Assume (H\ref{hypav:1}-\ref{hypav:6}), (H\ref{hypav:8}). Then for any $(\xi,\hat{\xi})\in(0,1/8)\times(0,1/2)$ there exists a $C<+\infty$ depending on the constants in (H\ref{hypav:1}-\ref{hypav:6}) (H\ref{hypav:8}) such that for any $\varphi\in \mathcal{L}_{v^{\xi}}(\mathsf{X})$, $n\geq 0$, $1\leq l \leq L$ we have
$$
|[\overline{\eta}_n^l-\overline{\eta}_n^{l-1}](\varphi)| \leq C\|\varphi\|_{v^{\xi}}\Delta_lv(x^*)^{8\xi+2\hat{\xi}}.
$$
\end{prop}

\begin{proof}
Define
\begin{eqnarray}
T_1 & := & \frac{[\eta_n^l-\eta_n^{l-1}](G_n\varphi)}{\eta_n^l(G_n)} \label{eq:diff_40} \\
T_2 & := & \frac{\eta_n^{l-1}(G_n\varphi)}{\eta_n^{l-1}(G_n)\eta_n^{l}(G_n)}[\eta_n^l-\eta_n^{l-1}](G_n). \label{eq:diff_41}
\end{eqnarray}
Then we have
\begin{equation}\label{eq:diff_42}
T_1 - T_2 = [\overline{\eta}_n^l-\overline{\eta}_n^{l-1}](\varphi).
\end{equation}
We will bound $|T_1|$ and $|T_2|$ and then sum the bounds to conclude.

{\bf Term} \eqref{eq:diff_40}: By Proposition \ref{prop:diff_1} and $\eta_n^l(G_n)\geq C$ we have
\begin{equation}\label{eq:diff_43}
|T_1| \leq C\|\varphi\|_{v^{\xi}}\Delta_lv(x^*)^{8\xi+2\hat{\xi}}.
\end{equation}

{\bf Term} \eqref{eq:diff_41}: By Proposition \ref{prop:diff_1}, (H\ref{hypav:5}) and $\eta_n^{s}(G_n)\geq C$, $s\in\{l,l-1\}$,  we have
\begin{equation}\label{eq:diff_44}
|T_2| \leq C\|\varphi\|_{v^{\xi}}\Delta_lv(x^*)^{8\xi+2\hat{\xi}}.
\end{equation}
Noting \eqref{eq:diff_42}, \eqref{eq:diff_43} and \eqref{eq:diff_44} the proof can be concluded.
\end{proof}

Denote by $\check{\overline{\eta}}_n^C$ the maximal coupling of $(\overline{\eta}_n^l,\overline{\eta}_n^{l-1})$.

\begin{cor}\label{cor:diff_1}
Assume (H\ref{hypav:1}-\ref{hypav:6}), (H\ref{hypav:8}). Then for any $(\xi,\hat{\xi})\in(0,1/16)\times(0,1/2)$ there exists a $C<+\infty$ depending on the constants in (H\ref{hypav:1}-\ref{hypav:6}) (H\ref{hypav:8}) such that for any $n\geq 0$, $1\leq l \leq L$ we have
$$
\check{\overline{\eta}}^C_n(\mathbb{I}_A(v\otimes v)^{\xi}) \leq C\Delta_lv(x^*)^{16\xi+2\hat{\xi}}.
$$
\end{cor}

\begin{proof}
Follows by the same proof as for Proposition \ref{prop:diff_4} except using Proposition \ref{prop:diff_5} instead of Propositon \ref{prop:diff_1} in the proof.
\end{proof}

Denote by $\check{\overline{\eta}}_n^W$ the optimal Wasserstein coupling of $(\overline{\eta}_n^l,\overline{\eta}_n^{l-1})$ that is, for $B\in\mathcal{X}\vee\mathcal{X}$
$$
\check{\overline{\eta}}_n^W(B) = \int_{0}^1 \mathbb{I}_B(F_{\overline{\eta}_n^l}^{-1}(u),F_{\overline{\eta}_n^{l-1}}^{-1}(u))du.
$$

\begin{lem}\label{lem:diff_7}
Assume (H\ref{hypav:1}-\ref{hypav:6}), (H\ref{hypav:8}). Then
if $\tilde{\varphi}\in\mathcal{L}_{(v\otimes v)^{\tilde{\xi}}}(\mathsf{X})$, for any $\tilde{\xi}\in(0,1/16)$ and set $\hat{\xi}\in (0,1/2)$ then there exists a $C<+\infty$ depending on the constants in (H\ref{hypav:1}-\ref{hypav:6}) (H\ref{hypav:8}) such that for any 
$\varphi\in \mathcal{B}_{b}(\mathsf{X})\cap\textrm{\emph{Lip}}_{\mathsf{d}_1}(\mathsf{X})$
$n\geq 0$, $1\leq l \leq L$ we have
$$
\check{\overline{\eta}}^W_n\Big((\varphi\otimes 1-1\otimes\varphi - ([\eta_n^l-\eta_n^{l-1}](\varphi)))^2\Big) \leq C\|\varphi\|_{\textrm{\emph{Lip}}}^2\|\tilde{\varphi}\|_{v^{\tilde{\xi}}}\Delta_lv(x^*)^{16\tilde{\xi}+2\hat{\xi}}.
$$
\end{lem}

\begin{proof}
We assume $n\geq 1$, the case $n=0$ is noted below.
As $\varphi\in \mathcal{B}_{b}(\mathsf{X})\cap\textrm{Lip}_{\mathsf{d}_1}(\mathsf{X})$, it easily follows that
$$
\check{\overline{\eta}}^W_n\Big((\varphi\otimes 1-1\otimes\varphi - ([\eta_n^l-\eta_n^{l-1}](\varphi)))^2\Big) \leq \|\varphi\|_{\textrm{Lip}}^2\check{\overline{\eta}}_{n-1}^W(\check{M}(\tilde{\varphi})).
$$
Applying \eqref{eq:coup_h_cont} (when $p=2$) and using the optimality of the Wasserstein coupling with $\tilde{\varphi}\in\mathcal{L}_{(v\otimes v)^{\tilde{\xi}}}(\mathsf{X})$, gives
$$
\check{\overline{\eta}}^W_n\Big((\varphi\otimes 1-1\otimes\varphi - ([\eta_n^l-\eta_n^{l-1}](\varphi)))^2\Big) \leq C\|\varphi\|_{\textrm{Lip}}^2\|\tilde{\varphi}\|_{v^{\tilde{\xi}}}(\Delta_l + 
\check{\overline{\eta}}^C_{n-1}(\mathbb{I}_A(v\otimes v)^{\tilde{\xi}})).
$$
Application of Corollary \ref{cor:diff_1} yields the desired result. The case $n=0$ follows as 
$$
\check{\overline{\eta}}^W_n\Big((\varphi\otimes 1-1\otimes\varphi - ([\eta_n^l-\eta_n^{l-1}](\varphi)))^2\Big) \leq \|\varphi\|_{\textrm{Lip}}^2\check{\overline{\eta}}_{n}^W(\tilde{\varphi})
$$
and the optimality of the Wasserstein coupling, $\tilde{\varphi}\in\mathcal{L}_{(v\otimes v)^{\tilde{\xi}}}(\mathsf{X})$ with Corollary \ref{cor:diff_1}.
\end{proof}

\begin{lem}\label{lem:diff_8}
Assume (H\ref{hypav:1}-\ref{hypav:6}), (H\ref{hypav:8}). Then, let $\lambda\in(0,1)$ be given,
if $\tilde{\varphi}\in\mathcal{L}_{(v\otimes v)^{\tilde{\xi}}}(\mathsf{X})$, for any $\tilde{\xi}\in(0,1/(16(1+\lambda)))$ and set $(\xi,\hat{\xi})\in(0,\min\{1/32,\lambda/(16(1+\lambda)),(1-2\tilde{\xi})/12\})\times (0,1/2)$ then there exists a $C<+\infty$ depending on the constants in (H\ref{hypav:1}-\ref{hypav:6}) (H\ref{hypav:8}) such that for any 
$n\geq 0$, $1\leq l \leq L$ we have
$$
\check{\overline{\eta}}^W_n(\tilde{\varphi}(v^{4\xi}\otimes v^{8\xi})) \leq C\|\tilde{\varphi}\|_{v^{\tilde{\xi}}}(\Delta_l)^{1/(1+\lambda)}v(x^*)^{20\xi + (16(\lambda+1)\tilde{\xi}+2\hat{\xi})/(1+\lambda)}.
$$
\end{lem}

\begin{proof}
We assume $n\geq 1$, the case $n=0$ is similar and omitted for brevity. By Cauchy-Schwarz and the structure of $\check{\overline{\eta}}^W_n$:
$$
\check{\overline{\eta}}^W_n(\tilde{\varphi}(v^{4\xi}\otimes v^{8\xi})) \leq 
\check{\overline{\eta}}_{n-1}^W(\check{M}(\tilde{\varphi}^2)^{1/2}\check{M}(v^{8\xi}\otimes v^{16\xi})^{1/2}).
$$
Applying \eqref{eq:coup_h_cont} (when $p=4$) gives the upper-bound
$$
\check{\overline{\eta}}^W_n(\tilde{\varphi}(v^{4\xi}\otimes v^{8\xi})) \leq  C\Big(
\check{\overline{\eta}}_{n-1}^W(\tilde{\varphi}\check{M}(v^{8\xi}\otimes v^{16\xi})^{1/2}) + 
\Delta_l \check{\overline{\eta}}_{n-1}^W(\check{M}(v^{8\xi}\otimes v^{16\xi})^{1/2})
\Big).
$$
To complete the proof, we focus on $\check{\overline{\eta}}_{n-1}^W(\tilde{\varphi}\check{M}(v^{8\xi}\otimes v^{16\xi})^{1/2})$ as the other term
can be controlled using the below arguments.

Now, we have, by H\"older's inequality
$$
\check{\overline{\eta}}_{n-1}^W(\tilde{\varphi}\check{M}(v^{8\xi}\otimes v^{16\xi})^{1/2})
\leq 
\check{\overline{\eta}}_{n-1}^W(\tilde{\varphi}^{1+\lambda})^{1/(1+\lambda)}
\check{\overline{\eta}}_{n-1}^W(
\check{M}(v^{8\xi}\otimes v^{16\xi})^{(1+\lambda)/(2\lambda)})^{\lambda/(1+\lambda)}.
$$
By the optimality of the Wasserstein coupling with $\tilde{\varphi}\in\mathcal{L}_{(v\otimes v)^{\tilde{\xi}}}(\mathsf{X})$, we have
$$
\check{\overline{\eta}}_{n-1}^W(\tilde{\varphi}\check{M}(v^{8\xi}\otimes v^{16\xi})^{1/2})
\leq
$$
$$
\Big(\|\tilde{\varphi}\|_{v^{\tilde{\xi}}}
\check{\overline{\eta}}^C_{n-1}(\mathbb{I}_A(v\otimes v)^{\tilde{\xi}(1+\lambda)})
\Big)^{1/(1+\lambda)}
\check{\overline{\eta}}_{n-1}^W(
\check{M}(v^{8\xi}\otimes v^{16\xi})^{(1+\lambda)/(2\lambda)})^{\lambda/(1+\lambda)}
$$
Then applying Corollary \ref{cor:diff_1} gives
$$
\check{\overline{\eta}}_{n-1}^W(\tilde{\varphi}\check{M}(v^{8\xi}\otimes v^{16\xi})^{1/2})
\leq 
$$
$$
C\|\tilde{\varphi}\|_{v^{\tilde{\xi}}}(\Delta_l)^{1/(1+\lambda)}v(x^*)^{(16(\lambda+1)\tilde{\xi}+2\hat{\xi})/(1+\lambda)}
\check{\overline{\eta}}_{n-1}^W(
\check{M}(v^{8\xi}\otimes v^{16\xi})^{(1+\lambda)/(2\lambda)})^{\lambda/(1+\lambda)}.
$$
Now for the remaining term, applying Cauchy-Schwarz twice gives
$$
\check{\overline{\eta}}_{n-1}^W(
\check{M}(v^{8\xi}\otimes v^{16\xi})^{(1+\lambda)/(2\lambda)})^{\lambda/(1+\lambda)}
\leq 
$$
$$
\overline{\eta}_{n-1}^l(M^l(v^{16\xi})^{(1+\lambda)/(2\lambda)})^{\lambda/(2(1+\lambda))}
\overline{\eta}_{n-1}^{l-1}(M^{l-1}(v^{32\xi})^{(1+\lambda)/(2\lambda)})^{\lambda/(2(1+\lambda))}.
$$
Applying Jensen twice gives
$$
\check{\overline{\eta}}_{n-1}^W(
\check{M}(v^{8\xi}\otimes v^{16\xi})^{(1+\lambda)/(2\lambda)})^{\lambda/(1+\lambda)}
\leq
\eta_{n}^l(v)^{4\xi}\eta_{n}^{l-1}(v)^{16\xi} \leq Cv(x^*)^{20\xi}.
$$
The proof is thus complete.
\end{proof}

\begin{rem}\label{rem:improve_rate}
As $\lambda > 0$ in Lemma \ref{lem:diff_8}, one almost has the (time-uniform) forward error rate for the WCPF. We believe that $\lambda=0$ is the case, however, due to technical
difficulties we have not obtained this. One of the issues of the proof is that the Wasserstein coupling \emph{is not} the coupling which minimizes the expectation of $\tilde{\varphi}(v^{4\xi}\otimes v^{8\xi})$ w.r.t.~any coupling of $(\overline{\eta}_n^l,\overline{\eta}_n^{l-1})$.
To see this, one can construct a functional covariance equality for $\mu(\tilde{\varphi}(v^{4\xi}\otimes v^{8\xi}))$ ($\mu$ is any coupling of $(\overline{\eta}_n^l,\overline{\eta}_n^{l-1})$ such that $\mu(\tilde{\varphi}(v^{4\xi}\otimes v^{8\xi}))<+\infty$) as in \cite[Theorem 3.1]{lo} (that is, in terms of the CDFs of $\mu$, $\overline{\eta}_n^l$ and $\overline{\eta}_n^{l-1}$ as in Hoeffding's Lemma) and then when centering with
$(\overline{\eta}_n^l\otimes \overline{\eta}_n^{l-1})(\tilde{\varphi}(v^{4\xi}\otimes v^{8\xi}))$ and applying Hoeffding-Fr\'echet bounds, one observes that $\check{\overline{\eta}}^W_n$ is not optimal. As a result, one cannot transfer between $\check{\overline{\eta}}^W_n$ and $\check{\overline{\eta}}^C_n$
as is done in the proofs of Lemmata \ref{lem:diff_7}-\ref{lem:diff_8}, nor can one use Kantorovich duality. However, one can obtain $\lambda=0$ if any of the following hold true:
\begin{enumerate}
\item{There exist a $\check{\overline{\eta}}^W_n-$integrable $\hat{V}:\mathsf{X}\times\mathsf{X}\rightarrow (0,\infty)$ such that there exist a $C<+\infty$ such that for every $(x,y)\in\mathsf{X}\times\mathsf{X}$,
$\tilde{\varphi}(x,y)v^{4\xi}(x)v^{8\xi}(y)\leq C \hat{V}(x,y)$ and $\hat{V}$ satisfies the Monge condition (e.g.~\cite[eq.~(3.1.7)]{rachev}).
}
\item{$v$ is bounded.}
\item{There exist a $C<+\infty$ such that for every $n\geq 0$, $\check{\overline{\eta}}^W_n(\tilde{\varphi}(v^{4\xi}\otimes v^{8\xi}))\leq C \check{\overline{\eta}}^C_n(\tilde{\varphi}(v^{4\xi}\otimes v^{8\xi}))$.}
\item{There exist a $C<+\infty$ such that for every $n\geq 0$, $1\leq l \leq L$, $\sup_{u\in[0,1]}|F_{\overline{\eta}_n^l}^{-1}(u)-F_{\overline{\eta}_n^{l-1}}^{-1}(u)|\leq C\Delta_l$.}
\end{enumerate}
In general we do not believe 1.~can hold in practice. 2.~is not useful in the context of the article. We believe 3.~\& 4.~to hold up-to some conditions, but have not obtained the proof.
\end{rem}

\subsection{Proof of Lemma \ref{lem:verify}}\label{app:diff_verify}

\begin{proof}[Proof of Lemma \ref{lem:verify}]
We begin by noting that for any $B\in\mathcal{X}$, $0\leq l \leq L$, $y\in\mathsf{X}$ we have
\begin{eqnarray}
M^l(\mathbb{I}_Bv)(y) & = & \sqrt\Big(\frac{(1+\delta_0)}{(1+\delta_0-\beta_l)}\Big)\exp\Big\{\frac{\alpha_l^2y^2}{2(1+\delta_0-\beta_l)}+1\Big\}\times \nonumber\\ & & \int_B \frac{1}{\sqrt{2\pi\beta_l}}\exp\Big\{-\frac{1+\delta_0-\beta_l}{2\beta_l(1+\delta_0)}\Big(x-\frac{\alpha_ly(1+\delta_0)}{1+\delta_0-\beta_l}\Big)^2\Big\}dx.\label{eq:verify:1}
\end{eqnarray}
This will be used below.

We give the proof for $\xi=1$ as it is similar in other cases. We have for any $0\leq l \leq L$, $n\geq 1$, $\varphi\in\mathcal{L}_{v}(\mathsf{X})\cap\textrm{Lip}_{v,\mathsf{d}_1}$
$$
|Q_n^l(\varphi)(x)-Q_n^l(\varphi)(y)| \leq 
$$
$$
\|\varphi\|_{v} M^l(v)(x)|G_{n-1}(x)-G_{n-1}(y)| + \|G_{n-1}\|\|\varphi\|_{v}\int_{\mathsf{X}} v(u) | M^l(x,u)-M^l(y,u)|du.
$$
As noted in Section \ref{sec:verify}, $[\alpha_l^2(1+\delta_0)]/[1+\delta_0-\beta_l]<1$ for any $0\leq l\leq L$, $\delta_0>0$, so using \eqref{eq:verify:1} one can determine that $M^l(v)(x)\in\mathcal{L}_v(\mathsf{X})$ for any $\delta_0>0$, with $\|M^l(v)\|_{v}$ independent of $l$.
In addition, for every $y\in\mathsf{Y}$, $G_{n-1}\in\textrm{Lip}_{\mathsf{d}_1}(\mathsf{X})$ with Lipschitz constant independent of $n$ so we need only consider the right-hand term in the above displayed equation.

Set $D_l(x,y):=\{u\in\mathsf{X}:M^l(x,u)-M^l(y,u)\}$, we consider only $\int_{D_l(x,y)} v(u) | M^l(x,u)-M^l(y,u)|du$ as the calculations on $D_l(x,y)^c$ are very similar. We suppose $x>y$ as the proof with $x<y$ is analogous, so we have by \eqref{eq:verify:1} and $x>y$
$$
\int_{D_l(x,y)} v(u) | M^l(x,u)-M^l(y,u)|du  =  \sqrt\Big(\frac{(1+\delta_0)}{(1+\delta_0-\beta_l)}\Big)\times
$$
$$
\Big\{\exp\Big\{\frac{\alpha_l^2x^2}{2(1+\delta_0-\beta_l)}+1\Big\}\int_{-\infty}^{c_l(x,y)} \frac{1}{\sqrt{2\pi\beta_l}}\exp\Big\{-\frac{1+\delta_0-\beta_l}{2\beta_l(1+\delta_0)}\Big(u-\frac{\alpha_lx(1+\delta_0)}{1+\delta_0-\beta_l}\Big)^2\Big\}du -
$$
$$
\exp\Big\{\frac{\alpha_l^2y^2}{2(1+\delta_0-\beta_l)}+1\Big\}\int_{-\infty}^{c_l(x,y)} \frac{1}{\sqrt{2\pi\beta_l}}\exp\Big\{-\frac{1+\delta_0-\beta_l}{2\beta_l(1+\delta_0)}\Big(u-\frac{\alpha_ly(1+\delta_0)}{1+\delta_0-\beta_l}\Big)^2\Big\}du\Big\}
$$
where $c_l(x,y)=(\alpha_l^2 x^2-\alpha_l^2y^2)/(x-y)$. Define
\begin{eqnarray*}
\overline{c}_{l,1}(x,y) & := & \sqrt\Big(\frac{1+\delta_0-\beta_l}{\beta_l(1+\delta_0)}\Big)\Big(c_l(x,y) - \frac{\alpha_lx(1+\delta_0)}{1+\delta_0-\beta_l}\Big) \\
\overline{c}_{l,2}(x,y) & := & \sqrt\Big(\frac{1+\delta_0-\beta_l}{\beta_l(1+\delta_0)}\Big)\Big(c_l(x,y) - \frac{\alpha_ly(1+\delta_0)}{1+\delta_0-\beta_l}\Big) \\
\tilde{v}(x) & := & \exp\Big\{\frac{\alpha_l^2y^2}{2(1+\delta_0-\beta_l)}+1\Big\}
\end{eqnarray*}
and $\Theta(x)$ as the CDF of a standard normal distribution,  then we have 
$$
\frac{\int_{D_l(x,y)} v(u) | M^l(x,u)-M^l(y,u)|du}{v(x)v(y)}  = 
$$
$$
\sqrt\Big(\frac{(1+\delta_0)}{(1+\delta_0-\beta_l)}\Big)
\Big(
\Big\{\frac{\tilde{v}(x)}{v(x)v(y)}\{\Theta(\overline{c}_{l,1}(x,y))-\Theta(\overline{c}_{l,2}(x,y))\Big\} + 
$$
$$
\Theta(\overline{c}_{l,2}(x,y))\Big[
\frac{\tilde{v}(x)}{v(x)}\Big\{\frac{1}{v(y)}-\frac{1}{v(x)}\Big\} + 
\frac{1}{v(x)}\Big\{\frac{\tilde{v}(x)}{v(x)}-\frac{\tilde{v}(y)}{v(y)}\Big\}\Big]
\Big).
$$
From here, it is straightforward to establish that all the functions in the differences are in the set $\textrm{Lip}_{\mathsf{d}_1}(\mathsf{X})$ with Lipschitz constants that do not depend on $l$ and moreover that the other terms are uniformly bounded in $x,y,l$ (where relevant) - the proofs
are omitted as they are standard.
\end{proof}

\end{document}